\documentclass[a4paper,11pt,reqno]{amsart}
\usepackage[centering, totalwidth = 390pt, totalheight = 610pt]{geometry}
\usepackage{amsmath,amssymb,amsfonts, url, stmaryrd}
\usepackage{eucal}
\usepackage{verbatim}
\usepackage{amsthm}
\usepackage{xspace}

\usepackage{enumerate}

\numberwithin{equation}{section}

\theoremstyle{plain}

\newtheorem{Theorem}{Theorem}

\newtheorem{Corollary}[Theorem]{Corollary}
\newtheorem{Proposition}[Theorem]{Proposition}
\newtheorem{Lemma}[Theorem]{Lemma}

\theoremstyle{definition}
\newtheorem{Definition}[Theorem]{Definition}

\newtheorem{Example}[Theorem]{Example}
\newtheorem{Remark}[Theorem]{Remark}

\newtheorem*{thank}{Acknowledgments}

\newcommand{\twocat}{\ensuremath{\textnormal{2-CAT}}\xspace}
\newcommand{\vcat}{\ensuremath{\textnormal{V-Cat}}\xspace}
\newcommand{\Cat}{\ensuremath{\textnormal{Cat}}\xspace}
\newcommand{\CAT}{\ensuremath{\textnormal{CAT}}\xspace}

\newcommand{\Set}{\ensuremath{\textnormal{Set}}\xspace}

\newcommand{\f}[1]{\ensuremath{\mathcal{#1}}\xspace}
\newcommand{\g}[1]{\ensuremath{\mathbb{#1}}\xspace}

\newcommand{\fcat}{\ensuremath{\mathcal{F}\textnormal{-CAT}}\xspace}

\newcommand{\wdoct}{\ensuremath{w\textnormal{-Doct}}\xspace}
\newcommand{\ldoct}{\ensuremath{l\textnormal{-Doct}}\xspace}
\newcommand{\fTAlg}{\ensuremath{\textnormal{T-}\mathbb{A}\textnormal{lg}}\xspace}
\newcommand{\fTCAlg}{\ensuremath{\textnormal{T}^{co}\textnormal{-}\mathbb{A}\textnormal{lg}}\xspace}
\newcommand{\fTAlgl}{\ensuremath{\textnormal{T-}\mathbb{A}\textnormal{lg}_{\textnormal{l}}}\xspace}
\newcommand{\TAlg}{\ensuremath{\textnormal{T-Alg}}\xspace}
\newcommand{\TCAlg}{\ensuremath{\textnormal{T}^{co}\textnormal{-Alg}}\xspace}
\newcommand{\MonCat}{\ensuremath{\textnormal{MonCat}}\xspace}
\newcommand{\fMonCat}{\ensuremath{\mathbb{M}\textnormal{onCat}}\xspace}
\newcommand{\ps}{\ensuremath{\textnormal{Ps}}\xspace}
\newcommand{\fps}{\ensuremath{\mathbb{P}\textnormal{s}}\xspace}
\newcommand{\Mon}{\ensuremath{\textnormal{Mon}}\xspace}
\newcommand{\fMon}{\ensuremath{\mathbb{M}\textnormal{on}}\xspace}
\newcommand{\phicol}{\ensuremath{\Phi\textnormal{-}\textnormal{Colim}}\xspace}
\newcommand{\fphicol}{\ensuremath{\Phi\textnormal{-}\mathbb{C}\textnormal{olim}}\xspace}
\newcommand{\TAlgs}{\ensuremath{\textnormal{T-Alg}_{\textnormal{s}}}\xspace}
\newcommand{\TAlgl}{\ensuremath{\textnormal{T-Alg}_{\textnormal{l}}}\xspace}
\newcommand{\TAlgc}{\ensuremath{\textnormal{T-Alg}_{\textnormal{c}}}\xspace}

\input xy
\xyoption{all}
\SelectTips{cm}{11}
\makeatletter
\def\matrixobject@{%
 \edef \next@{={\DirectionfromtheDirection@ }}%
 \expandafter \toks@ \next@ \plainxy@
 \let\xy@@ix@=\xyq@@toksix@
 \xyFN@ \OBJECT@}
\let\xy@entry@@norm=\entry@@norm
\def\entry@@norm@patched{%
 \let\object@=\matrixobject@
 \xy@entry@@norm }
\AtBeginDocument{\let\entry@@norm\entry@@norm@patched}
\makeatother

\newcommand{\twocong}[2][0.5]{\ar@{}[#2] \save ?(#1)*{\cong}\restore}
\newcommand{\twoeq}[2][0.5]{\ar@{}[#2] \save ?(#1)*{=}\restore}
\newcommand{\rtwocell}[3][0.5]{\ar@{}[#2] \ar@{=>}?(#1)+/l 0.2cm/;?(#1)+/r 0.2cm/^{#3}}
\newcommand{\ltwocell}[3][0.5]{\ar@{}[#2] \ar@{=>}?(#1)+/r 0.2cm/;?(#1)+/l 0.2cm/^{#3}}
\newcommand{\ltwocello}[3][0.5]{\ar@{}[#2] \ar@{=>}?(#1)+/r 0.2cm/;?(#1)+/l 0.2cm/_{#3}}
\newcommand{\dtwocell}[3][0.5]{\ar@{}[#2] \ar@{=>}?(#1)+/u  0.2cm/;?(#1)+/d 0.2cm/^{#3}}
\newcommand{\dltwocell}[3][0.5]{\ar@{}[#2] \ar@{=>}?(#1)+/ur  0.2cm/;?(#1)+/dl 0.2cm/^{#3}}
\newcommand{\drtwocell}[3][0.5]{\ar@{}[#2] \ar@{=>}?(#1)+/ul  0.2cm/;?(#1)+/dr 0.2cm/^{#3}}
\newcommand{\dthreecell}[3][0.5]{\ar@{}[#2] \ar@3{->}?(#1)+/u  0.2cm/;?(#1)+/d 0.2cm/^{#3}}
\newcommand{\utwocell}[3][0.5]{\ar@{}[#2] \ar@{=>}?(#1)+/d 0.2cm/;?(#1)+/u 0.2cm/_{#3}}
\newcommand{\dtwocelltarg}[3][0.5]{\ar@{}#2 \ar@{=>}?(#1)+/u  0.2cm/;?(#1)+/d 0.2cm/^{#3}}
\newcommand{\utwocelltarg}[3][0.5]{\ar@{}#2 \ar@{=>}?(#1)+/d  0.2cm/;?(#1)+/u 0.2cm/_{#3}}

\newcommand{\thing}{.}

\newcommand{\atwo}{\textbf{2}\xspace}
\newcommand{\hto}{\ensuremath{\, \mathaccent\shortmid\rightarrow\,}}

\title{Two Dimensional Monadicity}
\begin{document}
\author[John Bourke]{John Bourke}
\address{Department of Mathematics and Statistics, Masaryk University, Kotl\'a\v rsk\'a 2, Brno 60000, Czech Republic}
\email{bourkej@math.muni.cz}
\subjclass[2000]{Primary: 18D05}

\date{\today}

\thanks{Supported by the Grant agency of the Czech Republic under the grant P201/12/G028.}



\begin{abstract}
The behaviour of limits of weak morphisms in 2-dimensional universal algebra is not 2-categorical in that, to fully express the behaviour that occurs, one needs to be able to quantify over strict morphisms amongst the weaker kinds.  \f F-categories were introduced to express this interplay between strict and weak morphisms.  We express doctrinal adjunction as an $\f F$-categorical lifting property and use this to give monadicity theorems, expressed using the language of $\f F$-categories, that cover each weaker kind of morphism.
\end{abstract}
 \leftmargini=2em

\def\xypic{\hbox{\rm\Xy-pic}}
\maketitle

\section{Introduction}
The category of monoids sits over the category of sets via a forgetful functor $U:\Mon \to \Set$.  This functor is \emph{monadic} in the sense that it has a left adjoint $F$ and the canonical comparison $E:\Mon \to \Set^{T}$ to the category of algebras for the induced monad $T=UF$ is an equivalence of categories.  So if you are interested in monoids you can set about proving some theorem about algebras for an abstract monad $T$ and be sure it holds for monoids, or any variety of universal algebra for that matter: this is the categorical approach to universal algebra via monads.  
\\
Before going down this path one thing must be established -- namely, the monadicity of $U$.  To this end the fundamental theorem is \emph{Beck's monadicity theorem} \cite{Beck1967Triples} which asserts that a functor $U:\f A \to \f B$ is monadic just when it admits a left adjoint, is conservative and creates $U$-absolute coequalisers.  What makes the theorem so useful in practice is that the conditions, up to the \emph{existence} of a left adjoint, are cast entirely in terms of the typically simple $U$ -- these conditions are clearly met for monoids or indeed any variety (see Section 6.8 of \cite{Mac-Lane1971Categories}). 
\\
Now our interest is not in universal algebra, but in two dimensional universal algebra and 2-monads, and monadicity as appropriate to this setting.  What do we mean by the varieties of 2-dimensional universal algebra?  Monoidal structure borne by categories provides a basic example: one observes that associated to this notion are several kinds of structure.  On the objects front we have at least strict monoidal categories and monoidal categories of interest.  Between these are strict, strong, lax and colax monoidal functors all commonly arising, between which we have just one kind of monoidal transformation.  Restricting ourselves to just one kind of object, let us take the monoidal categories, we still find that we are presented with four related 2-categories $\MonCat_{w}$, where $w \in \{s,p,l,c\}$, living over the 2-category of categories \Cat as on the left below.
$$
\xy
(-30,0)*+{\txt{(1)}}="a";
(-15,0)*+{\MonCat_{s}}="00";
(10,0)*+{\MonCat_{p}}="10"; (30,10)*+{\MonCat_{l}}="11";(30,-10)*+{\MonCat_{c}}="1-1";(10,-25)*+{\Cat}="1-2";
{\ar^{} "00"; "10"}; 
{\ar^{} "10"; "11"}; 
{\ar^{} "10"; "1-1"}; 
{\ar_{V_{s}} "00"; "1-2"}; 
{\ar_{V_{p}} "10"; "1-2"}; 
{\ar@{.>}^<<<<<<<{V_{l}} "11"; "1-2"}; 
{\ar^{V_{c}} "1-1"; "1-2"}; 
\endxy
\xy
(-15,0)*+{\TAlgs}="00";
(10,0)*+{\TAlg_{p}}="10"; (30,10)*+{\TAlgl}="11";(30,-10)*+{\TAlgc}="1-1";(10,-25)*+{\Cat}="1-2";
{\ar^{} "00"; "10"}; 
{\ar^{} "10"; "11"}; 
{\ar^{} "10"; "1-1"}; 
{\ar_{U_{s}} "00"; "1-2"}; 
{\ar_{U_{p}} "10"; "1-2"}; 
{\ar@{.>}^<<<<<<<{U_{l}} "11"; "1-2"}; 
{\ar^{U_{c}} "1-1"; "1-2"}; 
\endxy
$$
The objects in each of these are monoidal categories; the morphisms are respectively strict, strong (or pseudo), lax and colax monoidal functors with monoidal transformations between them in each case.  The inclusions witness that strict morphisms can be viewed as pseudomorphisms $(s \leq p)$, which can in turn be viewed as lax or colax $(p \leq l)$ and ($p \leq c)$.
\\
The situation corresponds with that of a 2-monad $T$ based on \Cat, associated with which are several kinds of algebra, including strict and pseudo-algebras.  We will only ever consider the \emph{strict algebras}: note that even \emph{non-strict} monoidal categories are the \emph{strict} algebras for a 2-monad, as further discussed below.  Between strict algebras are strict, pseudo, lax and colax morphisms, with again a single notion of algebra transformation.  As before we obtain a diagram of 2-categories\begin{footnote}{The objects in each of these 2-categories are the strict algebras.}\end{footnote} and 2-functors over \Cat, as on the right above.
\\
Comparing the diagrams left and right above we see that a monadicity theorem in this setting ought to match each 2-category on the left with the corresponding one on the right in a compatible way.  More precisely, there should exist a 2-monad $T$ on \Cat and, for each $w \in \{s,p,l,c\}$, an equivalence of 2-categories $E_{w}:\MonCat_{w} \to \TAlg_{w}$ over \Cat, as left below
$$
\xy
(-30,0)*+{\txt{(2)}}="a";
(0,0)*+{\MonCat_{w}}="00";
(30,0)*+{\TAlg_{w}}="10";(15,-20)*+{\Cat}="1-2";
{\ar^{E_{w}} "00"; "10"}; 
{\ar_{V_{w}} "00"; "1-2"}; 
{\ar^{U_{w}} "10"; "1-2"}; 
\endxy
\hspace{2cm}
\xy
(0,0)*+{\MonCat_{w_{1}}}="00"; (30,0)*+{\TAlg_{w_{1}}}="10"; 
(0,-20)*+{\MonCat_{w_{2}}}="01"; (30,-20)*+{\TAlg_{w_{2}}}="11";
{\ar^{E_{w_{1}}} "00";"10"};
{\ar ^{}"10";"11"};
{\ar _{}"00";"01"};
{\ar _{E_{w_{2}}} "01";"11"};
\endxy$$
and these equivalences should be natural in the inclusions $w_{1} \leq w_{2}$ for $w_{1},w_{2} \in \{s,p,l,c\}$, as on the right.
\\
Now it is well known that such a 2-monad does indeed exist, with, moreover, each comparison $E_{w}:\MonCat_{w} \to \TAlg_{w}$ an isomorphism of 2-categories.  Likewise many of the other varieties of 2-dimensional universal algebra are monadic in this sense.  With such varieties as the primary object of study the subject of 2-dimensional monad theory was developed, notably in \cite{Blackwell1989Two-dimensional}.  General results were obtained such as those enabling one to deduce that each inclusion $\MonCat_{s} \to \MonCat_{w}$ has a left 2-adjoint, or establish the bicategorical completeness and cocompleteness of $\MonCat_{p}$.
\\
Of course to apply such abstract results one must first establish monadicity.  How is this known?  Here the subject diverges substantially from the 1-dimensional approach of Beck's theorem.   The standard approach is that of colimit presentations \cite{Lack2010A-2-categories}\cite{Kelly1993Adjunctions}.  Here one explicitly constructs the intended 2-monad $T$ as an iterated colimit of free ones, and then performs lengthy calculations with the universal property of the colimit $T$ to establish monadicity in the sense described for monoidal categories above.
\\
Now although the natural analogue of Beck's theorem has been obtained for \emph{pseudomonads} \cite{LeCreurer2002Beck} and pseudoalgebra pseudomorphisms this does not specialise to capture monadicity in the precise sense described above, even when $w=p$.
\\
Our objective in the present paper is to establish  2-dimensional monadicity theorems, in which monadicity is recognised not by using an explicit description of a 2-monad, but by analysing the manner in which the varieties of 2-dimensional universal algebra sit over the base 2-category -- as in Diagram 1.  Apart from characterising such monadic situations our main results, Theorem~\ref{thm:EMExtension} through Theorem~\ref{thm:monadicity}, enable one to establish monadicity when a workable description of a 2-monad is not easily forthcoming.
\\
In seeking to understand monadicity in the above sense the first important observation is that the world of \emph{strict morphisms} is easily understood: one can check that $V_{s}:\MonCat_{s} \to \Cat$ has a left 2-adjoint -- say, by an adjoint functor theorem -- and then apply the enriched version of Beck's monadicity theorem \cite{Dubuc1970Kan-extensions} to establish strict monadicity.  Having observed this to be the case the natural question to ask is:
\emph{Which properties of the commutative triangle
$$\xy
(0,0)*+{\MonCat_{s}}="00";
(30,0)*+{\MonCat_{w}}="10";(15,-20)*+{\Cat}="1-2";
{\ar^{} "00"; "10"}; 
{\ar_{V_{s}} "00"; "1-2"}; 
{\ar^{V_{w}} "10"; "1-2"}; 
\endxy$$
ensure that the canonical isomorphism $E:\MonCat_{s} \to \TAlg_{s}$ extends to an isomorphism $E_{w}:\MonCat_{w} \to \TAlg_{w}$ over the base?}  We answer this question in Theorem~\ref{thm:monadicity} but not using the language of 2-categories.  For it turns out that these determining properties are not 2-categorical in nature: they cannot be expressed as properties of the 2-categories or 2-functors in the above diagram considered individually.  Rather, to express these properties we must be able to \emph{single out strict morphisms amongst each weaker kind}.  Consequently we treat the inclusion $\MonCat_{s} \to \MonCat_{w}$ as a single entity, an \emph{\f F-category} $\fMonCat_{w}$, and the above triangle as a single \emph{\f F-functor} $V:\fMonCat_{w} \to \Cat$.  
\\
Let us now give an overview of the paper and of our line of argument.  \f F-categories were introduced in \cite{Lack2011Enhanced} in order to explain certain relationships between strict and weak morphisms in 2-dimensional universal algebra.  We recall some basic facts about \f F-categories in Section 2, in particular discussing \f F-categories  $\fTAlg_{w}$ of algebras for a 2-monad and \f F-categories $\fMonCat_{w}$ of monoidal categories -- we will use monoidal categories as our running example throughout the paper.
\\
Given a strict monoidal functor $F$ and an adjunction $(\epsilon, F \dashv G, \eta)$ the right adjoint $G$ obtains a unique lax monoidal structure such that the adjunction becomes a monoidal adjunction.  This is an instance of \emph{doctrinal adjunction}, the main topic of Section 3, and the first key \f F-categorical property that we meet.  We express three variants of doctrinal adjunction -- $w$-doctrinal adjunction for $w \in \{l,p,c\}$ -- as lifting properties of an \f F-functor, so that the case just described asserts that the forgetful \f F-functor $V:\fMonCat_{l} \to \Cat$ satisfies $l$-doctrinal adjunction.  We define the closely related class $\wdoct$ of \emph{$w$-doctrinal} \f F-functors and analyse the relationships between the different notions for $w \in \{l,p,c\}$; each such class of \f F-functor is shown to be an orthogonality class in the category of \f F-categories.
\\
In the fourth section we turn to the reason \f F-categories were introduced in \cite{Lack2011Enhanced} -- namely, because of the interplay between strict and weak morphisms, \emph{tight} and \emph{loose}, that occurs when considering limits of weak morphisms in 2-dimensional universal algebra.  The crucial limits are $\bar{w}$-limits of loose morphisms for $w \in \{l,p,c\}$ -- after defining these we describe the illuminating case of the colax limit of a lax monoidal functor.  We then examine how such limits allow one to represent loose morphisms by \emph{tight spans} -- the nature of this representation is analysed in detail.
\\
This analysis allows us to prove the key result of the paper, Theorem~\ref{thm:orthogonality} of Section 5, an orthogonality result which has nothing to do with 2-monads at all.  Its immediate consequence, Corollary~\ref{thm:decomposition}, ensures that the decomposition in \fcat
$$
\xy
(0,0)*+{\MonCat_{s}}="00";(25,0)*+{\fMonCat_{w}}="10";(50,0)*+{\Cat}="20";
{\ar^{j} "00"; "10"}; 
{\ar^<<<<<{V} "10"; "20"}; 
\endxy
\xy
(0,0)*+{=};
\endxy
\xy
(0,0)*+{\MonCat_{s}}="00";(25,8)*+{\MonCat_{s}}="11";(25,-7)*+{\MonCat_{w}}="1-1";(50,0)*+{\Cat}="20";
{\ar^{1} "00"; "11"}; 
{\ar_{j} "00"; "1-1"}; 
{\ar^{j} "11"; "1-1"}; 
{\ar^{V_{s}} "11"; "20"}; 
{\ar_{V_{w}} "1-1"; "20"}; 
\endxy
$$
of the forgetful 2-functor $V_{s}:\MonCat_{s} \to \Cat$ is an orthogonal $(^{\bot}\wdoct,\wdoct)$-decomposition.  Likewise for a 2-monad $T$ on \Cat the \f F-category $\fTAlg_{w}$ is obtained as a $(^{\bot}\wdoct,\wdoct)$-factorisation of $U_{s}:\TAlg_{s} \to \Cat$.
\\
Our monadicity results, given in Section 6, use the uniqueness of $(^{\bot}\wdoct,\wdoct)$-decompositions to extend our understanding of monadicity in the strict setting to cover each weaker kind of morphism.  For instance, the isomorphism $E:\MonCat_{s} \to \TAlg_{s}$ over \Cat induces a commuting diagram as on the outside of
$$
\xy
(0,0)*+{\MonCat_{s}}="00";(30,0)*+{\fMonCat_{w}}="10";(60,0)*+{\Cat}="20";
{\ar^{j} "00"; "10"}; 
{\ar^{V} "10"; "20"}; 
(0,-15)*+{\TAlg_{s}}="01";(30,-15)*+{\fTAlg_{w}}="11";(60,-15)*+{\Cat}="21";
{\ar_{E} "00"; "01"}; 
{\ar@{.>}|{E_{w}} "10"; "11"}; 
{\ar_{j} "01"; "11"}; 
{\ar_{U} "11"; "21"}; 
{\ar^{1} "20"; "21"}; 
\endxy
$$
with each horizontal path an orthogonal $(^{\bot}\wdoct,\wdoct)$-decomposition.  The two outer vertical isomorphisms then induce a unique invertible filler $E_{w}:\fMonCat_{w} \to \fTAlg_{w}$, so establishing monadicity in each weaker context.  This is the idea behind the main monadicity result, Theorem~\ref{thm:monadicity}.  Naturality in different weaknesses (as in Diagram 2 above) is treated in Theorem~\ref{thm:naturality}.\\
In the seventh and final section we describe examples and applications of our results.  We begin by completing the example of monoidal categories before moving on to more complex cases.  In Theorem~\ref{thm:monoids} we give an example of the kind of monadicity result that cannot be established using techniques, like presentations, that require explicit knowledge of a 2-monad.
\begin{thank}
The author thanks Richard Garner, Stephen Lack, Ignacio L\'opez Franco, Michael Shulman and Luk\'a\v s Vok\v r\'inek for useful feedback during the preparation of this work, and Michael Shulman for carefully reading a preliminary draft.  Thanks are due to the referee whose insightful observations enabled us to remove unnecessary hypotheses from the main results.  Thanks also to the organisers of the \emph{PSSL93} and the \emph{Workshop on category theory, in honour of George Janelidze} for providing the opportunity to present it, and the members of the \emph{Brno Category Theory Seminar} for listening to a number of talks about it.
\end{thank}
\section{\f F-categories in 2-dimensional universal algebra}
In this section we recall the notion of \f F-category, introduced in \cite{Lack2011Enhanced}, and a few basic facts about them.  
\subsection{\f F-categories}
An \emph{\f F-category} \g A is a 2-category with two kinds of 1-cell: those of the 2-category itself which are called \emph{loose} and a subcategory of \emph{tight} morphisms containing all of the identities.  A second perspective is that an \f F-category \g A is specified by a pair of 2-categories $\f A_{\tau}$ and $\f A_{\lambda}$ connected by a 2-functor $$j:\f A_{\tau} \to \f A_{\lambda}$$ which is the \emph{identity on objects, faithful and locally fully faithful}.  Here $\f A_{\lambda}$ contains all of the morphisms, which is to say the loose ones, and all 2-cells between them, whilst $\f A_{\tau}$ contains the tight morphisms together with all 2-cells between them in $\f A_{\lambda}$.  Loose morphisms in \g A are drawn with wavy arrows $A \rightsquigarrow B$ and tight morphisms with straight arrows $A \to B$, so that a typical diagram in \g A would be
$$\xy
(0,0)*+{A}="00"; (20,0)*+{B}="10"; 
(20,-20)*+{C}="11";
{\ar@{~>}^{f} "00";"10"};
{\ar ^{g}"10";"11"};
{\ar @{~>}_{h}"00";"11"};
{\ar@{=>}^{\alpha}(17,-4)*+{};(12,-9)*+{}};
\endxy$$
For each pair of objects $A,B \in \g A$ the inclusion of hom categories $$j_{A,B}:\f A_{\tau}(A,B) \to \f A_{\lambda}(A,B)$$ constitutes an \emph{injective on objects fully faithful functor}.  In fact  \f F-categories are precisely categories enriched in \f F, the full subcategory of the arrow category $\Cat^{\atwo}$ whose objects are those functors which are both injective on objects and fully faithful.  $\f F$ is a complete and cocomplete cartesian closed category so that the full theory of enriched categories \cite{Kelly1982Basic} can be applied to the study of \f F-categories.
\\
To begin with we have \fcat, the 2-category of \f F-categories, \f F-functors and \f F-natural transformations.  An \f F-functor $F:\g A \to \g B$ consists of a pair of 2-functors $F_{\tau}:\f A_{\tau} \to \f B_{\tau}$ and $F_{\lambda}:\f A_{\lambda} \to \f B_{\lambda}$ rendering commutative the square
$$\xy
(0,0)*+{\f A_{\tau}}="00"; (15,0)*+{\f A_{\lambda}}="10"; 
(0,-15)*+{\f B_{\tau}}="01"; (15,-15)*+{\f B_{\lambda}}="11";
{\ar^{j_{A}} "00";"10"};
{\ar ^{F_{\lambda}}"10";"11"};
{\ar _{F_{\tau}}"00";"01"};
{\ar _{j_{B}} "01";"11"};
\endxy$$
This equally amounts to a 2-functor $F_{\lambda}:\f A_{\lambda} \to \f B_{\lambda}$ which \emph{preserves tightness}.  An \f F-natural transformation $\eta:F \to G$ is a 2-natural transformation $\eta:F_{\lambda} \to G_{\lambda}$ with \emph{tight components}.
\subsection{2-categories as \f F-categories}
Each 2-category \f A may be viewed as an \f F-category in two extremal ways: as an \f F-category in which only the identities are tight, or as an \f F-category in which \emph{all morphisms are tight}, whereupon the induced \f F-category has the form $$1:\f A \to \f A$$
When we view a 2-category \f A as an \f F-category it will \emph{always be in this second sense} and we again denote it by \f A.  \\
With this convention established we can treat 2-categories as special kinds of \f F-categories and unambiguously speak of \f F-functors $F:\f A \to \g B$ from 2-categories to \f F-categories,  or from \f F-categories to 2-categories as in $G:\g B \to \f C$.  These \f F-functors appear as triangles
$$\xy
(10,0)*+{\f A}="00"; 
(0,-15)*+{\f B_{\tau}}="01"; (20,-15)*+{\f B_{\lambda}}="11";
{\ar ^{F_{\lambda}}"00";"11"};
{\ar _{F_{\tau}}"00";"01"};
{\ar _{j_{B}} "01";"11"};
\endxy
\hspace{2cm}
\xy
(0,0)*+{\f B_{\tau}}="00"; (20,0)*+{\f B_{\lambda}}="10"; 
(10,-15)*+{\f C}="01";
{\ar^{j_{A}} "00";"10"};
{\ar ^{G_{\lambda}}"10";"01"};
{\ar _{G_{\tau}}"00";"01"};
\endxy
$$
with an \f F-functor between 2-categories just a 2-functor.  Observe that each $\f F$-category $\g A$ induces an \f F-functor from its 2-category of tight morphisms $$j:\f A_{\tau} \to \g A$$ which is the identity on tight morphisms and $j:\f A_{\tau} \to \f A_{\lambda}$ on loose ones -- we abuse notation by using $j$ in either situation.
\subsection{\f F-categories of monoidal categories}
In two-dimensional universal algebra one encounters morphisms of four different flavours and so \f F-categories naturally arise.  Here we recall the various \f F-categories associated to the notion of monoidal structure -- we recall the defining equations for monoidal functors as we will use these later on.\\
The data for a monoidal category consists of a tuple $\overline{A}=(A,\otimes,i^{A},\lambda^{A},\rho^{A}_{l},\rho^{A}_{r})$ where we use juxtaposition for the tensor product.  A lax monoidal functor $(F,f,f_{0}):\overline{A} \rightsquigarrow \overline{B}$ consists of a functor $F:A \to B$, coherence constraints $f_{a,b}:Fa\otimes Fb \to F(a\otimes b)$ natural in $a$ and $b$ and a comparison $f_{0}:i^{B} \to Fi^{A}$, all satisfying the three conditions below
$$\xy
(0,0)*+{(Fa\otimes Fb)\otimes Fc}="10"; (40,0)*+{F(a\otimes b)\otimes Fc}="20"; (80,0)*+{F((a\otimes b)\otimes c)}="30"; 
(0,-12)*+{Fa\otimes (Fb\otimes Fc)}="11"; (40,-12)*+{Fa\otimes F(b\otimes c)}="21"; (80,-12)*+{F(a\otimes (b\otimes c))}="31"; 
{\ar ^{f_{a,b}\otimes 1}"10";"20"};
{\ar^{f_{a\otimes b,c}}"20";"30"};
{\ar _{1\otimes f_{b,c}}"11";"21"};
{\ar_{f_{a,b\otimes c}}"21";"31"};
{\ar_{\lambda_{Fa,Fb,Fc}^{B}}"10";"11"};
{\ar^{F\lambda_{a,b,c}^{A}}"30";"31"};
\endxy$$
$$\xy
(00,0)*+{i^{B}\otimes Fa}="10"; (25,0)*+{Fi^{A}\otimes Fa}="20"; (50,0)*+{F(i^{A}\otimes a)}="30"; 
(0,-12)*+{Fa}="11"; (50,-12)*+{Fa}="31"; 
{\ar^<<<<{f_{0}\otimes 1}"10";"20"};
{\ar^{f_{i,a}}"20";"30"};
{\ar@{=}_{}"11";"31"};
{\ar_{\rho_{l}^{B}}"10";"11"};
{\ar^{F\rho_{l}^{A}}"30";"31"};
\endxy
\hspace{0.3cm}
\xy
(00,0)*+{(Fa)\otimes i^{B}}="10"; (25,0)*+{Fa\otimes Fi^{A}}="20"; (50,0)*+{F(a\otimes i^{A})}="30"; 
(0,-12)*+{Fa}="11"; (50,-12)*+{Fa}="31"; 
{\ar ^{1\otimes f_{0}}"10";"20"};
{\ar^{f_{a,i}}"20";"30"};
{\ar@{=}_{}"11";"31"};
{\ar_{\rho_{r}^{B}}"10";"11"};
{\ar^{F\rho_{r}^{A}}"30";"31"};
\endxy$$
which we call the \emph{associativity, left unit} and \emph{right unit} conditions.  We call $(F,f,f_{0})$ strong or strict monoidal just when the constraints $f_{a,b}$ and $f_{0}$ are invertible or identities respectively; reversing these constraints we obtain the notion of a colax monoidal functor.  
$$
\xy
(00,0)*+{Fa\otimes Fb}="10"; (30,0)*+{Ga\otimes Gb}="20"; 
(0,-12)*+{F(a\otimes b)}="11"; (30,-12)*+{G(a\otimes b)}="21"; 
{\ar ^{\eta_{a}\otimes \eta_{b}}"10";"20"};
{\ar_{\eta_{a\otimes b}}"11";"21"};
{\ar_{f_{a,b}}"10";"11"};
{\ar^{g_{a,b}}"20";"21"};
\endxy
\hspace{1cm}
\xy
(00,0)*+{i^{B}}="10"; (20,0)*+{Fi^{A}}="20"; (20,-12)*+{Gi^{A}}="21"; 
{\ar ^{f_{0}}"10";"20"};
{\ar_{g_{0}}"10";"21"};
{\ar^{\eta_{i^{A}}}"20";"21"};
\endxy
$$
Between lax monoidal functors are monoidal transformations $\eta:(F,f,f_{0}) \to (G,g,g_{0})$: these are natural transformations $\eta:F \to G$ satisfying the two conditions above, which we call the \emph{tensor} and \emph{unit} conditions for a monoidal transformation.
\\
For $w \in \{s,l,p,c\}$ we thus have $w$-monoidal functors (with $p$-monoidal meaning strong monoidal).  Together with monoidal categories and monoidal transformations these form a 2-category $\MonCat_{w}$ which sits over \Cat via a forgetful 2-functor $V_{w}:\MonCat_{w} \to \Cat$.
Now strict monoidal functors are strong $(s \leq p)$ and strong monoidal functors can be viewed as lax ($p \leq l$) or colax $(p \leq c)$.  Whenever $w_{1} \leq w_{2}$ it follows that we have an \f F-category $\fMonCat_{w_{1},w_{2}}$ of monoidal categories with tight and loose morphisms the $w_{1}$ and $w_{2}$-monoidal functors respectively, as specified by the inclusion $j:\MonCat_{w_{1}} \to \MonCat_{w_{2}}$.  Each such \f F-category comes equipped with a forgetful \f F-functor $V:\fMonCat_{w_{1},w_{2}} \to \Cat$ where $V_{\tau}=V_{w_{1}}$ and $V_{\lambda}=V_{w_{2}}$ -- see the commuting triangles in Diagram 1 of the Introduction.  Of particular importance will be those \f F-categories $\fMonCat_{s,w}$ for $s \leq w$, which we denote by $\fMonCat_{w}$.

\subsection{\f F-categories of algebras}
Of prime importance are those \f F-categories associated to a 2-monad $T$ on a 2-category \f C.  Each 2-monad has various associated flavours of algebra and morphism.  We will only be interested in \emph{strict algebras} and will call them algebras.  Between algebras we have strict, pseudo, lax and colax morphisms -- as with monoidal functors we specify these using $s$, $p$, $l$ and $c$.\\
Drawn in turn from left to right below is the data $(f,\overline{f})$ for a strict, pseudo, lax or colax morphism of algebras $(f,\overline{f}):(A,a) \rightsquigarrow (B,b)$
$$\xy
(0,0)*+{TA}="00"; (20,0)*+{TB}="10"; 
(0,-20)*+{A}="01"; (20,-20)*+{B}="11";
{\ar^{Tf} "00";"10"};
{\ar ^{b}"10";"11"};
{\ar _{a}"00";"01"};
{\ar _{f} "01";"11"};
\endxy
\hspace{0.75cm}
\xy
(0,0)*+{TA}="00"; (20,0)*+{TB}="10"; 
(0,-20)*+{A}="01"; (20,-20)*+{B}="11";
{\ar^{Tf} "00";"10"};
{\ar ^{b}"10";"11"};
{\ar _{a}"00";"01"};
{\ar _{f} "01";"11"};
(10,-10)*+{\cong};
(10,-7)*+{^{\overline{f}}};
\endxy
\hspace{0.75cm}
\xy
(0,0)*+{TA}="00"; (20,0)*+{TB}="10"; 
(0,-20)*+{A}="01"; (20,-20)*+{B}="11";
{\ar^{Tf} "00";"10"};
{\ar ^{b}"10";"11"};
{\ar _{a}"00";"01"};
{\ar _{f} "01";"11"};
{\ar@{=>}^{\overline{f}}(10,-7)*+{};(10,-13)*+{}};
\endxy
\hspace{0.75cm}
\xy
(0,0)*+{TA}="00"; (20,0)*+{TB}="10"; 
(0,-20)*+{A}="01"; (20,-20)*+{B}="11";
{\ar^{Tf} "00";"10"};
{\ar ^{b}"10";"11"};
{\ar _{a}"00";"01"};
{\ar _{f} "01";"11"};
{\ar@{=>}_{\overline{f}};(10,-13)*+{};(10,-7)*+{}}
\endxy$$
Thus $\overline{f}$ is an identity 2-cell in the first case, invertible in the second, and points into or out of $f$ in the lax or colax cases; in all cases these 2-cells are required to satisfy two coherence conditions \cite{Kelly1974Review}.\\
There is a single notion of 2-cell between any kind of algebra morphisms; for instance given a pair of lax morphisms $(f,\overline{f}),(g,\overline{g}):(A,a) \rightsquigarrow (B,b)$ an algebra 2-cell $\alpha:(f,\overline{f}) \Rightarrow (g,\overline{g})$ is a 2-cell $\alpha:f \Rightarrow g$ satisfying
$$
\xy
(0,0)*+{TA}="11";
(20,0)*+{TB}="31"; (0,-20)*+{A}="12";(20,-20)*+{B}="32";
{\ar^{b} "31"; "32"}; 
{\ar@/^1pc/^{f} "12"; "32"}; 
{\ar@/_1pc/_{g} "12"; "32"}; 
{\ar_{a} "11"; "12"}; 
{\ar^{Tf} "11"; "31"}; 
{\ar@{=>}_{\alpha}(10,-18)*+{};(10,-23)*+{}};
{\ar@{=>}_{\overline{f}}(10,-4)*+{};(10,-10)*+{}};
\endxy
\hspace{1cm}
\xy
(0,-10)*+{=};
\endxy
\hspace{1cm}
\xy
(0,0)*+{TA}="11";
(20,0)*+{TB}="31"; (0,-20)*+{A}="12";(20,-20)*+{B}="32";
{\ar^{b} "31"; "32"}; 
{\ar_{g} "12"; "32"}; 
{\ar_{a} "11"; "12"}; 
{\ar@/^1pc/^{Tf} "11"; "31"}; 
{\ar@/_1pc/_{Tg} "11"; "31"}; 
{\ar@{=>}^{\overline{g}}(10,-10)*+{};(10,-16)*+{}};
{\ar@{=>}_{T\alpha}(10,3)*+{};(10,-2)*+{}};
\endxy
$$
whilst the equation in the colax case looks like the lax case with the directions reversed.

Algebras, $w$-algebra morphisms and transformations live in 2-categories $\TAlg_{w}$ for $w \in \{s,p,l,c\}$, each of which comes equipped with an evident forgetful 2-functor to the base, which we always denote by $U_{w}:\TAlg_{w} \to \f C$.  Each strict morphism is a pseudo morphism ($s \leq p$), and each pseudomorphism can be viewed either as lax $(p \leq l)$ or colax $(p \leq c)$.   It follows that for each pair $w_{1},w_{2} \in \{s,p,l,c\}$ satisfying $w_{1} \leq w_{2}$ we have an \f F-category $\fTAlg_{w_{1},w_{2}}$ with tight morphisms the $w_{1}$-morphisms, and loose morphisms the $w_{2}$-morphisms.  Each comes equipped with a forgetful \f F-functor $U:\fTAlg_{w_{1},w_{2}} \to \f C$ (so $U_{\tau}=U_{w_{1}}$ and $U_{\lambda}=U_{w_{2}}$) -- see the commuting triangles of Diagram 1 of the Introduction.
\\
Of particular importance are those \f F-categories $\fTAlg_{s,w}$ whose tight morphisms are the strict ones, and following \cite{Lack2011Enhanced} we abbreviate these by $\fTAlg_{w}$.  As well as using $j:\TAlgs \to \TAlg_{w}$ for the defining inclusion and $U:\fTAlg_{w} \to \f C$ for the forgetful \f F-functor, we occasionally use $j_{w}$ or $U_{w}$ if we are in the presence of multiple $j$'s or $U$'s.

\subsection{Duality between lax and colax morphisms}
Colax algebra morphisms are lax algebra morphisms with 2-cells reversed.  This statement can be made precise using the covariant duality 2-functor $(-)^{co}:\twocat \to \twocat$ which takes a 2-category $\f C$ to the 2-category $\f C^{co}$ with the same underlying category but with 2-cells reversed.  Likewise it takes a 2-monad $T:\f C \to \f C$ to a 2-monad $T^{co}:\f C^{co} \to \f C^{co}$ and one then has, as noted in \cite{Kelly1974Doctrinal}, an equality $\TCAlg_{l}=\TAlg_{c}^{co}$ which restricts to $\TCAlg_{s}=\TAlg_{s}^{co}$.  The $(-)^{co}$ duality naturally extends to a 2-functor $$(-)^{co}:\fcat \to \fcat$$ under whose action we have that $\fTAlg_{c}^{co}= \fTCAlg_{l}$ and moreover that $$U^{co}:\fTAlg_{c}^{co} \to \f C^{co}\hspace{0.5cm}\textnormal{equals}\hspace{0.5cm}U:\fTCAlg_{l} \to \f C^{co}\hspace{0.2cm}.$$
A consequence of this duality is that each theorem about lax morphisms has a dual version concerning colax morphisms.  Indeed all of our definitions and results in the colax case will be dual to those in the lax setting -- though we will \emph{state} these results for colax morphisms it will always suffice to \emph{prove} results only in the lax setting.
\subsection{Equivalence of \f F-categories}
Our monadicity theorems in Section 6 will assert that certain \f F-categories are \emph{equivalent} to \f F-categories of algebras for a 2-monad.  By an equivalence of \f F-categories we mean an equivalence in the 2-category of \f F-categories \fcat, which is to say an equivalence of \f V-categories for $\f V = \f F$.
\\
Recall from \cite{Kelly1982Basic} that a \f V-functor $F:\f A \to \f B$ is an equivalence just when it is \emph{essentially surjective on objects} and \emph{fully faithful} in the enriched sense.  When $\f V = \f F$ the first condition amounts to $F_{\tau}:\f A_{\tau} \to \f B_{\tau}$ being essentially surjective on objects, in the usual sense.  This also implies the weaker statement that $F_{\lambda}:\f A_{\lambda} \to \f B_{\lambda}$ is essentially surjective on objects.  Enriched fully faithfulness of $F$ amounts to ${F_{\tau}}_{A,B}:\f A_{\tau}(A,B) \to \f B_{\tau}(FA,FB)$ and ${F_{\lambda}}_{A,B}:\f A_{\lambda}(A,B) \to \f B_{\lambda}(FA,FB)$ being isomorphisms of categories for each pair $A,B \in \g A$.\\  
We conclude that $F$ is an equivalence of \f F-categories just when both 2-functors $F_{\tau}$ and $F_{\lambda}$ are essentially surjective on objects and 2-fully faithful, which is to say that both $F_{\tau}$ and $F_{\lambda}$ are 2-equivalences, or equivalences of 2-categories.

\section{Doctrinal adjunction and \f F-categorical lifting properties}
If a strict monoidal functor has a right adjoint that right adjoint admits a unique lax monoidal structure such that the adjunction lifts to a monoidal adjunction.  This is an instance of \emph{doctrinal adjunction} -- the topic of the present section.  We begin by recalling Kelly's treatment of doctrinal adjunction in the setting of 2-monads, recasting the notion in \f F-categorical terms, so that the above special case becomes the assertion that the forgetful \f F-functor $V:\fMonCat_{l} \to \Cat$ satisfies $l$-doctrinal adjunction -- we treat the cases $w \in \{l,p,c\}$.  In Section 3.2 we define the closely related notion of a $w$-doctrinal \f F-functor before showing that the $w$-doctrinal \f F-functors form an orthogonality class in \fcat.

\subsection{Doctrinal adjunction \f F-categorically}
Doctrinal adjunction was first studied in Kelly's paper \cite{Kelly1974Doctrinal} of the same name.  Motivated by the example of adjunctions between monoidal categories amongst others, all known to be describable using 2-monads via clubs \cite{Kelly1974On-clubs} or other techniques, he gave his treatment in the setting of 2-dimensional monad theory.  Let us now recall the relevant aspects of this.  Given $T$-algebras $(A,a)$ and $(B,b)$ and a morphism $f:A \to B$ together with an adjunction $(\epsilon,f \dashv g,\eta)$ in the base, his Theorem 1.2 asserts that there is a bijection between colax algebra morphisms of the form $(f,\overline{f}):(A,a) \rightsquigarrow (B,b)$ and lax morphisms of the form $(g,\overline{g}):(B,b) \rightsquigarrow (A,a)$.  The structure 2-cells $\overline{f}:f.a \Rightarrow b.Tf$ and $\overline{g}:a.Tg \Rightarrow g.b$ are expressed in terms of one another as mates as below
$$
\xy
(00,0)*+{TA}="11";(30,0)*+{TB}="31"; (00,-35)*+{A}="12";(30,-35)*+{B}="32";
(00,-15)*+{TA}="a";(30,-20)*+{B}="b";
{\ar^{b} "31"; "b"}; 
{\ar_{f} "12"; "32"}; 
{\ar_{a} "a"; "12"}; 
{\ar^{Tf} "11"; "31"}; 
{\ar_{1} "11"; "a"}; 
{\ar^{Tg} "31"; "a"}; 
{\ar_{g} "b"; "12"}; 
{\ar^{1} "b"; "32"}; 
{\ar@{=>}_{\overline{g}}(13,-14)*+{};(18,-19)*+{}};
{\ar@{=>}^{T\eta}(8,-6)*+{};(13,-6)*+{}};
{\ar@{=>}_{\epsilon}(18,-30)*+{};(23,-30)*+{}};
\endxy
\hspace{3cm}
\xy
(00,0)*+{TB}="11";(30,0)*+{TA}="31"; (00,-35)*+{B}="12";(30,-35)*+{A}="32";
(00,-15)*+{TB}="a";(30,-20)*+{A}="b";
{\ar^{a} "31"; "b"}; 
{\ar_{g} "12"; "32"}; 
{\ar_{b} "a"; "12"}; 
{\ar^{Tg} "11"; "31"}; 
{\ar_{1} "11"; "a"}; 
{\ar^{Tf} "31"; "a"}; 
{\ar_{f} "b"; "12"}; 
{\ar^{1} "b"; "32"}; 
{\ar@{=>}_{\overline{f}}(18,-20)*+{};(13,-15)*+{}};
{\ar@{=>}_{T\epsilon}(13,-6)*+{};(8,-6)*+{}};
{\ar@{=>}_{\eta}(23,-30)*+{};(18,-30)*+{}};
\endxy
$$
Since lax and colax morphisms cannot be composed this relationship cannot be expressed 2-categorically or indeed \f F-categorically -- it can be captured using \emph{double categories} as in Example 5.4 of \cite{Shulman2011Comparing}.  However if we start with $(f,\overline{f})$ a pseudomorphism then it does live in the same 2-category as the resultant lax morphism $(g,\overline{g}):(B,b) \rightsquigarrow (A,a)$.  Moreover it is shown in Proposition 1.3 of \cite{Kelly1974Doctrinal} that the unit and counit $\eta$ and $\epsilon$ then become algebra 2-cells $\eta:1 \Rightarrow (g,\overline{g}) \circ (f,\overline{f})$ and $\epsilon:(f,\overline{f}) \circ (g,\overline{g}) \Rightarrow 1$ in \TAlgl and so yield an adjunction $(\epsilon,(f,\overline{f}) \dashv (g,\overline{g}),\eta)$ in \TAlgl.  
\\
Dually, if $(g,\overline{g})$ were a pseudomorphism, then upon equipping $f$ with the corresponding colax structure $(f,\overline{f})$ the adjunction $(\epsilon,f \dashv g,\eta)$ lifts to an adjunction $(\epsilon,(f,\overline{f}) \dashv (g,\overline{g}),\eta)$ in the 2-category $\TAlg_{c}$.\\
The invertibility of $\overline{f}$ does not imply the invertibility of its mate $\overline{g}$, or vice-versa.  However, if both $\eta$ and $\epsilon$ are invertible, then $\overline{f}$ is invertible just when $\overline{g}$ is, which is to say that if $(\epsilon,f \dashv g,\eta)$ is an adjoint equivalence and either $f$ or $g$ admits the structure of a pseudomorphism the adjoint equivalence lifts to an adjoint equivalence in $\TAlg_{p}$.\\
Let us now abstract these lifting properties of adjunctions and adjoint equivalences into properties of the forgetful \f F-functors $U:\fTAlg_{w_{1},w_{2}} \to \f C$.  By an \emph{adjunction or adjoint equivalence in an \f F-category} \g A we will mean an adjunction or adjoint equivalence in its 2-category $\f A_{\lambda}$ of loose morphisms.  Given an \f F-functor $H:\g A \to \g B$ and an adjunction $(\epsilon,f \dashv g, \eta)$ in \g B a \emph{lifting of this adjunction along $H$} is an adjunction $(\epsilon^{\prime},f^{\prime} \dashv g^{\prime},\eta^{\prime})$ in \g A such that $Hf^{\prime}=f, Hg^{\prime}=g, H\epsilon^{\prime}=\epsilon$ and $H\eta^{\prime}=\eta$.  We will likewise speak of liftings of adjoint equivalences along $H$.  An \f F-functor $H: \g A \to \g B$ is said to satisfy
\begin{itemize}
\item \emph{weak $l$-doctrinal adjunction} if for each tight arrow $f:A \to B \in \g A$ each adjunction $(\epsilon,Hf \dashv g, \eta)$ in \g B lifts along $H$ to an adjunction in $\g A$ with left adjoint $f$.
\item \emph{weak $p$-doctrinal adjunction} if for each tight arrow $f:A \to B \in \g A$ each adjoint equivalence $(\epsilon,Hf \dashv g, \eta)$ in \g B lifts along $H$ to an adjoint equivalence in $\g A$ with left adjoint $f$.\begin{footnote}{This lifting property appears biased but of course is not, since the left adjoint of an adjoint equivalence is equally its right adjoint.}\end{footnote}
\item \emph{weak $c$-doctrinal adjunction} if for each tight arrow $f:A \to B \in \g A$ each adjunction $(\epsilon,g \dashv Hf, \eta)$ in \g B lifts along $H$ to an adjunction in \g A with right adjoint $f$.
\end{itemize}
The lifting properties for algebras described above can be rephrased as asserting exactly that \emph{the forgetful \f F-functor $U:\fTAlg_{p,w} \to \f C$ satisfies weak $w$-doctrinal adjunction for $w \in \{l,p,c\}$}.  Each of these statements asserts that if we are given a pseudomorphism of algebras whose underlying arrow has some kind of adjoint, then that adjunction lifts in a certain way; of course if the starting pseudomorphism were in fact strict then the same lifting will exist so that, in particular, \emph{the forgetful \f F-functor $U:\fTAlg_{w} \to \f C$ satisfies weak $w$-doctrinal adjunction for $w \in \{l,p,c\}$}.
\\
In fact such forgetful \f F-functors lift these adjunctions \emph{uniquely} -- as will follow upon considering the following simple lifting properties.  Recall that a 2-functor $H:\f A \to \f B$ \emph{reflects identity 2-cells} or is \emph{locally conservative} when it reflects the property of a 2-cell being an identity or an isomorphism, and is \emph{locally faithful} if it reflects the equality of parallel 2-cells.  Let us say that an \f F-functor $H:\g A \to \g B$ has any of these three local properties when its loose part $H_{\lambda}:\f A_{\lambda} \to \f B_{\lambda}$ has them: this means that $H$ has these properties with respect to all 2-cells and not just those between tight morphisms.  The following is evident from the definition of an algebra 2-cell.
\begin{Proposition}
Given $w_{1},w_{2} \in \{s,l,p,c\}$ satisfying $w_{1} \leq w_{2}$ the forgetful \f F-functor $U:\fTAlg_{w_{1},w_{2}} \to \f C$ reflects identity 2-cells, is locally conservative and locally faithful.  In particular these properties are true of each $U:\fTAlg_{w} \to \f C$.
\end{Proposition}
\begin{Definition}
Let $w \in \{l,p,c\}$.  An \f F-functor $H:\g A \to \g B$ is said to satisfy \emph{$w$-doctrinal adjunction} if it satisfies the unique form of weak $w$-doctrinal adjunction.
\end{Definition}
For instance $H:\g A \to \g B$ satisfies $l$-doctrinal adjunction if for each tight arrow $f:A \to B \in \g A$ each adjunction $(\epsilon,Hf \dashv g, \eta)$ in \g B lifts \emph{uniquely} along $H$ to an adjunction in $\g A$ with left adjoint $f$.  Let us note that, since the $(-)^{co}$ duality interchanges left and right adjoints in an \f F-category, $H$ satisfies $c$-doctrinal adjunction just when $H^{co}$ satisfies $l$-doctrinal adjunction.  Since adjoint equivalences are fixed $H$ satisfies $p$-doctrinal adjunction just when $H^{co}$ does.
\begin{Proposition}\label{prop:ids}
Let $w \in \{l,p,c\}$ and consider $H:\g A \to \g B$.
\begin{enumerate}
\item If $H$ satisfies weak $w$-doctrinal adjunction and reflects identity 2-cells it satisfies $w$-doctrinal adjunction.
\item If $H$ is locally conservative and satisfies $l$-doctrinal adjunction or $c$-doctrinal adjunction then $H$ satisfies $p$-doctrinal adjunction.
\end{enumerate}
\end{Proposition}
\begin{proof}
\begin{enumerate}
\item
To prove the cases $w=l$ and $w=p$ it will suffice to show that any two adjunctions $(\epsilon,f \dashv g,\eta)$ and $(\epsilon^{\prime},f \dashv g^{\prime},\eta^{\prime})$ in \g A with common left adjoint and common image under $H$ necessarily coincide in \g A.  Since adjoints are unique up to isomorphism we have $m:g \cong g^{\prime}$ given by the composite
$$\xy
(-20,-15)*+{B}="e0";
(0,0)*+{A}="a0"; (20,0)*+{A}="b0";(0,-15)*+{B}="c0";
{\ar^{1} "a0"; "b0"}; 
{\ar@{~>}_{f} "a0"; "c0"};
{\ar@{~>}^{g} "e0"; "a0"}; 
{\ar_{1} "e0"; "c0"};   
{\ar@{~>}_{g^{\prime}} "c0"; "b0"};   
{\ar@{=>}_{\epsilon}(-8,-8)*+{};(-8,-13)*+{}};
{\ar@{=>}_{\eta^{\prime}}(7,-2)*+{};(7,-7)*+{}};
\endxy$$
Using the triangle equations for $f \dashv g$ and $f \dashv g^{\prime}$ we see that $(m\thing f)\circ \eta=\eta^{\prime}$ and that $\epsilon^{\prime}\circ(f\thing m)=\epsilon$.  Now the image of the 2-cell $m$ under $H$ is an identity by one of the triangle equations for $Hf \dashv Hg = Hg^{\prime}$.  Therefore $m$ is an identity and $g=g^{\prime}$.  Since $m\thing f$ and $f\thing m$ are identities it follows that $\eta=\eta^{\prime}$ and $\epsilon=\epsilon^{\prime}$ too.  The case $w=c$ is dual.
\item Suppose that $H$ satisfies $l$-doctrinal adjunction and is locally conservative.  Then given a tight arrow $f \in \g A$ each adjoint equivalence $(\epsilon,Hf \dashv g,\eta)$ in \g B lifts uniquely to an adjunction $(\epsilon^{\prime},f \dashv g^{\prime},\eta^{\prime})$ in \g A.  Since $H$ is locally conservative both $\epsilon^{\prime}$ and $\eta^{\prime}$ are invertible because their images are.  Therefore the lifted adjunction is an adjoint equivalence so that $H$ satisfies $p$-doctrinal adjunction. The $c$-case is dual.
\end{enumerate}
\end{proof}
\begin{Corollary}
Let $w \in \{l,p,c\}$.  Then $U:\fTAlg_{w} \to \f C$ satisfies $w$-doctrinal adjunction.  Furthermore both $U:\fTAlg_{l} \to \f C$ and $U:\fTAlg_{c} \to \f C$ satisfy $p$-doctrinal adjunction.
\end{Corollary}
\begin{proof}
Since $U:\fTAlg_{w} \to \f C$ satisfies weak $w$-doctrinal adjunction and reflects identity 2-cells it follows from Proposition~\ref{prop:ids}.1 that $U$ satisfies $w$-doctrinal adjunction.  Since it is locally conservative the second part of the claim follows from Proposition~\ref{prop:ids}.2.
\end{proof}
\begin{Example}\label{thm:DoctrinalMonoidal}
In the concrete setting of monoidal categories doctrinal adjunction is well known.  Here we describe only those aspects relevant to our needs: namely, that the forgetful \f F-functors $V:\fMonCat_{w} \to \Cat$ satisfy $w$-doctrinal adjunction for $w \in \{l,p,c\}$.  Consider then a strict monoidal functor $F:\overline{A} \to \overline{B}$ and an adjunction of categories $(\epsilon, F \dashv G,\eta)$.  The right adjoint $G$ obtains the structure of a lax monoidal functor
$\overline{G}=(G,g,g_{0}):\overline{B} \rightsquigarrow \overline{A}$ with constraints $g_{x,y}$ and $g_{0}$ given by the composites
$$\xy
(0,0)*+{Gx\otimes Gy}="00"; (35,0)*+{GF(Gx\otimes Gy)}="10"; (70,0)*+{G(FGx\otimes FGy)}="20";(105,0)*+{G(x\otimes y)}="30";
{\ar^<<<<<<{\eta_{Gx\otimes Gy}} "00";"10"};
{\ar@{=}^{}"10";"20"};
{\ar^<<<<<<{G(\epsilon_{x}\otimes \epsilon_{y})} "20";"30"};
\endxy
$$
and $\eta_{i^{A}}:i^{A} \to GFi^{A} = Gi^{B}$.
It is straightforward to show that, with respect to these constraints, the natural transformations $\epsilon$ and $\eta$ become monoidal transformations.  Therefore we obtain the lifted adjunction $(\epsilon, F \dashv (G,g,g_{0}),\eta)$ in $\fMonCat_{l}$ whose uniqueness follows, using Proposition~\ref{prop:ids}.1, from the fact that $V$ reflects identity 2-cells; thus $V:\fMonCat_{l} \to \Cat$ satisfies $l$-doctrinal adjunction.  The cases $w=p$ and $w=c$ are entirely analogous.
\end{Example}
Note that unless $w = l$ the forgetful \f F-functor $V:\fMonCat_{w} \to \Cat$ does not satisfy $l$-doctrinal adjunction.  That $V:\fMonCat_{l} \to \Cat$ itself does so is due to the fact that the constraints $f_{a,b}:Fa\otimes Fb \to F(a\otimes b)$ and $f_{0}:i^{B} \to Fi^{A}$ point in the correct direction -- into $F$ -- and are not invertible.  Whilst $l$-doctrinal adjunction captures, to some extent, this \emph{laxness} it does not determine in any way the \emph{coherence axioms} for a lax monoidal functor.  This is illustrated by the fact that there is an \f F-category of monoidal categories, strict and \emph{incoherent} lax monoidal functors (these have components $f$ and $f_{0}$ oriented as above but satisfying no equations) and this too sits over \Cat via a forgetful \f F-functor satisfying $l$-doctrinal adjunction.
\subsection{Reflections and doctrinal \f F-functors}
Although each forgetful \f F-functor $U:\fTAlg_{w} \to \f C$ satisfies $w$-doctrinal adjunction it turns out that, insofar as this property characterises such \f F-functors, the only relevant adjunctions take a more specialised form.  We will begin by treating the cases $w=l$ and $w=p$, before treating the case $w=c$ by duality.
\newline{}
Let us call an adjunction $(1,f \dashv g,\eta)$ with tight left adjoint and identity counit an \emph{l-reflection}.  If, in addition, the unit $\eta$ is invertible then we call the adjunction a \emph{p-reflection}.  A $p$-reflection is of course an adjoint equivalence.  \newline{}
We remark that any adjunction $(\epsilon,f \dashv g,\eta)$ is determined by three of its four parts: in particular the unit $\eta$ is the unique 2-cell $1 \Rightarrow g\thing f$ satisfying the triangle equation $(\epsilon \thing f)\circ (f\thing \eta) =1$.  In a $w$-reflection $(1,f \dashv g,\eta)$ the unit $\eta$ is therefore uniquely determined by the adjoints $f \dashv g$ and the knowledge that the counit is the identity -- we can thus faithfully abbreviate a $w$-reflection $(1,f \dashv g,\eta)$ by $f \dashv g$ if convenient.
Whilst we only need to consider liftings of $w$-reflections we also need to consider liftings of \emph{morphisms of $w$-reflections}.  Consider $w$-reflections $(1,f_{1} \dashv g_{1},\eta_{1})$ and $(1,f_{2} \dashv g_{2},\eta_{2})$ and a \emph{tight} commutative square $(r,s):f_{1} \to f_{2}$ as left below.
$$\xy
(0,0)*+{A}="00"; (15,0)*+{C}="10"; 
(0,-10)*+{B}="01"; (15,-10)*+{D}="11";
{\ar^{r} "00";"10"};
{\ar ^{f_{2}}"10";"11"};
{\ar _{f_{1}}"00";"01"};
{\ar _{s} "01";"11"};
 \endxy
 \hspace{2cm}
 \xy
(0,0)*+{A}="00"; (15,0)*+{C}="10"; 
(0,-10)*+{B}="01"; (15,-10)*+{D}="11";
{\ar^{r} "00";"10"};
{\ar@{~>} _{g_{2}}"11";"10"};
{\ar@{~>} ^{g_{1}}"01";"00"};
{\ar _{s} "01";"11"};
 \endxy$$
 We call $(r,s):f_{1} \to f_{2}$ a \emph{morphism of $w$-reflections} if the above square of right adjoints also commutes and furthermore the compatibility with units $r\thing \eta_{1}=\eta_{2}\thing r$ is met.  Compatibility with the identity counits is automatic.  The following lemma is sometimes useful for recognising morphisms of $w$-reflections.
\begin{Lemma}\label{prop:morphAdj}
Let $w \in \{l,p\}$.  Consider $w$-reflections $(1,f_{1} \dashv g_{1},\eta_{1})$ and $(1,f_{2} \dashv g_{2},\eta_{2})$ and a tight commuting square $(r,s):f_{1} \to f_{2}$.  Then $(r,s)$ is a morphism of $w$-reflections just when its mate $m(r,s):g_{2}\thing s \Rightarrow r\thing g_{1}$ is an identity 2-cell.
\end{Lemma} 
\begin{proof}
It suffices to consider the case $w=l$.  Given a morphism of $l$-reflections $(r,s):f_{1} \to f_{2}$
 $$\xy
(0,0)*+{A}="00"; (15,0)*+{B}="10"; 
(0,-15)*+{C}="01"; (15,-15)*+{D}="11";
{\ar^{r} "00";"10"};
{\ar ^{f_{2}}"10";"11"};
{\ar _{f_{1}}"00";"01"};
{\ar _{s} "01";"11"};
 \endxy
\hspace{0.3cm}
\xy
(-20,-15)*+{C}="e0";(35,0)*+{B}="f0";
(0,0)*+{A}="a0"; (15,0)*+{B}="b0";(0,-15)*+{C}="c0";(15,-15)*+{D}="d0";
{\ar^{r} "a0"; "b0"}; 
{\ar_{s} "c0"; "d0"}; 
{\ar^{f_{1}} "a0"; "c0"};
{\ar_{f_{2}} "b0"; "d0"};  
{\ar@{~>}^{g_{1}} "e0"; "a0"}; 
{\ar_{1} "e0"; "c0"};   
{\ar^{1} "b0"; "f0"}; 
{\ar@{~>}_{g_{2}} "d0"; "f0"};   
{\ar@{=>}_{\eta_{2}}(22,-2)*+{};(22,-7)*+{}};
(30,-15)*+{C}="e0";
(50,0)*+{A}="a0"; (70,0)*+{A}="b0";(50,-15)*+{C}="c0";(85,0)*+{B}="f0";(65,-15)*+{D}="d0";
{\ar^{1} "a0"; "b0"}; 
{\ar_{f_{1}} "a0"; "c0"};
{\ar@{~>}^{g_{1}} "e0"; "a0"}; 
{\ar^{r} "b0"; "f0"}; 
{\ar_{1} "e0"; "c0"};   
{\ar@{~>}_{g_{1}} "c0"; "b0"};   
{\ar@{=>}_{\eta_{1}}(57,-2)*+{};(57,-7)*+{}};
{\ar_{s} "c0"; "d0"}; 
{\ar@{~>}_{g_{2}} "d0"; "f0"};   
\endxy$$
consider the mate of the left square: the central composite above.  As $(r,s)$ is a morphism of $l$-reflections we have $\eta_{2}\thing r=r\thing \eta_{1}$ so that the mate reduces to the rightmost composite: this is an identity by the triangle equation for $f_{1} \dashv g_{1}$.\newline{}
Conversely suppose that the mate of $(r,s):f_{1} \to f_{2}$ is an identity.  Then $(r,s)$ certainly commutes with the right adjoints so that it remains to verify compatibility with the units.  This is straightforward.
\end{proof}
Note that each \f F-functor preserves morphisms of $w$-reflections.  The following lifting properties are crucial.  For $w \in \{l,p\}$ we say that an \f F-functor $H: \g A \to \g B$ satisfies:
\begin{itemize}
\item \emph{$w$-Refl} if given a tight arrow $f:A \to B \in \g A$ each $w$-reflection $(1,Hf \dashv g,\eta)$ lifts uniquely along $H$ to a $w$-reflection $(1,f \dashv g^{\prime},\eta^{\prime})$ in $\g A$.
\item \emph{$w$-Morph} if given $w$-reflections $(1,f_{1} \dashv g_{1},\eta_{1})$ and $(1,f_{2} \dashv g_{2},\eta_{2})$ a tight commuting square $(r,s):f_{1} \to f_{2}$ is a morphism of $w$-reflections just when its image $(Hr,Hs):Hf_{1} \to Hf_{2}$ is one.
\end{itemize}
We say that $H:\g A \to \g B$ satisfies \emph{$c$-Refl} or \emph{$c$-Morph} if $H^{co}:\g A^{co} \to \g B^{co}$ satisfies $l$-Refl or $l$-Morph respectively.  We remark that these $c$-variants concern \emph{$c$-reflections} which are adjunctions with \emph{tight right adjoint} and \emph{identity unit}.  
\begin{Definition}
Let $w \in \{l,p,c\}$.  An \f F-functor $H:\g A \to \g B$ is said to be \emph{$w$-doctrinal} if it satisfies $w$-Refl, $w$-Morph and is locally faithful.  We denote the class of $w$-doctrinal \f F-functors by $w$-Doct.
\end{Definition}
The condition that $W$ be locally faithful may seem somewhat unnatural -- see the discussion after Theorem~\ref{thm:orthogonality} for our reasons for including it.\\
Evidently we have that $H$ is $c$-doctrinal just when $H^{co}$ is $l$-doctrinal.  Furthermore $H$ is $p$-doctrinal just when $H^{co}$ is $p$-doctrinal.  Let us now compare these lifting properties with those of Section 3.1.
\begin{Lemma}\label{thm:doctrinalLemma}
Let $w \in \{l,p,c\}$ and consider $H:\g A \to \g B$.
\begin{enumerate}
\item If $H$ satisfies $w$-doctrinal adjunction and reflects identity 2-cells then it satisfies $w$-Refl and $w$-Morph.
\item If $H$ satisfies $w$-doctrinal adjunction, reflects identity 2-cells and is locally faithful then it is $w$-doctrinal.
\item If $H$ is locally conservative, reflects identity 2-cells, is locally faithful and satisfies either $l$ or $c$-doctrinal adjunction then it is $p$-doctrinal.
\end{enumerate}
\end{Lemma}
\begin{proof}
We will only consider the $l$-case of (1) and (2), all being essentially identical, with the $l$ and $c$ cases dual.
\begin{enumerate}
\item Consider a tight morphism $f:A \to B \in \g A$ and $l$-reflection $(1,Hf \dashv g,\eta)$.  Because $H$ satisfies $l$-doctrinal adjunction this lifts uniquely to an adjunction $(\epsilon^{\prime},f \dashv g^{\prime},\eta^{\prime})$ in \g A.  Since $H\epsilon^{\prime}=1$ and $H$ reflects identity 2-cells $\epsilon^{\prime}$ is an identity.  This verifies $l$-Refl.  For $l$-Morph consider $l$-reflections $(1,f_{1} \dashv g_{1},\eta_{1})$ and $(1,f_{2} \dashv g_{2},\eta_{2})$ in \g A and a tight morphism $(r,s):f_{1} \to f_{2}$ such that $(Hr,Hs):Hf_{1} \to Hf_{2}$ is a morphism of $l$-reflections; we must show $(r,s):f_{1} \to f_{2}$ is one too.  By Lemma~\ref{prop:morphAdj} this is equally to say that the mate $m_{r,s}:r\thing g_{1} \Rightarrow g_{2}\thing s$ of this square is an identity 2-cell, so giving a commutative square $(s,r):g_{1} \to g_{2}$.  So it will suffice to check $Hm_{r,s}$ is an identity.  But $Hm_{r,s}=m_{Hr,Hs}$ which is an identity since $(Hr,Hs)$ is a morphism of $l$-reflections.
\item Since $w$-doctrinal just means $w$-Refl and $w$-Morph together with local faithfulness this follows immediately from Part 1.
\item Since by Proposition~\ref{prop:ids}.2 such an \f F-functor satisfies $p$-doctrinal adjunction this follows from Part 2.
\qedhere
\end{enumerate}
\end{proof}

\begin{Corollary}\label{thm:algdoctrinal}
Let $T$ be a 2-monad on \f C.  For each $w \in \{l,p,c\}$ the forgetful \f F-functor $U:\fTAlg_{w} \to \f C$ is $w$-doctrinal.  Furthermore both $U:\fTAlg_{l} \to \f C$ and $U:\fTAlg_{c} \to \f C$ are $p$-doctrinal.
\end{Corollary}
\begin{proof}
Since $U:\fTAlg_{w} \to \f C$ satisfies $w$-doctrinal adjunction, reflects identity 2-cells and is locally conservative the first claim follows from Lemma~\ref{thm:doctrinalLemma}.2; since each such $U$ is also locally conservative the second claim follows from Lemma~\ref{thm:doctrinalLemma}.3.
\end{proof}
Let us conclude by mentioning a further evident example of $w$-doctrinal \f F-functors.

\begin{Proposition}\label{prop:equivalences}
Let $w \in \{l,p,c\}$.  If $H:\g A \to \g B$ is such that $H_{\lambda}:\f A_{\lambda} \to \f B_{\lambda}$ is 2-fully faithful then $H$ is $w$-doctrinal for each $w$.  In particular each equivalence of \f F-categories is $w$-doctrinal for each $w$.
\end{Proposition}

\subsection{A small orthogonality class}
Though not strictly necessary in what follows let us remark that, with the exception of weak $w$-doctrinal adjunction, all of the lifting/reflection properties so far considered are expressible as orthogonal lifting properties in \fcat.  Certainly it is not hard to see that this is true of the property of reflecting identity 2-cells, or of being locally faithful or locally conservative.  Less obvious is that this is true of the notion of $w$-doctrinal adjunction or of the conditions $w$-Refl and $w$-Morph, in particular of the condition $w$-Morph concerning liftings of morphisms of adjunctions.  We describe the conditions $l$-Refl and $l$-Morph here -- the $p$ and $c$ cases being similar.  To this end consider the following $\f F$-category $\g Adj_{l}$ depicted in its entirety on the left below
$$\xy
(0,7)*+{0}="00"; (0,-7)*+{0}="01";
(20,7)*+{1}="10";(20,-7)*+{1}="11";
{\ar_{1} "00"; "01"}; 
{\ar^{f} "00"; "10"};
{\ar@{~>}^{g} "10"; "01"}; 
{\ar_{f} "01"; "11"};
{\ar^{1} "10"; "11"};   
{\ar@{=>}^{\eta}(3,2)*+{};(8,2)*+{}};
\endxy
 \hspace{0.5cm}
 \xy
(0,7)*+{\atwo}="a0"; (20,7)*+{\g Adj_{l}}="b0";(0,-7)*+{\g A}="c0";(20,-7)*+{\g B}="d0";
{\ar^{j} "a0"; "b0"}; 
{\ar_{} "a0"; "c0"}; 
{\ar^{} "b0"; "d0"}; 
{\ar_{H} "c0"; "d0"}; 
{\ar@{.>}|{} "b0"; "c0"}; 
\endxy
\hspace{0.5cm}
 \xy
(0,7)*+{\fcat(\g Adj_{l},\g A)}="a0"; (35,7)*+{\fcat(\g Adj_{l},\g B)}="b0";(0,-7)*+{\fcat(\atwo,\g A)}="c0";(35,-7)*+{\fcat(\atwo ,\g B)}="d0";
{\ar^{H_{*}} "a0"; "b0"}; 
{\ar_{j^{*}} "a0"; "c0"}; 
{\ar^{j^{*}} "b0"; "d0"}; 
{\ar_{H_{*}} "c0"; "d0"}; 
\endxy
 $$
 where $g\thing f$ is loose, $f\thing g=1$, $f\thing \eta=1$ and $\eta \thing g=1$.  This is the free adjunction with identity counit and tight left adjoint $f$.   It has a single non-identity tight arrow $f$ so that ${\g Adj_{l}}_{\tau}$ equals the free tight arrow $\atwo$; therefore the inclusion ${\f Adj_{l}}_{\tau} \to \g Adj_{l}$ is the \f F-functor $j:\atwo \to \g Adj_{l}$ selecting $f$.  Now to give a commutative square as in the middle above is to give a tight arrow $f \in \g A$ and adjunction $(1,Hf \dashv g,\eta)$ in \g B.  To give a filler is to lift the adjunction along $H$ to an adjunction $(1,f \dashv g^{\prime},\eta^{\prime})$ in \g A; thus $H$ is orthogonal to $j$ when condition $l$-Refl is met.  The conditions $l$-Refl and $l$-Morph together assert exactly that $j$ and $H$ are orthogonal in \fcat as a 2-category -- this means that the right square above is a pullback in \CAT.  In fact by a standard argument the conditions $l$-Refl and $l$-Morph jointly amount to ordinary (1-categorical) orthogonality against the single \f F-functor $j \times 1:\atwo \times \atwo \to \g Adj_{l} \times \atwo$, of which $j$ is a retract.\newline{}
Therefore the class of $w$-doctrinal \f F-functors, $\wdoct$, forms an orthogonality class in the category of \f F-categories and \f F-functors.  The following section is geared towards establishing sufficient conditions on an \f F-category \g A under which the inclusion $j:\f A_{\tau} \to \g A$ belongs to $^{\bot}\wdoct$ -- is orthogonal to each $w$-doctrinal \f F-functor.

\section{Loose morphisms as spans}
We now consider completeness properties of \f F-categories appropriate to those of the form $\fTAlg_{w}$.  In an \f F-category \g A with such completeness properties we can represent loose morphisms in \g A by certain tight spans.  In the lax setting, for instance, each loose morphism $f:A \rightsquigarrow B$ is represented as a tight span $C_{f}:A \hto B$
$$\xy
(0,0)*+{C_{f}}="a0"; (-20,0)*+{A}="b0";(20,0)*+{B}="c0";
{\ar_{p_{f}\dashv r_{f}} "a0"; "b0"}; 
{\ar^{q_{f}} "a0"; "c0"}; 
\endxy$$
by taking its \emph{colax} limit.  The tight left leg $p_{f}:C_{f} \to A$ is part of an $l$-reflection $(1,p_{f} \dashv r_{f},\eta_{f})$ from which $f$ can be recovered as the composite $q_{f}\thing r_{f}:A \rightsquigarrow C_{f} \to B$.  In the present section we analyse this representation of loose morphisms by tight spans in detail, beginning with a discussion of the relevant limits.

\subsection{Limits of loose morphisms}
The main reason that \f F-categories were introduced in \cite{Lack2011Enhanced} was to capture the behaviour of limits in 2-categories of the form $\TAlg_{w}$ for $w \in \{l,p,c\}$.  In such 2-categories many \emph{2-categorical} limits have an \f F-categorical aspect -- namely, that the projections from the limit are \emph{strict and jointly detect strictness}.  Interpreted in the \f F-category $\fTAlg_{w}$ this asserts that the \emph{limit projections are tight and jointly detect tightness}.  This latter property is exactly that which distinguishes \f F-categorical limits (in the sense of $\f F$-enriched category theory) from 2-categorical ones (see Proposition 3.6 of \cite{Lack2011Enhanced}).   We will only need a few basic \f F-categorical limits and our treatment, in what follows, is elementary.\newline{}
There are three limits to consider here -- \emph{colax, pseudo and lax limits of loose morphisms} -- these correspond, in turn, to lax, pseudo and colax morphisms.  As we always lead with the case of lax morphisms we focus here primarily on colax limits.\newline{}
Given a loose morphism $f:A \rightsquigarrow B \in \g A$ its (colax/pseudo/lax)-limit consists of an object $C_{f}/P_{f}/L_{f}$ and a (colax/pseudo/lax)-cone $(p_{f},\lambda_{f},q_{f})$ as below
$$
\xy
(0,0)*+{C_{f}}="a0"; (-15,-15)*+{A}="b0";(15,-15)*+{B}="c0";
{\ar_{p_{f}} "a0"; "b0"}; 
{\ar^{q_{f}} "a0"; "c0"}; 
{\ar@{~>}_{f} "b0"; "c0"}; 
{\ar@{=>}_{\lambda_{f}}(0,-5)*+{};(0,-11)*+{}};
\endxy
\hspace{1cm}
\xy
(0,0)*+{P_{f}}="a0"; (-15,-15)*+{A}="b0";(15,-15)*+{B}="c0";
{\ar_{p_{f}} "a0"; "b0"}; 
{\ar^{q_{f}} "a0"; "c0"}; 
{\ar@{~>}_{f} "b0"; "c0"}; 
(0,-8)*{{\cong}_{\lambda_{f}}};
\endxy
\hspace{1cm}
\xy
(0,0)*+{L_{f}}="a0"; (-15,-15)*+{A}="b0";(15,-15)*+{B}="c0";
{\ar_{p_{f}} "a0"; "b0"}; 
{\ar^{q_{f}} "a0"; "c0"}; 
{\ar@{~>}_{f} "b0"; "c0"}; 
{\ar@{=>}_{\lambda_{f}}(0,-11)*+{};(0,-5)*+{}};
\endxy$$
with both projections $p_{f}$ and $q_{f}$ \emph{tight}.\newline{}
The colax limit $C_{f}$\begin{footnote}{Colax limits of arrows are usually called oplax limits of arrows.  We prefer colax limits here since they are lax limits in $\g A^{co}$ rather than $\g A^{op}$ and, similarly, sit better with our usage of lax and colax morphisms.}\end{footnote} is required to be the usual 2-categorical \emph{colax limit of an arrow} \cite{Lack2002Limits} in $\f A_{\lambda}$: this means that it has the following two universal properties.
\begin{enumerate}
\item Given any \emph{colax cone} $(r,\alpha,s)$ as below
$$\xy
(0,0)*+{X}="a0"; (-15,-15)*+{A}="b0";(15,-15)*+{B}="c0";
{\ar@{~>}_{r} "a0"; "b0"}; 
{\ar@{~>}^{s} "a0"; "c0"}; 
{\ar@{~>}_{f} "b0"; "c0"}; 
{\ar@{=>}_{\alpha}(0,-5)*+{};(0,-11)*+{}};
\endxy$$
there exists a unique $t:X \rightsquigarrow C_{f}$ satisfying $$p_{f}\thing t=r\textnormal{, }q_{f}\thing t=s\textnormal{ and }\lambda_{f}\thing t=\alpha$$
\item Given a pair of colax cones $(r,\alpha,s)$ and $(r^{\prime},\alpha^{\prime},s^{\prime})$ with common base $X$ together with 2-cells $\theta_{r}:r \Rightarrow r^{\prime} \in \f A_{\lambda}(X,A)$ and $\theta_{s}:s \Rightarrow s^{\prime} \in \f A_{\lambda}(X,B)$ satisfying
$$\xy
(0,0)*+{s}="00"; 
(15,0)*+{s^{\prime}}="10"; (0,-10)*+{fr}="01";
(15,-10)*+{fr^{\prime}}="11";
{\ar@{=>}^{\theta_{s}} "00"; "10"}; 
{\ar@{=>}^{\alpha^{\prime}} "10"; "11"}; 
{\ar@{=>}_{\alpha} "00"; "01"}; 
{\ar@{=>}_{f\theta_{r}} "01"; "11"}; 
\endxy$$
there exists a unique 2-cell $\phi:t \Rightarrow t^{\prime} \in \f A_{\lambda}(X,C_{f})$ between the induced factorisations such that $$p_{f}\thing \phi=\theta_{r}\textnormal{ and }q_{f}\thing \phi=\theta_{s} \textnormal{ .}$$
\end{enumerate}
For $C_{f}$ to be the colax limit of $f$ in the \f F-categorical sense we must also have
\begin{enumerate}
\item[(3)]
A morphism $t:X \rightsquigarrow C_{f}$ is tight just when $p_{f}\thing t$ and $q_{f}\thing t$ are tight -- \emph{the projections jointly detect tightness.}
\end{enumerate}
In the case of the pseudolimit $P_{f}$ the 2-cell $\lambda_{f}$ is required to be invertible.  If we call those colax cones with an invertible 2-cell \emph{pseudo-cones} then the universal properties of (1) and (2) above are only changed by replacing colax cones by pseudo-cones -- thus $(p_{f},\lambda_{f},q_{f})$ is the universal pseudo-cone.  The \f F-categorical aspect of (3) remains the same.  The lax limit of $f$ is simply the colax limit in $\g A^{co}$.  \newline{}
For $w \in \{l,p,c\}$ let us write \emph{$w$-limit of a loose morphism} as an abbreviation, so that $c$-limit stands for colax limit, for instance.  Now we mentioned above that lax morphisms correspond to colax limits and so on.  To capture this let us set $\overline{l}=c$, $\bar{p}=p$ and $\bar{c}=l$ as in \cite{Lack2011Enhanced}: the correspondence is captured by the following result.
\begin{Proposition}\label{prop:limits}
Let $w \in \{l,p,c\}$.  The forgetful \f F-functor $U:\fTAlg_{w} \to \f C$ creates $\overline{w}$-limits of loose morphisms.
\end{Proposition} 
In its \f F-categorical formulation above this is a specialisation of Theorem 5.13 of \cite{Lack2011Enhanced} which characterises those limits created by the forgetful \f F-functors $U:\fTAlg_{w} \to \f C$.  However for $\overline{w}$-limits of loose morphisms, as concern us, the result goes back to \cite{Blackwell1989Two-dimensional} and \cite{Lack2002Limits}.\newline{}
Because lax limits are colax limits in $\g A^{co}$ we can and will avoid them entirely.  In order to work with colax and pseudolimits simultaneously let us introduce a final piece of notation.  For $w \in \{l,p\}$ we will use 
$$\xy
(0,0)*+{\bar{W}_{f}}="a0"; (-15,-15)*+{A}="b0";(15,-15)*+{B}="c0";
{\ar_{p_{f}} "a0"; "b0"}; 
{\ar^{q_{f}} "a0"; "c0"}; 
{\ar@{~>}_{f} "b0"; "c0"}; 
{\ar@{=>}_{\lambda_{f}}(0,-5)*+{};(0,-11)*+{}};
\endxy$$
to denote the $\overline{w}$-limit of $f$ and universal $\overline{w}$-cone: so $C_{f}$ and its colax cone when $w=l$ and $P_{f}$ when $w=p$.  When $w=p$ the 2-cell $\lambda_{f}$ should be interpreted as invertible.  

\begin{Example}
In \Cat the colax limit of a functor $F:A \to B$ is given by the comma category $B/F$: this has objects $(x,\alpha:x \to Fa,a)$ and morphisms $(r,s):(x,\alpha,a) \to (y,\beta,b)$ given by pairs of arrows $r:x \to y \in \f B$ and $s:a \to b \in \f A$ rendering commutative the square on the left.
$$
\xy
(0,0)*+{x}="00"; (20,0)*+{Fa}="10"; (0,-12)*+{y}="01"; (20,-12)*+{Fb}="11"; {\ar^{\alpha} "00";"10"};
{\ar_{\beta} "01";"11"};
{\ar_{r}"00";"01"};
{\ar^{Fs}"10";"11"};
\endxy
\hspace{2cm}
\xy
(0,0)*+{B/F}="a0"; (-15,-12)*+{A}="b0";(15,-12)*+{B}="c0";
{\ar_{p} "a0"; "b0"}; 
{\ar^{q} "a0"; "c0"}; 
{\ar_{F} "b0"; "c0"}; 
{\ar@{=>}_{\lambda}(0,-4)*+{};(0,-9)*+{}};
\endxy$$
The projections $p:B/F \to A$ and $q:B/F \to B$ of the colax cone $(p,\lambda,q)$ act on a morphism $(r,s)$ of $B/F$ as $p(r,s)=s:a \to b$ and $q(r,s)=r:x \to y$; the value of $\lambda:q \Rightarrow pF$ at $(x,\alpha,a)$ is simply the morphism $\alpha:x \to Fa$ itself.  
\\
The pseudolimit of $F$ is the full subcategory of $B/F$ whose objects are those pairs $(\alpha:x \to Fa,a)$ with $\alpha$ invertible, whilst the lax limit of $F$ is the comma category $F/B$.
\end{Example}
\begin{Example}\label{thm:LimitsMonoidal}
It is illuminating to consider the colax limit of a lax monoidal functor $F=(F,f,f_{0}):\overline{A} \rightsquigarrow \overline{B}$.   The forgetful \f F-functor $U:\fMonCat_{l} \to \Cat$ creates these limits: to see how this goes first consider the colax limit of the functor $F$, the comma category $B/F$ equipped with its colax cone $(p,\lambda,q)$ described above.  The crux of the argument is to show that this lifts uniquely to a colax cone in $\fMonCat_{l}$: that \emph{$B/F$ admits a unique monoidal structure such that $p$ and $q$ become strict monoidal and $\lambda$ a monoidal transformation}.  \newline{}
So consider two objects $(x,\alpha,a)$ and $(y,\beta,b)$ of $B/F$: if $p$ and $q$ are to be strict monoidal the tensor product $(x,\alpha,a)\otimes (y,\beta,b)$ must certainly be of the form $(x\otimes y,\theta,a\otimes b)$; furthermore the tensor condition for $\lambda$ to be a monoidal transformation interpreted at this pair asserts precisely that $(x,\alpha,a)\otimes (y,\beta,b)$ equals
$$
\xy
(0,0)*+{x\otimes y}="00"; (30,0)*+{Fa\otimes Fb}="10"; (60,0)*+{F(a\otimes b)}="20";
{\ar^{\alpha \otimes \beta} "00";"10"};
{\ar^{f_{a,b}}"10";"20"};
\endxy
$$
Likewise the unit condition for a monoidal transformation forces us to define the unit of $B/F$ to be $(f_{0}:i^{B} \to Fi^{A},i^{A})$.  For $p$ and $q$ to preserve tensor products of morphisms we must define the tensor product as $(r,s)\otimes (r^{\prime},s^{\prime})=(r\otimes r^{\prime},s\otimes s^{\prime})$ at morphisms of $B/F$ -- to say the resulting pair is a morphism of $B/F$ is then to say that the following square is commutative.
$$\xy
(0,0)*+{x\otimes y}="00"; (30,0)*+{Fa\otimes Fb}="10"; (60,0)*+{F(a\otimes b)}="20";(0,-12)*+{x^{\prime}\otimes y^{\prime}}="01"; (30,-12)*+{Fa^{\prime}\otimes Fb^{\prime}}="11"; (60,-12)*+{F(a^{\prime}\otimes b^{\prime})}="21";
{\ar^{\alpha \otimes \beta} "00";"10"};
{\ar ^{f_{a,b}}"10";"20"};
{\ar_{\alpha^{\prime} \otimes \beta^{\prime}} "01";"11"};
{\ar _{f_{a^{\prime},b^{\prime}}}"11";"21"};
{\ar_{r\otimes r^{\prime}}"00";"01"};
{\ar^{F(s\otimes s^{\prime})}"20";"21"};
{\ar^{Fs\otimes Fs^{\prime}}"10";"11"};
\endxy$$
The left square trivially commutes and the right square commutes by naturality of the $f_{a,b}$.  It remains to give the associator and the left and right unit constraints for the monoidal structure on $B/F$ -- this is where the coherence axioms for a lax monoidal functor finally come into play.  Certainly if $p$ and $q$ are to be strict monoidal they must preserve the associators strictly: this means that the associator at a triple of objects $((x,\alpha,a),(y,\beta,b),(z,\gamma,c))$ of $B/F$ must be given by $(\lambda^{B}_{x,y,z},\lambda^{A}_{a,b,c})$.  To say that this is a morphism of $B/F$ is equally to say that the composite square 
$$\xy
(0,0)*+{(x\otimes y)\otimes z}="00"; (40,0)*+{(Fa\otimes Fb)\otimes Fc}="10"; (78,0)*+{F(a\otimes b)\otimes Fc}="20"; (113,0)*+{F((a\otimes b)\otimes c)}="30"; 
(0,-12)*+{x\otimes (y\otimes z)}="01"; (40,-12)*+{Fa\otimes (Fb\otimes Fc)}="11"; (78,-12)*+{Fa\otimes F(b\otimes c)}="21"; (113,-12)*+{F(a\otimes (b\otimes c))}="31"; 
{\ar^<<<<<<<{(\alpha \otimes \beta)\otimes \gamma} "00";"10"};
{\ar ^{f_{a,b}\otimes 1}"10";"20"};
{\ar^{f_{a\otimes b,c}}"20";"30"};
{\ar_<<<<<<<{\alpha \otimes (\beta \otimes \gamma)} "01";"11"};
{\ar _{1\otimes f_{b,c}}"11";"21"};
{\ar_{f_{a,b\otimes c}}"21";"31"};
{\ar_{\lambda^{B}_{x,y,z}}"00";"01"};
{\ar^{\lambda^{B}_{Fa,Fb,Fc}}"10";"11"};
{\ar^{F\lambda^{A}_{a,b,c}}"30";"31"};
\endxy$$
is commutative.  The left square commutes by naturality of the associators in $B$ with the right square asserting exactly the associativity condition for a lax monoidal functor.  Similarly the left and right unit constraints at $(x,\alpha,a)$ must be given by $(\rho^{B}_{l}x,\rho^{A}_{l}a)$ and $(\rho^{B}_{r}x,\rho^{A}_{r}a)$ -- that these lift to isomorphisms of $B/F$ likewise correspond to the left and right unit conditions for a lax monoidal functor.  Having given the monoidal structure for $B/F$ it remains to check it verifies the axioms for a monoidal category, but all of these clearly follow from the corresponding axioms for $\overline{A}$ and $\overline{B}$ because $p$ and $q$ preserve the structure strictly and are jointly faithful.
\\
Finally one needs to verify that this uniquely lifted colax cone in $\fMonCat_{l}$ satisfies the universal property of the colax limit of $(F,f,f_{0})$ therein.  That $p$ and $q$ jointly detect tightness follows from the fact that they jointly reflect identity arrows -- from here it is straightforward to verify that $B/F$ has the universal property of the colax limit in $\MonCat_{l}$.  For $w \in \{p,c\}$ one constructs the $\bar{w}$-limit of a loose morphism in $\fMonCat_{w}$ in an entirely similar way.
\end{Example}
Observe that in lifting the colax cone $(p,\lambda,q)$ to $\fMonCat_{l}$ we used all of the coherence axioms for a lax monoidal functor, and indeed these generating coherence axioms are required for the colax cone to lift.  Thus while $l$-doctrinal adjunction is related to the \emph{laxness} -- orientation and non-invertibility -- of our lax monoidal functors, colax limits of loose morphisms concern the \emph{coherence} axioms these lax morphisms must satisfy.

\subsection{Loose morphisms as tight spans}
For $w \in \{l,p\}$ we suppose that \g A admits $\bar{w}$-limits of loose morphisms.  Then given $f:A \rightsquigarrow B$ we have the commutative triangle on the left below.
$$\xy
(0,5)*+{A}="a0"; (-15,-10)*+{A}="b0";(15,-10)*+{B}="c0";
{\ar_{1} "a0"; "b0"}; 
{\ar@{~>}^{f} "a0"; "c0"}; 
{\ar@{~>}_{f} "b0"; "c0"}; 
\endxy
\hspace{1.5cm}
\xy
\textnormal{=}
\endxy
\hspace{1.5cm}
\xy
(0,15)*+{A}="d0";
(0,0)*+{\bar{W}_{f}}="a0"; (-15,-15)*+{A}="b0";(15,-15)*+{B}="c0";
{\ar@{~>}_{r_{f}} "d0"; "a0"}; 
{\ar@/_1.5pc/_{1} "d0"; "b0"}; 
{\ar@{~>}@/^1.5pc/^{f} "d0"; "c0"}; 
{\ar_{p_{f}} "a0"; "b0"}; 
{\ar^{q_{f}} "a0"; "c0"}; 
{\ar@{~>}_{f} "b0"; "c0"}; 
{\ar@{=>}_{\lambda_{f}}(0,-5)*+{};(0,-11)*+{}};
\endxy$$
By the universal property of $\bar{W}_{f}$ we obtain a unique 1-cell $r_{f}:A \rightsquigarrow \bar{W}_{f}$ satisfying
\begin{equation}\label{eq:conR}
p_{f}\thing r_{f}=1, \textnormal{ } q_{f}\thing r_{f}=f \textnormal{ and }\lambda_{f}\thing r_{f}=1
\end{equation}
as expressed in the equality of pasting diagrams above.  Since $p_{f}$ and $q_{f}$ jointly detect tightness we also have that
\begin{equation}\label{eq:rTight}
r_{f}\textnormal{ is tight just when }f\textnormal{ is.}
\end{equation}

\begin{Proposition}\label{prop:factor}
Given $f:A \rightsquigarrow B$ as above we have a $w$-reflection $(1,p_{f} \dashv r_{f},\eta_{f})$ where $\eta_{f}:1 \Rightarrow r_{f}p_{f}$ is the unique 2-cell satisfying 
\begin{equation}\label{eq:adjR}
p_{f}\thing \eta_{f}=1\textnormal{ and }q_{f}\thing \eta_{f}=\lambda_{f}.
\end{equation}
\end{Proposition}
\begin{proof}
Let us consider firstly the case $w=l$.  We need to give a unit $\eta_{f}:1 \Rightarrow r_{f}.p_{f}$.  To give such a 2-cell is, by the 2-dimensional universal property of $C_{f}$, equally to give 2-cells $\theta_{1}:p_{f}\thing (1) \Rightarrow p_{f}\thing (r_{f}\thing p_{f})$ and $\theta_{2}:q_{f}\thing (1) \Rightarrow q_{f}\thing (r_{f}\thing p_{f})$ satisfying $\theta_{1} \circ \lambda_{f}=(\lambda_{f}\thing r_{f}\thing p_{f}) \circ \theta_{2}$.  We take $\theta_{1}$ to be the identity and $\theta_{2}$ to be $\lambda_{f}:q_{f} \Rightarrow f\thing p_{f}$; the required equality involving $\theta_{1}$ and $\theta_{2}$ is then the assertion that $\lambda_{f}$ equals itself.  We thus obtain a unique $\eta_{f}:1 \Rightarrow r_{f}\thing p_{f}$ such $p_{f}\thing \eta_{f}=1$ and $q_{f}\thing \eta_{f}=\lambda_{f}$.  If the identity 2-cell $p_{f}\thing q_{f}=1$ is to be the counit of the adjunction then the triangle equations become $p_{f}\thing \eta_{f}=1$ and $\eta_{f}\thing r_{f}=1$.  So it remains to check that $\eta_{f}\thing r_{f}=1$ for which it suffices, again by the 2-dimensional universal property of $C_{f}$, to show that $p_{f}\thing \eta_{f}\thing r_{f}=1$ and $q_{f}\thing \eta_{f}\thing r_{f}=1$.  The first of these holds since $p_{f}\thing \eta_{f}=1$; the second since $q_{f}\thing \eta_{f}=\lambda_{f}$ and $\lambda_{f}\thing r_{f}=1$.\newline{}
The case $w=p$ is essentially identical -- the key point is that the 2-cells $\theta_{1}=1$ and $\theta_{2}=\lambda_{f}$ used above to construct $\eta_{f}$ are now both invertible.  That $\eta_{f}$ is itself invertible follows from the fact that $p_{f}$ and $q_{f}$ are jointly conservative -- this conservativity follows from the 2-dimensional universal property of $P_{f}$.
\end{proof} 
The above constructions have their genesis in the proof of Theorem 4.2 of \cite{Blackwell1989Two-dimensional}, in which \f F-categorical aspects of pseudolimits of arrows in $\TAlg_{p}$ were used to study establish properties of pseudomorphism classifiers.  If we ignore \f F-categorical aspects then the above constructions and resulting factorisations $f=q_{f}\thing r_{f}$ have appeared in other contexts too.  In the pseudolimit case the factorisation is the (trivial cofibration, fibration)-factorisation of the natural model structure on a 2-category \cite{Lack2007Homotopy-theoretic}.  In \Cat the factorisation $q_{f}\thing r_{f}:A \to C_{f} \to B$ of a functor $f$ through its colax limit coincides with its factorisation $A \to B/f \to B$ through the comma category $B/f$ -- this is the factorisation $(Lf,Rf)$ of a natural weak factorisation system on \Cat described in \cite{Grandis2006Natural}.  

Let us return to $f:A \rightsquigarrow B$ as in Proposition~\ref{prop:factor}.  By that result we have a span of tight morphisms (\emph{a tight span}) as below:
$$
\xy
(0,0)*+{A}="0";(20,0)*+{\bar{W}_{f}}="1";(40,0)*+{B}="2";
{\ar_{p_{f} \dashv r_{f}} "1"; "0"}; {\ar^{q_{f}} "1"; "2"}; 
\endxy
$$
whose left leg $p_{f}$ is equipped with the structure of a $w$-reflection $(1,p_{f} \dashv r_{f},\eta_{f})$.  More generally let us use the term \emph{$w$-span} to refer to a tight span equipped with the structure of a $w$-reflection on its left leg.  Now let $\bar{W}_{f}:A \hto B$ denote the $w$-span just described.  Given $f:A \rightsquigarrow B$ and $g:B \rightsquigarrow C$ we are going to show that $\bar{W}_{f}:A \hto B$ and $\bar{W}_{g}:B \hto C$ can be composed to give a $w$-span $\bar{W}_{g}\bar{W}_{f}:A \hto C$ and furthermore we will study the relationship between $\bar{W}_{g}\bar{W}_{f}$ and $\bar{W}_{gf}$.\newline{}  
In order to consider composition of such spans we will need to consider \emph{tight pullbacks}.  Given tight morphisms $f:A \to C$ and $g:B \to C$ in \g A the tight pullback $D$ of $f$ and $g$
$$\xy
(0,0)*+{D}="a0"; (15,0)*+{A}="b0";(0,-10)*+{B}="c0";(15,-10)*+{C}="d0";
{\ar^{p} "a0"; "b0"}; 
{\ar_{g} "c0"; "d0"}; 
{\ar_{q} "a0"; "c0"};
{\ar^{f} "b0"; "d0"};  
\endxy$$
is the pullback in the 2-category $\f A_{\lambda}$ with, moreover, both projections $p$ and $q$ tight and jointly detecting tightness.  This is equally to say that $D$ is a pullback in $\f A_{\tau}$ which is preserved by the inclusion $j:\f A_{\tau} \to \f A_{\lambda}$.

\begin{Lemma}\label{prop:pullbacks1}
Let $w \in \{l,p\}$ and \g A admit $\bar{w}$-limits of loose morphisms.  At $f:A \rightsquigarrow B$ consider the induced tight projection $p_{f}:\bar{W}_{f} \to A$ from the limit.  The tight pullback of $p_{f}$ along any tight morphism $g:C \to A$ exists.
\end{Lemma}
\begin{proof}
In fact the tight pullback is given by $\bar{W}_{fg}$, the $w$-limit of the composite $fg:C \rightsquigarrow A$.  For observe that by the universal property of $\bar{W}_{f}$ the $\overline{w}$-cone $(g\thing p_{fg},\lambda_{fg},q_{fg})$ induces a unique tight map $t:\bar{W}_{fg} \to \bar{W}_{g}$ such that the left square below commutes
$$\xy
(0,0)*+{\bar{W}_{fg}}="00"; (20,0)*+{\bar{W}_{f}}="10"; (40,0)*+{B}="20";(0,-12)*+{C}="01"; (20,-12)*+{A}="11"; (40,-12)*+{B}="21";
{\ar^{t} "00";"10"};
{\ar ^{q_{f}}"10";"20"};
{\ar_{g} "01";"11"};
{\ar @{~>}_{f}"11";"21"};
{\ar_{p_{fg}}"00";"01"};
{\ar^{1}"20";"21"};
{\ar_{p_{f}}"10";"11"};
{\ar@{=>}_{\lambda_{f}}(32,-3)*+{};(28,-9)*+{}};
\endxy$$
and such that $\lambda_{f}\thing t=\lambda_{gf}$.  Now the universal property of $\bar{W}_{fg}$ implies that the left square is a tight pullback.
\end{proof}
\begin{Lemma}\label{prop:pullbacks2}
For $w \in \{l,p\}$ consider a tight pullback square
$$\xy
(0,0)*+{A}="00"; (15,0)*+{B}="10";(0,-10)*+{C}="01";(15,-10)*+{D}="11";
{\ar^{r} "00"; "10"}; 
{\ar_{s} "01"; "11"}; 
{\ar_{f_{1}} "00"; "01"};
{\ar^{f_{2}} "10"; "11"};  
\endxy$$
and $w$-reflection $(1,f_{2} \dashv g_{2},\eta_{2})$.  There exists a unique $w$-reflection $(1,f_{1} \dashv g_{1},\eta_{1})$ such that $(r,s):f_{1} \to f_{2}$ is a morphism of $w$-reflections.
\end{Lemma}
\begin{proof}
First suppose that $w=l$.  Then the left square below
$$\xy
(0,0)*+{C}="00"; (15,0)*+{B}="10";(0,-10)*+{C}="01";(15,-10)*+{D}="11";
{\ar@{~>}^{g_{2}\thing s} "00"; "10"}; 
{\ar_{s} "01"; "11"}; 
{\ar_{1} "00"; "01"};
{\ar^{f_{2}} "10"; "11"};  
\endxy
\hspace{2cm}
\xy
(0,0)*+{C}="00"; (15,0)*+{D}="10";(0,-10)*+{A}="01";(15,-10)*+{B}="11";
{\ar^{s} "00"; "10"}; 
{\ar_{r} "01"; "11"}; 
{\ar@{~>}_{g_{1}} "00"; "01"};
{\ar@{~>}^{g_{2}} "10"; "11"};  
\endxy$$
commutes and so induces a unique loose morphism $g_{1}:C \rightsquigarrow A$ such that $f_{1}\thing g_{1}=1$ and such that the right square commutes.  These necessary commutativities will ensure the claimed uniqueness.  By the 2-dimensional universal property of the pullback there exists a unique 2-cell $\eta_{1}:1 \Rightarrow g_{1}\thing f_{1}$ such that $r\thing \eta_{1}=\eta_{2}\thing r$ and $f_{1}\thing \eta_{1}=1$.  The other triangle equation $\eta_{1}\thing f_{1}=1$ also follows from the universal property of the pullback.  Since $(r,s)$ commutes with both adjoints and the units it is a morphism of $l$-reflections.  Note that if $\eta_{2}$ is invertible then, since the pullback projections are jointly conservative, so too is $\eta_{1}$.  This gives the case $w=p$.
\end{proof}
Now given $f:A \rightsquigarrow B$ and $g:B \rightsquigarrow C$ we can form the composite span
$$\xy
(0,0)*+{A}="00";(15,15)*+{\bar{W}_{f}}="11";(30,0)*+{B}="20";(45,15)*+{\bar{W}_{g}}="31";(60,0)*+{C}="40";(30,30)*+{\bar{W}_{g}\bar{W}_{f}}="22"; 
{\ar_{p_{f}\dashv r_{f}} "11"; "00"};
{\ar_{q_{f}} "11"; "20"}; 
{\ar_{p_{g}\dashv r_{g}} "31"; "20"};
{\ar^{q_{g}} "31"; "40"}; 
{\ar_{p_{g,f}\dashv r_{g,f}} "22"; "11"};
{\ar^{q_{g,f}} "22"; "31"};  
\endxy$$
in which the central square is the tight pullback of $p_{g}$ along $q_{f}$.  (This pullback exists by Lemma~\ref{prop:pullbacks1}).  By Lemma~~\ref{prop:pullbacks2} there exists a unique $w$-reflection $(1,p_{g,f}\dashv r_{g,f},\eta_{g,f})$ such that 
\begin{equation}\label{eq:spancomp}
(q_{g,f},q_{f}):p_{g,f} \to p_{g}\textnormal{ is a morphism of }w\textnormal{-reflections.}
\end{equation}
We can then compose the $w$-reflections $p_{f} \dashv r_{f}$ and $p_{g,f} \dashv r_{g,f}$ to obtain another $w$-reflection $p_{f}\thing p_{g,f} \dashv r_{g,f}\thing r_{f}$, so that the outer span becomes a $w$-span $\bar{W}_{g}\bar{W}_{f}:A \hto C$.\newline{}
Let us consider the relationship between $\bar{W}_{g}\bar{W}_{f}$ and $\bar{W}_{gf}$.  By the universal property of $\bar{W}_{gf}$ the $\overline{w}$-cone left below 
$$\xy
(15,15)*+{\bar{W}_{g}\bar{W}_{f}}="f0"; 
(0,0)*+{\bar{W}_{f}}="a0"; (-15,-15)*+{A}="b0";(15,-15)*+{B}="c0";
{\ar_{p_{f}} "a0"; "b0"}; 
{\ar^{q_{f}} "a0"; "c0"}; 
{\ar@{~>}_{f} "b0"; "c0"}; 
{\ar@{=>}_{\lambda_{f}}(0,-5)*+{};(0,-11)*+{}};
(30,0)*+{\bar{W}_{g}}="d0";(45,-15)*+{C}="e0";
{\ar_{p_{g}} "d0"; "c0"}; 
{\ar^{q_{g}} "d0"; "e0"}; 
{\ar@{~>}_{g} "c0"; "e0"}; 
{\ar@{=>}_{\lambda_{g}}(30,-5)*+{};(30,-11)*+{}};
{\ar_{p_{g,f}} "f0"; "a0"}; 
{\ar^{q_{g,f}} "f0"; "d0"}; 
\endxy
\hspace{0.3cm}
\textnormal{=}
\hspace{0.3cm}
\xy
(15,15)*+{\bar{W}_{g}\bar{W}_{f}}="f0"; 
(0,0)*+{\bar{W}_{f}}="a0"; (-15,-15)*+{A}="b0";(15,-15)*+{B}="c0";
{\ar_{p_{f}} "a0"; "b0"}; 
{\ar@{~>}_{f} "b0"; "c0"}; 
{\ar@{=>}_{\lambda_{gf}}(15,-5)*+{};(15,-11)*+{}};
(30,0)*+{\bar{W}_{g}}="d0";(45,-15)*+{C}="e0";
{\ar^{q_{g}} "d0"; "e0"}; 
{\ar@{~>}_{g} "c0"; "e0"}; 
{\ar_{p_{g,f}} "f0"; "a0"}; 
{\ar^{q_{g,f}} "f0"; "d0"}; 
(15,0)*+{\bar{W}_{gf}}="g0"; 
{\ar|{k_{g,f}} "f0"; "g0"}; 
{\ar^{p_{gf}} "g0"; "b0"}; 
{\ar_{q_{gf}} "g0"; "e0"}; 
\endxy
$$
induces a unique tight arrow $k_{g,f}:\bar{W}_{g}\bar{W}_{f} \to \bar{W}_{gf}$ satisfying the equations
\begin{equation}\label{eq:comp}
p_{gf}\thing k_{g,f}=p_{f}\thing p_{g,f}\textnormal{, }q_{gf}\thing k_{g,f}=q_{g}\thing q_{g,f}
\textnormal{ and } 
\lambda_{gf}\thing k_{g,f}=(g\thing \lambda_{f}\thing p_{g,f})\circ (\lambda_{g}\thing q_{g,f})
\end{equation}
equally expressed in the equality of pasting diagrams above.  Furthermore
\begin{Lemma}\label{prop:spanmap}
Let $w \in \{l,p\}$.  In the span map
$$\xy
(0,0)*+{A}="00";(0,12)*+{A}="01";(25,12)*+{\bar{W}_{g}\bar{W}_{f}}="11";(25,0)*+{\bar{W}_{gf}}="10";(50,0)*+{B}="20";(50,12)*+{B}="21";
{\ar_{p_{f}\thing p_{g,f}} "11"; "01"};
{\ar_{p_{gf}} "10"; "00"};
{\ar^{q_{g}\thing q_{g,f}} "11"; "21"}; 
{\ar^{q_{gf}} "10"; "20"}; 
{\ar^{k_{g,f}} "11"; "10"}; 
{\ar^{1} "01"; "00"}; 
{\ar^{1} "21"; "20"}; 
\endxy
$$
the commuting square $(k_{g,f},1):p_{f}\thing p_{g,f} \to p_{gf}$ is a morphism of $w$-reflections.
\end{Lemma}
\begin{proof}
To show that $(k_{g,f},1):p_{f}\thing p_{g,f} \to p_{gf}$ is a morphism of $w$-reflections it suffices, by Lemma~\ref{prop:morphAdj}, to show that the mate of this square is an identity.  Because the $w$-reflection $p_{f}\thing p_{g,f} \dashv r_{g,f}\thing r_{f}$ has identity counit the mate of $(k_{g,f},1)$ is simply
$$\xy
(-20,-15)*+{A}="e0";(40,0)*+{\bar{W}_{gf}}="f0";
(0,0)*+{\bar{W}_{g}\bar{W}_{f}}="a0"; (20,0)*+{\bar{W}_{gf}}="b0";(0,-15)*+{A}="c0";(20,-15)*+{A}="d0";
{\ar^{k_{g,f}} "a0"; "b0"}; 
{\ar_{1} "c0"; "d0"}; 
{\ar|{p_{f}\thing p_{g,f}} "a0"; "c0"};
{\ar_{p_{gf}} "b0"; "d0"};  
{\ar@{~>}^{r_{g,f}\thing r_{f}} "e0"; "a0"}; 
{\ar_{1} "e0"; "c0"};   
{\ar^{1} "b0"; "f0"}; 
{\ar@{~>}_{r_{gf}} "d0"; "f0"};   
{\ar@{=>}_{\eta_{gf}}(24,-1)*+{};(28,-7)*+{}};
\endxy$$
Now the 2-dimensional universal property of $\bar{W}_{gf}$ implies that the projections $p_{gf}$ and $q_{gf}$ jointly reflect identity 2-cells: see Lemma 3.1 of \cite{Lack2002Limits} in the case of the colax limit.  Therefore it suffices to show that the composite of the above 2-cell with both $p_{gf}$ and $q_{gf}$ yields an identity.
One of the triangle equations for $p_{gf}\dashv r_{gf}$ gives $p_{gf}\thing \eta_{gf}$=1; thus it remains to show $q_{gf}\thing \eta_{gf}\thing k_{g,f}\thing r_{g,f}\thing r_{f}=1$.  By  ~\eqref{eq:adjR} this equals $\lambda_{gf}\thing k_{g,f} \thing r_{g,f}\thing r_{f}$ and by definition of $k_{g,f}$ (as in ~\eqref{eq:comp})  we have $\lambda_{gf}\thing k_{g,f}=(g\thing \lambda_{f}\thing p_{g,f})\circ (\lambda_{g}\thing q_{g,f})$.  Therefore it suffices to show that the 2-cells $(g\thing \lambda_{f}\thing p_{g,f})\thing r_{g,f}\thing r_{f}$ and $(\lambda_{g}\thing q_{g,f})\thing r_{g,f}\thing r_{f}$ are identities separately.  
With regards the former we have that $\lambda_{f}\thing p_{g,f}\thing r_{g,f}\thing r_{f}=\lambda_{f}\thing r_{f}=1$ where we first use that $p_{g,f}\thing r_{g,f}=1$ and then  ~\eqref{eq:conR}; for the other composite we have $(\lambda_{g}\thing q_{g,f})\thing r_{g,f}\thing r_{f}=\lambda_{g}\thing r_{g}\thing q_{f}\thing r_{f}=1$.  The first equation holds by ~\eqref{eq:spancomp} and the second equation by  ~\eqref{eq:conR}.
\end{proof}

\begin{Remark}
Given any \f F-category \g A one can define a bicategory $Span_{w}(\g A)$ of $w$-spans in \g A.  Their composition extends that described for $w$-spans of the form $\bar{W}_{f}:A \hto B$ above.  To ensure composites exist one allows only those $w$-spans whose left legs admit tight pullbacks along arbitrary tight maps.  Now when \g A admits $w$-limits of loose morphisms the assigment of $\bar{W}_{f}:A \hto B$ to $f$ can be extended to an identity on objects lax functor from the underlying category of $\f A_{\lambda}$ to $Span_{w}(\g A)$ with the span map $k_{g,f}:\bar{W}_{g}\bar{W}_{f} \to \bar{W}_{gf}$ describing one of the coherence constraints.
\end{Remark}

\subsection{Representing 2-cells via span transformations}
Lastly we consider how to represent a 2-cell $\alpha:f \Rightarrow g$ by a span map $\bar{W}_{f} \to \bar{W}_{g}$.  Here the cases $w=l$ and $w=p$ diverge.
\subsubsection{The case w=l}
Suppose that \g A admits colax limits of loose morphisms.  Given $\alpha:f \Rightarrow g$ the colax cone left below
$$\xy
(0,10)*+{C_{f}}="a0";
 (-20,-10)*+{A}="b0";(20,-10)*+{B}="c0";
{\ar_{p_{f}} "a0"; "b0"}; 
{\ar^{q_{f}} "a0"; "c0"}; 
{\ar@/^1pc/@{~>}^{f} "b0"; "c0"}; 
{\ar@{=>}_{\alpha}(0,-7)*+{};(0,-13)*+{}};
{\ar@/_1pc/@{~>}_{g} "b0"; "c0"}; 
{\ar@{=>}_{\lambda_{f}}(0,3)*+{};(0,-2)*+{}};
\endxy
\hspace{1cm}
\xy
\textnormal{=}
\endxy
\hspace{1cm}
\xy
(0,15)*+{C_{f}}="d0";
(0,0)*+{C_{g}}="a0"; (-20,-15)*+{A}="b0";(20,-15)*+{B}="c0";
{\ar_{c_{\alpha}} "d0"; "a0"}; 
{\ar@/_1.5pc/_{p_{f}} "d0"; "b0"}; 
{\ar@/^1.5pc/^{q_{f}} "d0"; "c0"}; 
{\ar_{p_{g}} "a0"; "b0"}; 
{\ar^{q_{g}} "a0"; "c0"}; 
{\ar@{~>}_{g} "b0"; "c0"}; 
{\ar@{=>}_{\lambda_{g}}(0,-5)*+{};(0,-11)*+{}};
\endxy$$
induces a unique \emph{tight} arrow $c_{\alpha}:C_{f} \to C_{g}$ satisfying
\begin{equation}\label{eq:l2}
p_{f}=c_{\alpha}\thing p_{g}, \textnormal{ } q_{f}=c_{\alpha}\thing q_{g} \textnormal{ and }(\alpha \thing p_{f})\circ \lambda_{f}=\lambda_{g}\thing c_{\alpha}
\end{equation}
In particular the 2-cell $\alpha$ is represented by a span map $c_{\alpha}:C_{f} \to C_{g}$ as below.
$$\xy
(0,0)*+{A}="00";(0,12)*+{A}="01";(20,12)*+{C_{f}}="11";(20,0)*+{C_{g}}="10";(40,0)*+{B}="20";(40,12)*+{B}="21";
{\ar_{p_{f}} "11"; "01"};
{\ar_{p_{g}} "10"; "00"};
{\ar^{q_{f}} "11"; "21"}; 
{\ar^{q_{g}} "10"; "20"}; 
{\ar^{c_{\alpha}} "11"; "10"}; 
{\ar^{1} "01"; "00"}; 
{\ar^{1} "21"; "20"}; 
\endxy
$$
\begin{Lemma}\label{prop:l2}
Let $c_{\alpha}:C_{f} \to C_{g}$ be as above and let $m_{\alpha}$ denote the mate of the square $(c_{\alpha},1):p_{f} \to p_{g}$ through the adjunctions $p_{f} \dashv r_{f}$ and $p_{g} \dashv r_{g}$.  
$$\xy
(0,0)*+{A}="00";(0,12)*+{A}="01";(20,12)*+{C_{f}}="11";(40,0)*+{B}="20";(20,0)*+{C_{g}}="10";(40,12)*+{B}="21";
{\ar@{~>}^{r_{f}} "01"; "11"};
{\ar@{~>}_{r_{g}} "00"; "10"};
{\ar^{q_{f}} "11"; "21"}; 
{\ar_{q_{g}} "10"; "20"}; 
{\ar^{c_{\alpha}} "11"; "10"};
{\ar@{=>}^{m_{\alpha}}(8,9)*+{};(8,4)*+{}};
{\ar_{1} "01"; "00"}; 
{\ar^{1} "21"; "20"};  
\endxy$$
The composite 2-cell $q_{g}\thing m_{\alpha}$ equals $\alpha$.
\end{Lemma}
\begin{proof}
Because the counit of $p_{f} \dashv r_{f}$ is an identity the mate $m_{\alpha}$ is simply given by
$$\xy
(00,0)*+{A}="11";(30,0)*+{C_{f}}="31"; (00,-25)*+{A}="12";(30,-25)*+{C_{g}}="32";
(00,-15)*+{A}="a";(30,-10)*+{C_{g}}="b";
{\ar^{c_{\alpha}} "31"; "b"}; 
{\ar@{~>}_{r_{g}} "12"; "32"}; 
{\ar_{1} "a"; "12"}; 
{\ar@{~>}^{r_{f}} "11"; "31"}; 
{\ar_{1} "11"; "a"}; 
{\ar^{p_{f}} "31"; "a"}; 
{\ar_{p_{g}} "b"; "12"}; 
{\ar^{1} "b"; "32"}; 
{\ar@{=>}_{\eta_{g}}(24,-19)*+{};(18,-22)*+{}};
\endxy$$
Therefore $q_{g}\thing m_{\alpha}=q_{g}\thing \eta_{g}\thing c_{\alpha}\thing r_{f}=\lambda_{g}\thing c_{\alpha}\thing r_{f}=(\alpha \thing  p_{f}\thing r_{f})\circ (\lambda_{f}\thing r_{f})=\alpha \circ 1=\alpha$ where the second, third and fourth equalities use ~\eqref{eq:adjR}, ~\eqref{eq:l2} and ~\eqref{eq:adjR} respectively.
\end{proof}
\subsubsection{The case w=p}
Suppose that \g A admits pseudolimits of loose morphisms.  If $\alpha:f \Rightarrow g$ is \emph{invertible}
then we may construct a map $P_{f} \to P_{g}$ in essentially the same way as for colax limits of loose morphisms.  However this approach does not work in general.  The following lemma describes a representation that works for non-invertible 2-cells.  It is based upon the notion of a \emph{transformation} of \emph{anafunctors} \cite{Makkai1996Avoiding} -- see also \cite{Roberts2012Internal} and \cite{Bartels2006Higher}.  Anafunctors can be viewed as spans of categories and functors in which the left leg is a surjective on objects equivalence.  If we take $\g A= \Cat$ and functors $F,G:A \to B$ then the associated spans $P_{F},P_{G}:A \hto B$ are anafunctors.  In this setting the 2-cell $\rho_{\alpha}$ described below specifies precisely a transformation between the anafunctors $P_{F}$ and $P_{G}$.

\begin{Lemma}\label{prop:p2}
Given $\alpha:f \Rightarrow g$ consider the tight pullback left below (existing by Lemma~\ref{prop:pullbacks1})
$$\xy
(0,-5)*+{K_{f,g}}="a0"; (20,-5)*+{P_{f}}="b0";(0,-25)*+{P_{g}}="c0";(20,-25)*+{A}="d0";
{\ar^{s_{f,g}} "a0"; "b0"}; 
{\ar_{p_{g}} "c0"; "d0"}; 
{\ar_{t_{f,g}} "a0"; "c0"};
{\ar^{p_{f}} "b0"; "d0"};  
{\ar^{u_{f,g}} "a0"; "d0"};  
\endxy
\hspace{1cm}
\xy
(-5,0)*+{A}="00"; (15,-15)*+{K_{f,g}}="1-1";(30,-15)*+{P_{f}}="2-1";(15,-30)*+{P_{g}}="1-2";(30,-30)*+{A}="2-2";
{\ar@{~>}^{v_{f,g}} "00"; "1-1"}; 
{\ar@/^1pc/@{~>}^{r_{f}} "00"; "2-1"}; 
{\ar@/_1pc/@{~>}_{r_{g}} "00"; "1-2"}; 
{\ar^{s_{f,g}} "1-1"; "2-1"}; 
{\ar^{q_{f}} "2-1"; "2-2"};  
{\ar_{t_{f,g}} "1-1"; "1-2"};
{\ar_{q_{g}} "1-2"; "2-2"}; 
{\ar@{=>}^{\rho_{\alpha}}(25,-20)*+{};(20,-25)*+{}} 
\endxy$$
with diagonal denoted $u_{f,g}$.
\begin{enumerate}
\item There is a unique $p$-reflection $u_{f,g} \dashv v_{f,g}$ such that $(s_{f,g},1):u_{f,g} \to p_{f}$ and $(t_{f,g},1):u_{f,g} \to p_{g}$ are morphisms of $p$-reflections.  If $f$ and $g$ are tight so too is $v_{f,g}$.
\item There exists a unique 2-cell $\rho_{\alpha}:q_{f}\thing s_{f,g} \Rightarrow q_{g}\thing t_{f,g}$ such that $\rho_{\alpha}.v_{f,g}=\alpha$.
\end{enumerate}
\end{Lemma}
\begin{proof}
\begin{enumerate}
\item
Consider the diagram
$$\xy
(-5,0)*+{A}="00"; (15,-15)*+{K_{f,g}}="1-1";(30,-15)*+{P_{f}}="2-1";(15,-30)*+{P_{g}}="1-2";(30,-30)*+{A}="2-2";
{\ar@{~>}^{v_{f,g}} "00"; "1-1"}; 
{\ar@/^0.7pc/@{~>}^{r_{f}} "00"; "2-1"}; 
{\ar@/_0.5pc/@{~>}^{r_{g}} "00"; "1-2"}; 
{\ar_{s_{f,g}} "1-1"; "2-1"}; 
{\ar^{p_{f}} "2-1"; "2-2"};  
{\ar^{t_{f,g}} "1-1"; "1-2"};
{\ar_{p_{g}} "1-2"; "2-2"}; 
\endxy$$
in which $v_{f,g}$ is the unique loose map satisfying $s_{f,g}\thing v_{f,g}=r_{f}$ and $t_{f,g}\thing v_{f,g}=r_{g}$.  It follows that $u_{f,g}\thing v_{f,g}=1$.  To give an invertible 2-cell $\theta_{f,g}:1 \cong v_{f,g}\thing u_{f,g}$ is, by the universal property of $K_{f,g}$, equally to give invertible 2-cells $\theta_{1}:s_{f,g} \cong s_{f,g}\thing v_{f,g}\thing u_{f,g}$ and $\theta_{2}:t_{f,g} \cong t_{f,g}\thing v_{f,g}\thing u_{f,g}$ satisfying $p_{f}\thing \theta_{1}=p_{g}\thing \theta_{2}$.  We set $\theta_{1}=\eta_{f}\thing s_{f,g}$ and $\theta_{2}=\eta_{g}\thing t_{f,g}$ noting that $p_{f}\thing \eta_{f}\thing s_{f,g}=1=p_{g}\thing \eta_{g}\thing s_{f,g}$.  The triangle equations for the $p$-reflection follow using the universal property of the pullback $K_{f,g}$ and that $s_{f,g}:u_{f,g} \to p_{f}$ and $t_{f,g}:u_{f,g} \to p_{g}$ are morphisms of $p$-reflections follows from the construction of $u_{f,g}$ and $\theta_{f,g}$.\newline{}
If $f$ and $g$ are tight so too, by ~\eqref{eq:rTight}, are $r_{f}$ and $r_{g}$.  By the universal property of the tight pullback $K_{f,g}$ it then follows that $v_{f,g}$ is tight.
\item 
From the first part we have $s_{f,g}\thing v_{f,g}=r_{f}$ and $t_{f,g}\thing v_{f,g}=r_{g}$ as in the two triangles above.  Now by ~\eqref{eq:adjR}  we have $q_{f}\thing r_{f}=f$ and $q_{g}\thing r_{g}=g$.  Therefore we can write $\alpha:(q_{f}\thing s_{f,g})\thing v_{f,g} \Rightarrow (q_{g}\thing t_{f,g})\thing v_{f,g}$.  Since $v_{f,g}:A \rightsquigarrow K_{f,g}$ is an equivalence in $\f A_{\lambda}$ the functor $\f A_{\lambda}(v_{f,g},A):\f A_{\lambda}(K_{f,g},A) \to \f A_{\lambda}(A,A)$ is an equivalence of categories -- using its fully faithfulness we obtain $\rho_{\alpha}:q_{f}\thing s_{f,g} \Rightarrow q_{g}\thing t_{f,g}$.
\qedhere
\end{enumerate}
\end{proof}


\section{Orthogonality}
The following theorem is the crucial result of the paper.  The monadicity theorems of Section 6 follow easily from it.  We note that both this theorem and the corollary that follows it are independent of the formalism of 2-monads.

\begin{Theorem}\label{thm:orthogonality}
Let $w \in \{l,p,c\}$.  Consider an \f F-category $\g A$ with $\bar{w}$-limits of loose morphisms.  Then the inclusion of tight morphisms $j:\f A_{\tau} \to \g A$ is orthogonal to each $w$-doctrinal \f F-functor.
\end{Theorem}
\begin{proof}
Consider a commuting square in \fcat
$$\xy
(0,0)*+{\f A_{\tau}}="a0"; (20,0)*+{\g A}="b0";(0,-10)*+{\g B}="c0";(20,-10)*+{\g C}="d0";
{\ar^{j} "a0"; "b0"}; 
{\ar_{R} "a0"; "c0"}; 
{\ar^{S} "b0"; "d0"}; 
{\ar_{H} "c0"; "d0"}; 
{\ar@{.>}|{K} "b0"; "c0"}; 
\endxy$$
in which $H$ is $w$-doctrinal.  We must show there exists a unique diagonal filler $K$.  We begin by noting that the cases $w=c$ and $w=l$ are dual since \g A satisfies the $c$-criteria of the theorem just when $\g A^{co}$ satisfies the $l$-criteria with, equally, $H$ $c$-doctrinal just when $H^{co}$ is $l$-doctrinal.  Therefore it will suffice to suppose $w \in \{l,p\}$.
\begin{enumerate}
\item
Before constructing the diagonal we fix some notation and make some observations about lifted adjunctions that will be repeatedly used in what follows.  Given a $w$-reflection $(1, f \dashv g,\eta) \in \g A$ we obtain a $w$-reflection $(1, Sf \dashv Sg,S\eta)$ in \g C with $Sf=Hf)$ since $f$ is tight.  As $H$ is $w$-doctrinal this lifts uniquely along $H$ to a $w$-reflection in \g B which we denote by $(1,Rf \dashv \overline{g},\overline{\eta})$.\newline{}
Next consider $w$-reflections $(1, f_{1} \dashv g_{1},\eta_{1})$ and $(1, f_{2} \dashv g_{2},\eta_{2})$ in \g A and a tight commuting square $(r,s):f_{1} \to f_{2}$ in \g A as left below
$$\hspace{0.2cm}
\xy
(0,0)*+{A}="a0"; (20,0)*+{C}="b0";(0,-10)*+{B}="c0";(20,-10)*+{D}="d0";
{\ar^{r} "a0"; "b0"}; 
{\ar_{s} "c0"; "d0"}; 
{\ar_{f_{1}} "a0"; "c0"};
{\ar^{f_{2}} "b0"; "d0"};  
\endxy
\hspace{1cm}
\xy
(0,0)*+{RA}="a0"; (20,0)*+{RC}="b0";(0,-10)*+{RB}="c0";(20,-10)*+{RD}="d0";
{\ar^{Rr} "a0"; "b0"}; 
{\ar_{Rs} "c0"; "d0"}; 
{\ar_{Rf_{1}} "a0"; "c0"};
{\ar^{Rf_{2}} "b0"; "d0"};  
\endxy
\hspace{1cm}
\xy
(0,0)*+{RB}="a0"; (20,0)*+{RD}="b0";(0,-10)*+{RA}="c0";(20,-10)*+{RC}="d0";
{\ar@{~>}_{\overline{g_{1}}} "a0"; "c0"}; 
{\ar@{~>}^{\overline{g_{2}}} "b0"; "d0"}; 
{\ar_{Rr} "c0"; "d0"};
{\ar^{Rs} "a0"; "b0"};
\endxy
$$
Since \f F-functors preserves morphisms of $w$-reflections the square $(Sr,Ss):Sf_{1} \to Sf_{2}$ is one in \g C.  Now the tight commutative square $(Rr,Rs):Rf_{1} \to Rf_{2}$ has image under $H$ the morphism of $w$-reflections $(Sr,Ss):Sf_{1} \to Sf_{2}$.  Because $H$ is $w$-doctrinal it follows that $(Rr,Rs):Rf_{1} \to Rf_{2}$ is a morphism between the lifted $w$-reflections in \g B.   In particular we obtain a commuting square of right adjoints $(Rs,Rr):\overline{g_{1}} \to \overline{g_{2}}$.\\
Finally consider a composable pair of tight left adjoints $f_{1}:A \to B$ and $f_{2}:B \to C$ with associated $w$-reflections $(1, f_{1} \dashv g_{1},\eta_{1})$ and $(1, f_{2} \dashv g_{2},\eta_{2})$ in \g A.  We can form the $w$-reflections $(1,Rf_{1} \dashv \overline{g_{1}},\overline{\eta_{1}})$ and $(1,Rf_{2} \dashv \overline{g_{2}},\overline{\eta_{2}})$ in \g B and compose these to obtain a further $w$-reflection $f_{2}\thing f_{1} \dashv \overline{g_{1}}\thing \overline{g_{2}}$ in \g B.  It is clear that this is a lifting of the $w$-reflection $S(f_{2}\thing f_{1}) \dashv S(g_{1}\thing g_{2})$ so that, by uniqueness of liftings, the $w$-reflections $Rf_{2}\thing Rf_{1} \dashv \overline{g_{1}}\thing \overline{g_{2}}$ and $R(f_{2}\thing f_{1}) \dashv \overline{g_{1}\thing g_{2}}$ coincide.
\item
Now to begin constructing $K$ observe that for the left triangle to commute we must define $KA=RA$ for each $A \in \g A$.
\item
Given $f:A \rightsquigarrow B \in \g A$ we recall from \eqref{eq:conR} its factorisation as $q_{f}\thing r_{f}:A \rightsquigarrow C_{f} \to B$ where $(1,p_{f} \dashv r_{f},\eta_{f})$.  Since $p_{f}$ is tight we have the lifted $w$-reflection $(1,Rp_{f} \dashv \overline{r_{f}},\overline{\eta}_{f})$ in \g B living over $(1,Sp_{f} \dashv Sr_{f},S\eta_{f})$.  We define $Kf:RA \rightsquigarrow RB$ as the composite $Rq_{f} \thing \overline{r}_{f}:RA \rightsquigarrow R\bar{W}_{f} \to RB$.  
\item Observe that $HKf=HRq_{f} \thing H\overline{r}_{f}= Sq_{f} \thing Sr_{f}=S(q_{f} \thing r_{f})=Sf$ as required.  To see that $K$ extends $R$ observe that if $f:A \to B$ is tight then, by \eqref{eq:rTight}, $r_{f}$ is tight too, so that we have a $w$-reflection $(1,Rp_{f} \dashv Rr_{f},R\eta_{f})$ living over $(1,HRp_{f} \dashv Sr_{f},\eta_{f})$; thus $(1,Rp_{f} \dashv Rr_{f},R\eta_{f})=(1,Rp_{f} \dashv \overline{r_{f}},\overline{\eta_{f}})$ so that $Kf=Rq_{f} \thing Rr_{f} = Rf$.  Thus $K$ coincides with $R$ on tight morphisms.
\item 
As $K$ agrees with $R$ on tight morphisms we already know that it preserves identity 1-cells.  To see that it preserves composition of 1-cells it will suffice to show that all of the regions of the diagram below commute.
$$
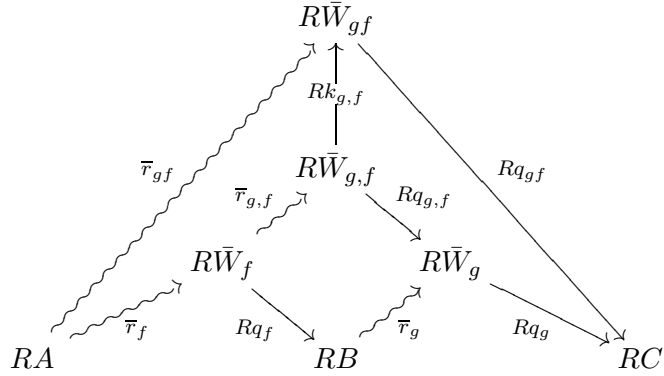
\xy
(15,30)*+{R\bar{W}_{gf}}="f0"; 
(-25,-15)*+{RA}="b0";(15,-15)*+{RB}="c0";
(0,-2.5)*+{R\bar{W}_{f}}="a0";
{\ar_{Rq_{f}} "a0"; "c0"}; 
(30,-2.5)*+{R\bar{W}_{g}}="d0";(55,-15)*+{RC}="e0";
{\ar_{Rq_{g}} "d0"; "e0"}; 
{\ar@{~>}_{\overline{r}_{g}} "c0"; "d0"}; 
{\ar^{Rq_{gf}} "f0"; "e0"}; 
(15,10)*+{R\bar{W}_{g,f}}="g0"; 
{\ar|{Rk_{g,f}} "g0"; "f0"}; 
{\ar@{~>}_{\overline{r}_{f}} "b0"; "a0"}; 
{\ar@{~>}^{\overline{r}_{gf}} "b0"; "f0"}; 
{\ar@{~>}^{\overline{r}_{g,f}} "a0"; "g0"}; 
{\ar^{Rq_{g,f}} "g0"; "d0"}; 
\endxy
$$
The rightmost quadrilateral certainly commutes as it is the image of a commutative diagram in $\f A_{\tau}$, from \eqref{eq:comp}, under $R$.  To see that the central square commutes recall from \eqref{eq:spancomp} the morphism of $w$-reflections $(q_{g,f},q_{f}):p_{g,f} \to p_{g}$ in \g A.  By Part 1 $(Rq_{g,f},Rq_{f}):Rp_{g,f} \to p_{g,f}$ is a morphism of $w$-reflections in \g B: now commutativity of the central square simply asserts commutativity with the right adjoints.  With regards the leftmost quadrilateral recall from Lemma~\ref{prop:spanmap} the morphism of $w$-reflections in \g A given by $(k_{g,f},1):p_{f}\thing p_{g,f} \to p_{gf}$.  By Part 1 we know that $(Rk_{g,f},1):R(p_{f}\thing  p_{g,f}) \to Rp_{gf}$ is a morphism of the lifted $w$-reflections so that, in particular, we have a commuting square of right adjoints $(1,Rk_{g,f}):\overline{r_{g,f}\thing r_{f}} \to \overline{r}_{gf}$ in \g B.  Again by Part 1 we have $\overline{r_{g,f}\thing r_{f}} = \overline{r}_{g,f} \overline{r}_{f}$ and therefore the desired commutativity.
\item We define $K$ differently on 2-cells depending upon whether $w=l$ or $w=p$.  Consider $\alpha:f \Rightarrow g$ and the case $w=l$.  The commuting square $(c_{\alpha},1):p_{f} \to p_{g}$ of Lemma~\ref{prop:l2} has image $(Rc_{\alpha},1):Rp_{f} \to Rp_{g}$.    We denote the mate of this square through the adjunctions $Rp_{f} \dashv \overline{r}_{f}$ and $Rp_{g} \dashv \overline{r}_{g}$ by $\overline{m}_{\alpha}:Rc_{\alpha}\thing \overline{r}_{f} \Rightarrow \overline{r}_{g}$.
$$\xy
(-10,0)*+{}="-10";
(0,0)*+{RA}="00";(0,12)*+{RA}="01";(20,12)*+{RC_{f}}="11";(40,0)*+{RB}="20";(20,0)*+{RC_{g}}="10";(40,12)*+{B}="21";
{\ar@{~>}^{\overline{r}_{f}} "01"; "11"};
{\ar@{~>}_{\overline{r}_{g}} "00"; "10"};
{\ar^{Rq_{f}} "11"; "21"}; 
{\ar_{Rq_{g}} "10"; "20"}; 
{\ar^{Rc_{\alpha}} "11"; "10"};
{\ar@{=>}^{\overline{m}_{\alpha}}(10,9)*+{};(10,4)*+{}};
{\ar_{1} "01"; "00"}; 
{\ar^{1} "21"; "20"};  
\endxy
$$
Now we set $K\alpha=Rq_{g}\thing \overline{m}_{\alpha}$ as depicted above.  Note that we have $\overline{m}_{\alpha}=\overline{\eta}_{g}\thing Rc_{\alpha}\thing \overline{r}_{f}$ since the counit of the $l$-reflection $Rp_{f} \dashv \overline{r}_{f}$ is an identity.\newline{}
Let us show that $HK\alpha=S\alpha$.  We have from Lemma~\ref{prop:l2} the decomposition in \g A of $\alpha$ as $q_{g}\thing m_{\alpha}= q_{g}\thing \eta_{g}\thing c_{\alpha}\thing r_{f}$ and also have $K\alpha=Rq_{g}\thing \overline{\eta}_{g}\thing Rc_{\alpha}\thing \overline{r}_{f}$ from above.  Thus $HK\alpha=Sq_{g}\thing S\eta_{g}\thing Sc_{\alpha}\thing Sr_{f}=S(q_{g}\thing \eta_{g}\thing c_{\alpha}\thing r_{f})=S\alpha$.\newline{}
To see that $K$ extends $R$ observe that if both $f$ and $g$ are tight then by Part 4 we have $(1,Rp_{f} \dashv \overline{r}_{f},\overline{\eta}_{f})=(1,Rp_{f} \dashv Rr_{f},R\eta_{f})$ and likewise for $g$.  Using the same decomposition of $\alpha$ we have $R\alpha=Rq_{g}\thing R\eta_{g}\thing Rc_{\alpha}\thing Rr_{f}=K\alpha$.
\newline{}
For $w=p$ consider the $p$-reflection $u_{f,g} \dashv v_{f,g}$ and the morphisms of $p$-reflections $(s_{f,g},1):u_{f,g} \to p_{f}$ and $(t_{f,g},1):u_{f,g} \to p_{g}$ of Lemma~\ref{prop:p2}.  By Part 1 their images $(Rs_{f,g},1):Ru_{f,g} \to Rp_{f}$ and $(Rt_{f,g},1):Ru_{f,g} \to Rp_{g}$ are morphisms of $p$-reflections in \g B so that the two triangles in the diagram below
$$\xy
(-5,0)*+{RA}="00"; (15,-15)*+{RK_{f,g}}="1-1";(35,-15)*+{RP_{f}}="2-1";(15,-30)*+{RP_{g}}="1-2";(35,-30)*+{A}="2-2";
{\ar@{~>}^{\overline{v}_{f,g}} "00"; "1-1"}; 
{\ar@/^1pc/@{~>}^{\overline{r}_{f}} "00"; "2-1"}; 
{\ar@/_1pc/@{~>}_{\overline{r}_{g}} "00"; "1-2"}; 
{\ar^{Rs_{f,g}} "1-1"; "2-1"}; 
{\ar^{Rq_{f}} "2-1"; "2-2"};  
{\ar_<<<{Rt_{f,g}} "1-1"; "1-2"};
{\ar_{Rq_{g}} "1-2"; "2-2"}; 
{\ar@{=>}^{R\rho_{\alpha}}(27,-20)*+{};(22,-25)*+{}} 
\endxy$$ commute.  We set $K\alpha =R\rho_{\alpha}\thing \overline{v}_{f,g}$.\newline{}
By Lemma~\ref{prop:p2} we have that $\alpha=\rho_{\alpha}\thing v_{f,g}$.  Therefore $S\alpha=S\rho_{\alpha}\thing Sv_{f,g}=HR\rho_{\alpha} \thing H\overline{v}_{f,g}=HK\alpha$.   If now both $f$ and $g$ are tight then so is $v_{f,g}$, again by Lemma ~\ref{prop:p2}.    It follows that the $p$-reflections $Ru_{f,g} \dashv \overline{v}_{f,g}$ and $Ru_{f,g} \dashv Rv_{f,g}$ coincide, and therefore that $R\alpha=R\rho_{\alpha}\thing Rv_{f,g}=K\alpha$.
\item Treating the cases $w=l$ and $w=p$ together again observe that  because $H$ is locally faithful the functoriality of $K$ on 2-cells trivally follows from the functoriality of $HK=S$ on 2-cells.  Thus $K$ is an \f F-functor.
\item For uniqueness let us observe that the definition of the filler is forced upon us by the constraints.  Consider any diagonal filler $L$.  We must have $Lf=L(q_{f} \thing r_{f})=Lq_{f} \thing Lr_{f}= Rq_{f} \thing Lr_{f}$.  Since $HL=S$ we would need that the image of the adjunction $(1,p_{f} \dashv r_{f},\eta_{f})$ under $L$ were a lifting of its image $(1,Sp_{f} \dashv Sr_{f},S\eta_{f})$ under $HL$.  But by uniqueness of liftings we would then have $(1,Lp_{f} \dashv Lr_{f},L\eta_{f})=(1,Rp_{f} \dashv \overline{r_{f}},\overline{\eta_{f}})$ so that $Lf=Rq_{f}\thing \overline{r}_{f}$.  Now for $w=l$ we have, by Lemma~\ref{prop:l2}, the decomposition of $\alpha:f \Rightarrow g$ as $q_{g}\thing m_{\alpha}=q_{g}\thing \eta_{g} \thing c_{\alpha} \thing r_{f}$ and so must have $L\alpha=Rq_{g}\thing \overline{\eta}_{g}\thing Rc_{\alpha}\thing \overline{r}_{f}$.  The argument when $w=p$ is similar but uses the decomposition $\alpha=\rho_{\alpha}\thing v_{f,g}$ of Lemma~\ref{prop:p2} instead.
\qedhere
\end{enumerate}
\end{proof}
Let us note that Theorem~\ref{thm:orthogonality} remains true even if we remove the assumption that $w$-doctrinal \f F-functors are locally faithful.  The only place that we used this assumption was in establishing the functoriality of the diagonal $K$ on 2-cells.  This can alternatively be established by carefully analysing the functoriality of the assignment in Section 4.3 which represents a 2-cell by a span transformation.  However the proof becomes significantly longer and more technical, very significantly when $w=p$.  Since local faithfulness is true of those \f F-functors $U:\fTAlg_{w} \to \f C$ that we seek to characterise (see Theorem~\ref{thm:monadicity}) and because the property is easily verified in practice, we have chosen to include it in our definition of $w$-doctrinal \f F-functor.

As an immediate consequence of Theorem~\ref{thm:orthogonality} we have:
\begin{Corollary}\label{thm:decomposition}
Let $w \in \{l,p,c\}$.  Consider an \f F-functor $H:\g A \to \f B$ to a 2-category and suppose that \g A has $\bar{w}$-limits of loose morphisms and that $H$ is $w$-doctrinal.  Then the decomposition in \fcat
$$\xy
(0,0)*+{\f A_{\tau}}="00";(15,0)*+{\g A}="10";(30,0)*+{\f B}="20";
{\ar^{j} "00"; "10"}; 
{\ar^<<<<<{H} "10"; "20"}; 
\endxy$$
of the 2-functor $H_{\tau}:\f A_{\tau} \to \f B$ is an orthogonal $(^{\bot}\wdoct,\wdoct)$-decomposition.
\end{Corollary}
Since orthogonal decompositions are essentially unique, this asserts that an \f F-category satisfying some completeness properties and sitting over a base 2-category in a certain way -- such as $\fMonCat_{w}$ or $\fTAlg_{w}$ over $\Cat$ -- is \emph{uniquely determined} by how its tight part sits over the base.  This is the core idea behind our monadicity results of Section 6.
\subsection{A note on alternative hypotheses}
There are other hypotheses upon \g A which ensure that $j:\f A_{\tau} \to \g A$ belongs to $^{\bot}\wdoct$.  Let us focus on the lax case.  Say that \g A admits \emph{loose morphism classifiers} if the inclusion $j:\f A_{\tau} \to \f A_{\lambda}$ has a left 2-adjoint $Q$ and write $p_{A}:A \rightsquigarrow QA$ and $q_{A}:QA \to A$ for the unit and counit.  Each loose $f:A \rightsquigarrow B$ corresponds to a tight morphism $f^{\prime}:QA \to B$ such that $f^{\prime}\thing p_{A}=f$.  One triangle equation gives $q_{A}\thing p_{A}=1$; if for each $A$ this happens to be the counit of an $l$-reflection $q_{A} \dashv p_{A}$ then each $f$ is represented by an $l$-span
$$
\xy
(0,0)*+{A}="0";(20,0)*+{QA}="1";(40,0)*+{B}="2";
{\ar_{q_{A} \dashv p_{A}} "1"; "0"}; {\ar^{f^{\prime}} "1"; "2"}; 
\endxy
$$
and it can be easily shown that $j:\f A_{\tau} \to \g A \in ^{\bot}\ldoct$.  Of course these hypotheses are strong: it is not easy to check whether a given \f F-category admits loose morphism classifiers.
\section{Monadicity}
In this section we give our monadicity theorems.  We begin by extending Eilenberg-Moore comparison 2-functors to \f F-functors.  In Theorem~\ref{thm:naturality} we show these comparisons to be natural in $w$, in the sense of Diagram 2 of the Introduction.  Our main result on monadicity is Theorem~\ref{thm:monadicity}.
\subsection{Extending the Eilenberg-Moore comparison}
 \begin{Theorem}\label{thm:EMExtension}
 Let $H:\g A \to \f B$ be an \f F-functor to a 2-category whose tight part $H_{\tau}:\f A_{\tau} \to \f B$ has a left adjoint and consider the induced Eilenberg-Moore comparison 2-functor $E$ left below.
$$
\xy
(0,0)*+{\f A_{\tau}}="a0"; (15,-20)*+{\f B}="b0";(30,0)*+{\TAlgs}="c0";
{\ar_{H_{\tau}} "a0"; "b0"}; 
{\ar^{E} "a0"; "c0"}; 
{\ar^{U_{s}} "c0"; "b0"}; 
\endxy
\hspace{2cm}
\xy
(0,0)*+{\f A_{\tau}}="00";
(25,0)*+{\g A}="10";
(0,-20)*+{\TAlg_{s}}="0-2";
(25,-20)*+{\fTAlg_{w}}="1-2";
(40,-10)*+{\f B}="2-1";
{\ar^{j} "00"; "10"}; 
{\ar^{j_{w}} "0-2"; "1-2"}; 
{\ar^{H} "10"; "2-1"}; 
{\ar_{E} "00"; "0-2"}; 
{\ar@{.>}_{E_{w}} "10"; "1-2"}; 
{\ar_{U} "1-2"; "2-1"}; 
\endxy
$$
Let $w \in \{l,p,c\}$ and suppose that \g A admits $\overline{w}$-limits of loose morphisms.  Then $E:\f A_{\tau} \to \TAlgs$ admits a unique extension to an \f F-functor $E_{w}:\g A \to \fTAlg_{w}$ over \f B, as depicted on the right above.
 \end{Theorem}
 \begin{proof}
 The commutativity of the outside of the right diagram just says that $U_{s} E = H_{\tau}$.  By Corollary~\ref{thm:algdoctrinal} and Theorem~\ref{thm:orthogonality} we know that $U:\fTAlg_{w} \to \f B$ is $w$-doctrinal and that $j:\f A_{\tau} \to \g A$ is orthogonal to each $w$-doctrinal \f F-functor, in particular $U$.  Therefore there exists a unique \f F-functor $E_{w}:\g A \to \fTAlg_{w}$ satisfying the depicted equations $E_{w}j=j_{w}E$ and $UE_{w}=H$.  These respectively assert that $E_{w}$ extends $E$ and lives over the base $\f B$.
 \end{proof}
 
In order to understand the naturality in $w$ of the above Eilenberg-Moore extensions $E_{w}:\g A \to \fTAlg_{w}$ it will be convenient to briefly consider $\f F_{2}$-categories.  An $\f F_{2}$-category consists of a 2-category equipped with three kinds of morphism: tight, loose and very loose, all satisfying the expected axioms.  For instance we have the $\f F_{2}$-category of monoidal categories, strict, strong and lax monoidal functors; likewise of algebras together with strict, pseudo and lax morphisms for a 2-monad.  One presentation of an $\f F_{2}$-category is as a triple on the left below
$$
\xy
(0,0)*+{\f A_{\tau}}="00";
(20,0)*+{\f A_{\lambda}}="10";
(40,0)*+{\f A_{\phi}}="20";
{\ar^{j} "00"; "10"}; 
{\ar^{j} "10"; "20"}; 
\endxy
\hspace{2cm}
\xy
(0,0)*+{\g A_{\tau,\lambda}}="00";
(20,0)*+{\g A_{\tau,\phi}}="10";
{\ar^{j} "00"; "10"}; 
\endxy
$$
in which each inclusion is the identity on objects, faithful and locally fully faithful.  Thus an $\f F_{2}$-category has three associated $\f F$-categories $\g A_{\tau,\lambda}$, $\g A_{\lambda,\phi}$ and $\g A_{\tau, \phi}$, and is determined by the first and third of these together with the inclusion \f F-functor, right above, which views tight and loose morphisms as tight and very loose respectively. \\
We commonly encounter $\f F_{2}$-categories sitting over a base 2-category as on the  left below.  See how monoidal categories, strict, strong and lax monoidal functors sit over \Cat for instance.
$$
\xy
(0,0)*+{\f A_{\tau}}="00";
(20,0)*+{\f A_{\lambda}}="10";
(40,0)*+{\f A_{\phi}}="20";
(20,-20)*+{\f B}="1-2";
{\ar^{j} "00"; "10"}; 
{\ar^{j} "10"; "20"}; 
{\ar_{H_{\tau}} "00"; "1-2"}; 
{\ar|{H_{\lambda}} "10"; "1-2"}; 
{\ar^{H_{\phi}} "20"; "1-2"}; 
\endxy
\hspace{2cm}
\xy
(0,0)*+{\g A_{\tau,\lambda}}="00";
(25,0)*+{\g A_{\tau,\phi}}="20";
(12.5,-20)*+{\f B}="1-2";
{\ar^{j} "00"; "20"}; 
{\ar_{H_{\tau,\lambda}} "00"; "1-2"}; 
{\ar^{H_{\tau,\phi}} "20"; "1-2"}; 
\endxy$$
To give such a diagram is equally to give a commutative triangle in \fcat as on the right above -- here the \f F-functors $H_{\tau,\lambda}$ and $H_{\tau,\phi}$ agree as $H_{\tau}$ on tight parts, and have loose parts $H_{\lambda}$ and $H_{\phi}$ respectively.  
\begin{Theorem}\label{thm:naturality}
Let $w \in \{l,c\}$.  Consider an $\f F_{2}$-category over a 2-category as below
$$\xy
(0,0)*+{\g A_{\tau,\lambda}}="00";
(25,0)*+{\g A_{\tau,\phi}}="20";
(12.5,-20)*+{\f B}="1-2";
{\ar^{j} "00"; "20"}; 
{\ar_{H_{\tau,\lambda}} "00"; "1-2"}; 
{\ar^{H_{\tau,\phi}} "20"; "1-2"}; 
\endxy$$
Suppose that $H_{\tau}$ admits a left adjoint and that $\g A_{\tau,\lambda}$ and $\g A_{\tau,\phi}$ satisfy the $(p/w)$ variants of the completeness criteria of Theorem~\ref{thm:EMExtension} so that the Eilenberg-Moore comparison 2-functor $E:\f A_{\tau} \to \TAlgs$ extends uniquely to \f F-functors $E_{p}:\g A_{\tau,\lambda} \to \fTAlg_{p}$ and $E_{w}:\g A_{\tau,\phi} \to \fTAlg_{w}$ over the base \f B.
\\
When all of this holds the square
$$\xy
(0,0)*+{\g A_{\tau,\lambda}}="00"; (20,0)*+{\g A_{\tau,\phi}}="10"; 
(0,-15)*+{\fTAlg_{p}}="01"; (20,-15)*+{\fTAlg_{w}}="11";
{\ar^{j} "00";"10"};
{\ar ^{E_{w}}"10";"11"};
{\ar _{E_{p}}"00";"01"};
{\ar _{j} "01";"11"};
\endxy$$
commutes.
\end{Theorem}
\begin{proof}
Consider the diagram left below
$$\xy
(-20,0)*+{\f A_{\tau}}="-10";
(0,0)*+{\g A_{\tau,\lambda}}="00"; (20,0)*+{\g A_{\tau,\phi}}="10"; 
(0,-15)*+{\fTAlg_{p}}="01"; (20,-15)*+{\fTAlg_{w}}="11";(40,-15)*+{\f B}="21";
{\ar^{j} "-10";"00"};
{\ar^{j} "00";"10"};
{\ar ^{E_{w}}"10";"11"};
{\ar _{E_{p}}"00";"01"};
{\ar _{j} "01";"11"};
{\ar _{U_{w}} "11";"21"};
\endxy
\hspace{2cm}
\xy
(0,0)*+{\f A_{\tau}}="00"; (20,0)*+{\g A_{\tau,\lambda}}="10"; 
(0,-15)*+{\fTAlg_{w}}="01"; (20,-15)*+{\f B}="11";
{\ar^{j} "00";"10"};
{\ar ^{H_{\tau,\lambda}}"10";"11"};
{\ar _{jE}"00";"01"};
{\ar@{.>}_{}"10";"01"};
{\ar _{U_{w}} "01";"11"};
\endxy$$
Now $E_{p}$ and $E_{w}$ agree as $E:\f A_{\tau} \to \TAlgs$ on tight morphisms.  Consequently both paths of the square coincide upon precomposition with $j:\f A_{\tau} \to \g A_{\tau,\lambda}$ as the composite $jE:\f A_{\tau} \to \TAlgs \to \fTAlg_{w}$.  Because both Eilenberg-Moore extensions $E_{p}$ and $E_{w}$ lie over the base \f B we find that postcomposing both paths of the square with $U_{w}$ yields the common composite $H_{\tau,\lambda}:\g A_{\tau,\lambda} \to \f B$.  Therefore both paths of the square are diagonal fillers for the square on the right above. Now by Corollary~\ref{thm:algdoctrinal} we know that $U_{w}:\fTAlg_{w} \to \f B$ is $p$-doctrinal.  But by Theorem~\ref{thm:orthogonality} $j:\f A_{\tau} \to \g A_{\tau,\lambda}$ is orthogonal to each $p$-doctrinal \f F-functor so that both paths coincide as the unique filler.
\end{proof}

 \subsection{2-categorical monadicity}
 We now turn to monadicity.  Let us say that a 2-functor with a left adjoint is \emph{monadic} if the induced Eilenberg-Moore comparison is a 2-equivalence and \emph{strictly monadic} if the comparison is an isomorphism.
 \begin{Theorem}\label{thm:monadicity}
 Let $H:\g A \to \f B$ be an \f F-functor to a 2-category \f B.  Let $w \in \{l,p,c\}$ and suppose that
 \begin{enumerate}
 \item $H_{\tau}:\f A_{\tau} \to \f B$ is monadic.
 \item \g A admits $\overline{w}$-limits of loose morphisms.
 \item \f B admits $\overline{w}$-limits of arrows.
 \item $H$ is $w$-doctrinal (it suffices that $H$ satisfies $w$-doctrinal adjunction, is locally faithful and reflects identity 2-cells).
 \end{enumerate}
 Then the 2-equivalence $E:\f A_{\tau} \to \TAlgs$ extends uniquely to an equivalence of \f F-categories $E_{w}:\g A \to \fTAlg_{w}$ over \f B.  Moreover if $H_{\tau}$ is strictly monadic then $E_{w}:\g A \to \fTAlg_{w}$ is an isomorphism of \f F-categories.
 \end{Theorem}
\begin{proof}
As in Theorem~\ref{thm:EMExtension} we have our extension $E_{w}:\g A \to \fTAlg_{w}$, unique in filling the square
$$
\xy
(0,10)*+{\f A_{\tau}}="00"; (30,10)*+{\g A}="10";(0,-20)*+{\fTAlg_{w}}="02";(30,-20)*+{\f B}="11";(0,-5)*+{\TAlgs}="01";
{\ar_{E} "00"; "01"}; 
{\ar_{j_{w}} "01"; "02"}; 
{\ar_{U} "02"; "11"}; 
{\ar^{j} "00"; "10"}; 
{\ar^{H} "10"; "11"}; 
{\ar@{.>}^{E_{w}} "10"; "02"}; 
\endxy$$
Recall that this filler exists because $U$ is $w$-doctrinal and $j$ orthogonal to such \f F-functors.
\\
We begin by proving the theorem in the case of strict monadicity and deduce the general case from that -- so suppose that $E:\f A_{\tau} \to \TAlgs$ is an isomorphism of 2-categories.   By Proposition~\ref{prop:limits} $U:\fTAlg_{w} \to \f B$ creates $\overline{w}$-limits of loose morphisms so that $\fTAlg_{w}$ admits them.  Now Theorem ~\ref{thm:orthogonality} implies that the inclusion $j_{w}:\TAlgs \to \fTAlg_{w}$ is orthogonal to each $w$-doctrinal \f F-functor; since $E:\f A_{\tau} \to \TAlgs$ is an isomorphism $j_{w}E:\f A \to \fTAlg_{w}$ is also orthogonal to $w$-doctrinal \f F-functors.  Now $H$ is $w$-doctrinal by assumption so that the two outer paths of the square are orthogonal decompositions of a common \f F-functor.  Therefore the unique filler $E_{w}:\g A \to \TAlg_{w}$ is an isomorphism.
\\
Suppose now that $E:\f A_{\tau} \to \TAlgs$ is only an equivalence of 2-categories.  The problem now is that $E$ will no longer be orthogonal to $w$-doctrinal \f F-functors.  We will rectify this by factoring the composite 2-functor $j_{w} E:\f A_{\tau} \to \TAlgs \to \TAlg_{w}$ as the identity on objects followed by 2-fully faithful through a 2-category $\f {\overline{A}}_{\lambda}$ as in the commutative square below.
$$\xy
(0,10)*+{\f A_{\tau}}="a0"; (30,10)*+{\f {\overline{A}}_{\lambda}}="b0";(0,-10)*+{\TAlgs}="c0";(30,-10)*+{\TAlg_{w}}="d0";
{\ar^{\overline{j}} "a0"; "b0"}; 
{\ar_{E} "a0"; "c0"}; 
{\ar^{K} "b0"; "d0"}; 
{\ar^{j_{w}} "c0"; "d0"}; 
\endxy$$
Then $K$ is 2-fully faithful by construction but also essentially surjective on objects since each of $E$, $\overline{j}$ and $j_{w}$ are; as such $K$ is a 2-equivalence.  Moreover the composite $K \overline{j} = j_{w} E$ is both faithful and locally fully faithful since both $j_{w}$ and $E$ are, whilst $K$ is 2-fully faithful.  Therefore $\overline{j}$ is faithful and locally fully faithful too and since it is identity on objects by construction it is the inclusion of an \f F-category $\overline{\g A}:\f A_{\tau} \to \overline {\f A}_{\lambda}$.  It now follows that the above commutative square, whose vertical legs are 2-equivalences, exhibits $L=(E,K):\overline{\g A} \to \fTAlg_{w}$ as an equivalence of \f F-categories. \\
Consider the diagram defining the extension $E_{w}$ again, now drawn on the left below
$$\xy
(0,10)*+{\f A_{\tau}}="00"; (30,10)*+{\g A}="10";(0,-20)*+{\fTAlg_{w}}="02";(30,-20)*+{\f B}="11";(0,-5)*+{\g {\overline{A}}}="01";
{\ar_{\overline{j}} "00"; "01"}; 
{\ar_{L} "01"; "02"}; 
{\ar_{U} "02"; "11"}; 
{\ar^{j} "00"; "10"}; 
{\ar^{H} "10"; "11"}; 
{\ar@{.>}^{E_{w}} "10"; "02"}; 
{\ar@{.>}_{\overline{E}} "10"; "01"}; 
\endxy
\hspace{2cm}
\xy
(0,10)*+{\f A_{\tau}}="00"; (40,10)*+{\g A}="10";(20,-20)*+{\fTAlg_{w}}="02";(40,-20)*+{\f B}="11";(0,-20)*+{\g {\overline{A}}}="01";
{\ar_{\overline{j}} "00"; "01"}; 
{\ar_{L} "01"; "02"}; 
{\ar_{U} "02"; "11"}; 
{\ar^{j} "00"; "10"}; 
{\ar^{H} "10"; "11"}; 
{\ar@{.>}^{\overline{E}} "10"; "01"}; 
\endxy
$$
with the left leg rewritten as the composite $L \circ \overline{j}$.  Since $L$ is \f F-fully faithful (\emph{2-fully faithful on tight and loose parts}) it is orthogonal to each bijective on objects \f F-functor, in particular $j$, so that we obtain a unique diagonal filler $\overline{E}:\g A \to \overline{\g A}$ making the two leftmost triangles commute.
\\
Our goal is to show that $E_{w}$ is an equivalence of \f F-categories - but since $L$ is an equivalence this is, by 2 out of 3, equivalently to show that $\overline{E}$ is an equivalence of \f F-categories.  Now consider the square on the right.  The bottom leg is $w$-doctrinal as both of its components are, $L$ by Proposition~\ref{prop:equivalences}.  Since the \f F-category $\overline{\g A}$ is equivalent to $\fTAlg_{w}$ it has the same completeness properties, so that, using Theorem~\ref{thm:orthogonality} again, the left leg is orthogonal to each $w$-doctrinal \f F-functor.  Therefore the right commutative square consists of two orthogonal decompositions of a common \f F-functor and we conclude that $\overline{E}$ is an isomorphism.
\end{proof}
Note that although Theorem~\ref{thm:monadicity}, in each of its variants, only asks that $\g A$ admits certain limits it follows that $H$ creates those limits: for $U:\fTAlg_{w} \to \f B$ does so by Proposition~\ref{prop:limits} and $E:\g A \to \fTAlg_{w}$ is an equivalence of \f F-categories.  In applying the theorem this is often useful in that it tells us how these limits must be constructed in \g A.
 \subsection{\f F-categorical monadicity}
Our focus has been upon 2-monads but indeed the above results extend in a routine way to cover \f F-categorical monadicity too.  Let us briefly explain how this goes.  An \f F-monad \cite{Lack2011Enhanced} is a monad in the 2-category \fcat and so consists of an \f F-functor $T:\g A \to \g A$ and two \f F-natural transformations satisfying the usual equations.  2-monads are just \f F-monads on 2-categories viewed as \f F-categories.  An \f F-monad $T$ induces 2-monads $T_{\tau}$ and $T_{\lambda}$ and so we have strict $T_{\tau}$ and $T_{\lambda}$-algebra morphisms as in the two leftmost diagrams below.
$$\xy
(0,0)*+{TA}="00"; (20,0)*+{TB}="10"; 
(0,-20)*+{A}="01"; (20,-20)*+{B}="11";
{\ar^{Tf} "00";"10"};
{\ar ^{b}"10";"11"};
{\ar _{a}"00";"01"};
{\ar _{f} "01";"11"};
\endxy
\hspace{1cm}
\xy
(0,0)*+{TA}="00"; (20,0)*+{TB}="10"; 
(0,-20)*+{A}="01"; (20,-20)*+{B}="11";
{\ar@{~>}^{Tf} "00";"10"};
{\ar ^{b}"10";"11"};
{\ar _{a}"00";"01"};
{\ar@{~>}_{f} "01";"11"};
\endxy
\hspace{1cm}
\xy
(0,0)*+{TA}="00"; (20,0)*+{TB}="10"; 
(0,-20)*+{A}="01"; (20,-20)*+{B}="11";
{\ar@{~>}^{Tf} "00";"10"};
{\ar ^{b}"10";"11"};
{\ar _{a}"00";"01"};
{\ar@{~>}_{f} "01";"11"};
{\ar@{=>}^{\overline{f}}(10,-7)*+{};(10,-13)*+{}};
\endxy
$$
These are the tight and loose morphisms of the Eilenberg-Moore \f F-category which is denoted by $\fTAlg_{s}$.  We also have pseudo, lax and colax $T_{\lambda}$-morphisms -- a lax $T_{\lambda}$-morphism is drawn above.  These are the loose morphisms of the \f F-categories $\fTAlg_{w}$ whose tight morphisms are the strict $T_{\tau}$-algebra maps.  Now for each $w \in \{l,p,c\}$ we have the $\f F_{2}$-category whose tight and loose morphisms are the strict $T_{\tau}$ and $T_{\lambda}$-morphisms and whose very loose morphisms are the $w\textnormal{-}T_{\lambda}$-morphisms: this is captured by the inclusion \f F-functor $j_{w}:\fTAlg_{s} \to \fTAlg_{w}$.  We have the evident forgetful \f F-functors $U_{s}:\fTAlg_{s} \to \g A$ and $U_{w}:\fTAlg_{w} \to \g A$ commuting with $j_{w}$ over the base.  The key point for our applications is that these have the same properties as in the 2-monads case: namely $U_{s}$ creates all limits, $U_{w}$ creates $\overline{w}$-limits of loose morphisms (this follows from Theorem 5.13 of \cite{Lack2011Enhanced}) and $U_{w}$ is $w$-doctrinal.  With these facts in place we can give our monadicity theorem for \f F-monads -- we leave it to the reader to formulate the naturality of the Eilenberg-Moore extensions, which can be done using \emph{$\f F_{3}$-categories}.
\begin{Theorem}\label{thm:fmonadicity}
Consider an $\f F_{2}$-category $\f A_{\tau} \to \f A_{\lambda} \to \f A_{\phi}$ over an \f F-category \g B as below
$$\xy
(0,0)*+{\g A_{\tau,\lambda}}="00";
(25,0)*+{\g A_{\tau,\phi}}="20";
(12.5,-20)*+{\g B}="1-2";
{\ar^{j} "00"; "20"}; 
{\ar_{G} "00"; "1-2"}; 
{\ar^{H} "20"; "1-2"}; 
\endxy$$
(this means that ${G}_{\tau}=H_{\tau}$).  Suppose that $G$ has a left \f F-adjoint so that we have the Eilenberg-Moore comparison \f F-functor $E:\g A_{\tau,\lambda} \to \fTAlg_{s}$ over the base \g B.\\
Now let $w \in \{l,p,c\}$ and suppose that both $\g A_{\tau,\lambda}$ and $\g A_{\tau,\phi}$ admit $\overline{w}$-limits of loose morphisms.  Then there exists a unique \f F-functor $E_{w}:\g A_{\tau,\phi} \to \fTAlg_{w}$ extending $E$ and living over the base, as depicted by the following everywhere commutative diagram.
$$\xy
(0,0)*+{\g A_{\tau,\lambda}}="00";
(25,0)*+{\g A_{\tau,\phi}}="10";
(0,-20)*+{\fTAlg_{s}}="0-2";
(25,-20)*+{\fTAlg_{w}}="1-2";
(40,-10)*+{\g B}="2-1";
{\ar^{j} "00"; "10"}; 
{\ar^{j_{w}} "0-2"; "1-2"}; 
{\ar^{H} "10"; "2-1"}; 
{\ar_{E} "00"; "0-2"}; 
{\ar@{.>}_{E_{w}} "10"; "1-2"}; 
{\ar_{U_{w}} "1-2"; "2-1"}; 
\endxy$$
If $G$ is \f F-monadic, \g B admits $\overline{w}$-limits of loose morphisms and $H$ is $w$-doctrinal then $E_{w}$ is an equivalence of \f F-categories, and an isomorphism of \f F-categories whenever $G$ is strictly monadic.
\end{Theorem}
\begin{proof}
The outside of the diagram clearly commutes -- since $Hj=G$ and $U_{w}j_{w}=U_{s}$ this just amounts to the fact that the Eilenberg-Moore comparison $E$ satisfies $U_{s}E=G$.  Now $U_{w}:\fTAlg_{w} \to \g B$ is $w$-doctrinal so that if we can show that the inclusion $j:\g A_{\tau,\lambda} \to \g A_{\tau,\phi}$ belongs to $^{\bot}\wdoct$ then we will obtain $E_{w}$ as the unique filler.  Now have a commutative triangle of inclusions
$$\xy
(0,0)*+{\f A_{\tau}}="00";
(25,0)*+{\g A_{\tau,\lambda}}="20";
(25,-15)*+{\g A_{\tau,\phi}}="1-2";
{\ar^{j} "00"; "20"}; 
{\ar_{j} "00"; "1-2"}; 
{\ar^{j} "20"; "1-2"}; 
\endxy$$
in which the two $j$'s moving from left to right belong to $^{\bot}\wdoct$ by Theorem~\ref{thm:orthogonality}; thus by 2 out of 3 $j:\g A_{\tau,\lambda} \to \g A_{\tau,\phi} \in ^{\bot}\wdoct$.  As such we obtain the Eilenberg-Moore extension $E_{w}$ as the unique filler.  The remainder of the proof is a straightforward modification of the proof of Theorem~\ref{thm:monadicity}.
\end{proof}

\section{Examples and applications}
We now turn to examples.  We begin by completing our running example of monoidal categories.  Of course it is well known that monoidal categories, and each flavour of morphism between them, can be described using 2-monads -- see Section 5.5 of \cite{Lack2010A-2-categories} for an argument via colimit presentations -- although this has not previously been established by application of a monadicity theorem.  We then turn to more complex examples.  Our final example, in 7.3, is new and typical of the kind of result which cannot be established using techniques, such as colimit presentations, that require explicit knowledge of a 2-monad.
\subsection{Monoidal categories}
Let us focus, as usual, on the lax monoidal functors of $\fMonCat_{l}$.  From Example~\ref{thm:DoctrinalMonoidal} we know that $V:\fMonCat_{l} \to \Cat$ satisfies $l$-doctrinal adjunction.  It is clearly locally faithful and reflects identity 2-cells.  From Example~\ref{thm:LimitsMonoidal} we know that $V:\fMonCat_{l} \to \Cat$ creates colax limits of loose morphisms.  Therefore to apply Theorem~\ref{thm:monadicity} and establish monadicity it remains to verify that the 2-functor $V_{s}:\MonCat_{s} \to \Cat$ is monadic.  That $V_{s}$ strictly creates $V_{s}$-absolute coequalisers (in the enriched sense) is true by essentially the same argument given for groups in 6.8 of \cite{Mac-Lane1971Categories}; by Beck's theorem in the enriched setting \cite{Dubuc1970Kan-extensions} it follows that $V_{s}$ is strictly monadic so long as it has a left 2-adjoint.\newline{}
Proposition 3.1 of \cite{Blackwell1989Two-dimensional} asserts that a 2-functor preserving \emph{cotensors with \atwo} admits a left 2-adjoint just when its underlying functor admits a left adjoint.  Now cotensors with \atwo are, in fact, just colax limits of identity arrows.  That these are created by $V_{s}$ follows from the fact that $V:\fMonCat_{l} \to \Cat$ creates colax limits of loose morphisms.  It therefore remains to show that the underlying functor $(V_{s})_{0}$ admits a left adjoint.  Since this functor creates all limits it suffices to show that $(V_{s})_{0}$ satisfies the solution set condition.  For this it suffices to show that given a small category $A$ each functor $F:A \to C=U\overline{C}$ to a monoidal category factors as $ME:A \to B \to C$ with $B$ monoidal, $M$ strict monoidal and the cardinality of the set of morphisms $Mor(B)$ bounded by that of $Mor(A)$.  Here $B$ will be the monoidal subcategory of $C$ generated by the image of $F$ and $M$ the inclusion of this monoidal subcategory.\newline{}
For the induced 2-monad $T$ on \Cat we now conclude, by Theorem~\ref{thm:monadicity}, that the isomorphism of 2-categories $E:\MonCat_{s} \to \TAlgs$ over \Cat extends uniquely to an isomorphism of \f F-categories $E_{l}:\fMonCat_{l} \to \fTAlg_{l}$ over \Cat.  Likewise one can verify, in an entirely similar way, that $V:\fMonCat_{w} \to \Cat$ satisfies the conditions of Theorem~\ref{thm:monadicity} in the cases $w \in \{p,c\}$.  It follows that we have isomorphisms $E_{w}:\fMonCat_{w} \to \fTAlg_{w}$ over \Cat for each $w \in \{l,p,c\}$. By Theorem~\ref{thm:naturality} these isomorphisms are natural in $p \leq l$ and $p \leq c$ in the sense of Diagram 2 of the Introduction.

\subsection{Categories with structure and variants}
Of course there was nothing special about our taking monoidal categories in the preceding section.  The same arguments can be used to establish monadicity of categories with any kind of algebraically specified structure and their various flavours of morphisms: categories with chosen limits of some kind for instance, distributive categories and so forth.  All of these cases are well known to be monadic using colimit presentations of 2-monads, although it requires substantial and laborious calculation to use that theory to establish monadicity in the detailed manner above.  Such a detailed treatment using colimit presentations, one expressed in terms of isomorphisms of 2-categories or \f F-categories, will not be found in the literature. Colimit presentations are one of the two standard techniques for understanding the monadicity of weaker kinds of morphisms; the other is direct calculation with a 2-monad known to exist.  As a representative example of this technique consider a small 2-category $\f J$ and the forgetful 2-functor $U:[\f J,\Cat] \to [ob \f J,\Cat]$ which restricts presheaves to families along the inclusion $ob \f J \to \f J$.  $U$ has a left 2-adjoint $F$ given by left Kan extension and is strictly monadic by Beck's theorem; moreover the induced 2-monad $T=UF$ admits a simple pointwise description: at $X \in [ob \f J,\Cat]$ we have $TX(j) = \Sigma_{i}\f J(i,j) \times Xi$.  Using this formula (as in \cite{Blackwell1989Two-dimensional}) one directly calculates that $T$-pseudomorphisms bijectively correspond to pseudonatural transformations and so on, eventually deducing an isomorphism $Ps(\f J,\Cat) \to \TAlg_{p}$.  It is in such cases, more specifically when $T$ is not so simple, that our results have most value.  
\\
An example of this kind was given in \cite{Kelly2000On-the}.  Given a complete and cocomplete symmetric monoidal closed category \f V the authors take as base the 2-category \vcat of small \f V-categories.  For a small class of weights $\Phi$ they consider the 2-category $\phicol$ (we will write $\phicol_{s}$) of \f V-categories with \emph{chosen} $\Phi$-weighted colimits, \f V-functors preserving those colimits strictly and \f V-natural transformations.  This 2-category lives over $\vcat$ via a forgetful 2-functor $U_{s}:\phicol_{s} \to \vcat$.  One also has \f V-functors preserving colimits in the usual, up to isomorphism, sense: these are the loose morphisms of the \f F-category $\fphicol_{p}:\phicol_{s} \to \phicol_{p}$ which sits over \vcat via a forgetful \f F-functor $U:\fphicol_{p} \to \vcat$.  The authors show that $U_{s}:\phicol_{s} \to \vcat$ is strictly monadic and then, by calculating directly with the induced 2-monad $T$, show that one obtains an isomorphism of 2-categories $U:\fphicol_{p} \to \fTAlg_{p}$.  Let us show how this can be deduced from Theorem~\ref{thm:monadicity}.  Firstly observe that $U$ satisfies $p$-doctrinal adjunction: this follows from the fact that any equivalence of \f V-categories preserves colimits.  Again since $U$ is locally fully faithful it is certainly locally faithful and reflects identity 2-cells.  It is not hard to see that $\fphicol_{p}$ admits pseudo-limits of loose morphisms -- indeed this is shown in Section 5 of \cite{Kelly2000On-the}.  It then follows immediately from Theorem~\ref{thm:monadicity} that the isomorphism of 2-categories $E:\phicol_{s} \to \TAlg_{s}$ extends uniquely to an isomorphism of \f F-categories $E:\fphicol_{p} \to \fTAlg_{p}$ over \vcat.
\\
Furthermore, because left adjoints preserve colimits, $U:\fphicol_{p} \to \vcat$ has the additional property of satisfying $c$-doctrinal adjunction; since it is isomorphic to $\fTAlg_{p}$ it also admits, by Theorem 2.6 of \cite{Blackwell1989Two-dimensional}, lax limits of loose morphisms.  Therefore Theorem~\ref{thm:monadicity} ensures that the composite $\fphicol_{p} \to \fTAlg_{p} \to \fTAlg_{c}$ is also an isomorphism, thus explaining why the colax and pseudo $T$-morphisms coincide as those \f V-functors which preserve $\Phi$-colimits.  Again one easily applies Theorem~\ref{thm:monadicity} to show that $\fTAlg_{l}$ is isomorphic to the \f F-category $\fphicol_{l}$ whose loose morphisms are arbitary \f V-functors.  This final isomorphism highlights the fact that $U_{l}:\TAlg_{l} \to \vcat$ is 2-fully faithful: 2-monads with this property are called \emph{lax idempotent/Kock-Z\"oberlein} and have been carefully studied in \cite{Kelly1997On-property-like}.

\subsection{In a monoidal 2-category}
In a monoidal category \f C one can consider the category of monoids $\Mon(\f C)$ or of commutative monoids.  If the forgetful functor $U:\Mon(\f C) \to \f C$ has a left adjoint then Beck's theorem can be applied, with no further information, to show that $U$ is monadic.  For our final example we study the analogous situation in the context of a monoidal 2-category \f C, in which one can consider monoids, pseudomonoids (generalising monoidal categories), braided pseudomonoids and so on.  We consider only the simplest case of monoids because, in the absence of a suitable graphical calculus, it is difficult to encode diagrams compactly.  We have the 2-category of monoids, strict monoid morphisms and monoid transformations $\Mon(\f C)_{s}$ and a forgetful 2-functor $U_{s}:\Mon(\f C)_{s} \to \f C$.  Just as before, the enriched version of Beck's theorem \cite{Dubuc1970Kan-extensions} can be applied to show that if $U_{s}$ has a left 2-adjoint then it is monadic.  However we now also have (lax/pseudo/colax)-morphisms of monoids and we would like to understand that these too are monadic in the appropriate sense: this is the content, under completeness conditions on the base, of the present example.
\\
By a monoidal 2-category \f C we will mean a monoidal \f V-category where $\f V=\Cat$: this satisfies the same axioms as a monoidal category with the exception that the tensor product, the associator and the other data involved are now 2-functorial and 2-natural.  In working with \f C we will write as though it were strict monoidal -- this is justified in the theorem that follows -- and will use juxtaposition for the tensor product.  A monoid in \f C is just a monoid in the usual sense.  Given monoids $(X,m_{X},i_{X})$ and $(Y,m_{Y},i_{Y})$ a lax monoid map $(f,\overline{f},f_{0}):(X,m_{X},i_{X}) \to (Y,m_{Y},i_{Y})$ consists of an arrow $f:X \to Y$ and 2-cells as below
$$\xy
(0,0)*+{X^{2}}="00"; (20,0)*+{Y^{2}}="10"; (0,-12)*+{X}="01"; (20,-12)*+{Y}="11";
 {\ar^{f^{2}} "00";"10"};
{\ar_{f} "01";"11"};
{\ar_{m_{X}}"00";"01"};
{\ar^{m_{Y}}"10";"11"};
{\ar@{=>}_{\overline{f}}(10,-4)*+{};(10,-9)*+{}};
\endxy
\hspace{1cm}
\xy
(0,0)*+{I}="00"; (-10,-12)*+{X}="10";  (10,-12)*+{Y}="11";
 {\ar_{i_{X}} "00";"10"};
{\ar^{i_{Y}} "00";"11"};
{\ar_{f}"10";"11"};
{\ar@{=>}_{f_{0}}(0,-4)*+{};(0,-9)*+{}};
\endxy$$
such that the equation
$$\xy
(0,0)*+{X^{3}}="00"; (20,0)*+{XY^{2}}="10"; (40,0)*+{Y^{3}}="20"; 
(0,-12)*+{X^{2}}="01"; (20,-12)*+{XY}="11"; (40,-12)*+{Y^{2}}="21"; (0,-24)*+{X}="02"; (40,-24)*+{Y}="22";
 {\ar^{1ff} "00";"10"};
  {\ar^{f11} "10";"20"};
{\ar_{1m_{X}} "00";"01"};
{\ar^{1m_{Y}}"10";"11"};
{\ar^{1m_{Y}}"20";"21"};
{\ar@{=>}_{1\overline{f}}(10,-4)*+{};(10,-9)*+{}};
 {\ar_{1f} "01";"11"};
  {\ar_{f1} "11";"21"};
{\ar_{f} "02";"22"};
{\ar_{m_{X}}"01";"02"};
{\ar^{m_{Y}}"21";"22"};
{\ar@{=>}_{\overline{f}}(20,-16)*+{};(20,-21)*+{}};
\endxy
\hspace{0.25cm}
\xy
(0,-12)*+{=};
\endxy
\hspace{0.25cm}
\xy
(0,0)*+{X^{3}}="00"; (20,0)*+{X^{2}Y}="10"; (40,0)*+{Y^{3}}="20"; 
(0,-12)*+{X^{2}}="01"; (20,-12)*+{XY}="11"; (40,-12)*+{Y^{2}}="21"; (0,-24)*+{X}="02"; (40,-24)*+{Y}="22";
 {\ar^{11f} "00";"10"};
  {\ar^{ff1} "10";"20"};
{\ar_{m_{X}1} "00";"01"};
{\ar_{m_{X}1}"10";"11"};
{\ar^{m_{Y}1}"20";"21"};
{\ar@{=>}_{\overline{f}1}(30,-4)*+{};(30,-9)*+{}};
 {\ar_{1f} "01";"11"};
  {\ar_{f1} "11";"21"};
{\ar_{f} "02";"22"};
{\ar_{m_{X}}"01";"02"};
{\ar^{m_{Y}}"21";"22"};
{\ar@{=>}_{\overline{f}}(20,-16)*+{};(20,-21)*+{}};
\endxy$$
holds and such that both composite 2-cells
$$\xy
(0,0)*+{X}="00"; (20,0)*+{X}="10"; (40,0)*+{Y}="20"; 
(0,-12)*+{X^{2}}="01"; (20,-12)*+{XY}="11"; (40,-12)*+{Y^{2}}="21"; (0,-24)*+{X}="02"; (40,-24)*+{Y}="22";
 {\ar^{1} "00";"10"};
  {\ar^{f} "10";"20"};
{\ar_{1i_{X}} "00";"01"};
{\ar_{1i_{Y}}"10";"11"};
{\ar^{1i_{Y}}"20";"21"};
{\ar@{=>}_{1f_{0}}(10,-4)*+{};(10,-9)*+{}};
 {\ar_{1f} "01";"11"};
  {\ar_{f1} "11";"21"};
{\ar_{f} "02";"22"};
{\ar_{m_{X}}"01";"02"};
{\ar^{m_{Y}}"21";"22"};
{\ar@{=>}_{1\overline{f}}(20,-16)*+{};(20,-21)*+{}};
\endxy
\hspace{1cm}
\xy
(0,0)*+{X}="00"; (20,0)*+{Y}="10"; (40,0)*+{Y}="20"; 
(0,-12)*+{X^{2}}="01"; (20,-12)*+{XY}="11"; (40,-12)*+{Y^{2}}="21"; (0,-24)*+{X}="02"; (40,-24)*+{Y}="22";
 {\ar^{f} "00";"10"};
  {\ar^{1} "10";"20"};
{\ar_{i_{X}1} "00";"01"};
{\ar_{i_{X}1}"10";"11"};
{\ar^{i_{Y}1}"20";"21"};
{\ar@{=>}_{f_{0}1}(30,-4)*+{};(30,-9)*+{}};
 {\ar_{1f} "01";"11"};
  {\ar_{f1} "11";"21"};
{\ar_{f} "02";"22"};
{\ar_{m_{X}}"01";"02"};
{\ar^{m_{Y}}"21";"22"};
{\ar@{=>}_{1\overline{f}}(20,-16)*+{};(20,-21)*+{}};
\endxy$$
are identities.  A monoid transformation $\alpha:(f,\overline{f},f_{0}) \Rightarrow (g,\overline{g},g_{0})$ is a 2-cell $\alpha:f \Rightarrow g$ satisfying the equations
$$
\xy
(0,0)*+{X^{2}}="11";
(20,0)*+{Y^{2}}="31"; (0,-15)*+{X}="12";(20,-15)*+{Y}="32";
{\ar^{m_{Y}} "31"; "32"}; 
{\ar@/_1pc/_{g} "12"; "32"}; 
{\ar_{m_{X}} "11"; "12"}; 
{\ar@/^1pc/^{f^{2}} "11"; "31"}; 
{\ar@/_1pc/_{g^{2}} "11"; "31"}; 
{\ar@{=>}^{\overline{g}}(10,-9)*+{};(10,-15)*+{}};
{\ar@{=>}_{\alpha^{2}}(10,2)*+{};(10,-3)*+{}};
\endxy
\xy
(0,-7.5)*+{=};
\endxy
\xy
(0,0)*+{X^{2}}="11";
(20,0)*+{Y^{2}}="31"; (0,-15)*+{X}="12";(20,-15)*+{Y}="32";
{\ar^{m_{Y}} "31"; "32"}; 
{\ar@/^1pc/^{f} "12"; "32"}; 
{\ar@/_1pc/_{g} "12"; "32"}; 
{\ar_{m_{X}} "11"; "12"}; 
{\ar@/^1pc/^{f^{2}} "11"; "31"}; 
{\ar@{=>}_{\alpha}(10,-13)*+{};(10,-18)*+{}};
{\ar@{=>}_{\overline{f}}(10,0)*+{};(10,-6)*+{}};
\endxy
\hspace{0.5cm}
\xy
(0,2)*+{I}="00"; (-10,-15)*+{X}="10";  (10,-15)*+{Y}="11";
 {\ar@/_0.5pc/_{i_{X}} "00";"10"};
{\ar@/^0.5pc/^{i_{Y}} "00";"11"};
{\ar@/^1pc/^{f}"10";"11"};
{\ar@/_1pc/_{g}"10";"11"};
{\ar@{=>}_{f_{0}}(0,-2)*+{};(0,-7)*+{}};
{\ar@{=>}_{\alpha}(0,-13)*+{};(0,-18)*+{}};
\endxy
\xy
(0,-7.5)*+{=};
\endxy
\xy
(0,2)*+{I}="00"; (-10,-15)*+{X}="10";  (10,-15)*+{Y}="11";
 {\ar@/_0.5pc/_{i_{X}} "00";"10"};
{\ar@/^0.5pc/^{i_{Y}} "00";"11"};
{\ar@/_1pc/_{g}"10";"11"};
{\ar@{=>}_{g_{0}}(0,-5)*+{};(0,-11)*+{}};
\endxy$$
These are the 2-cells of the \f F-category $\fMon(\f C)_{l}:\Mon(\f C)_{s} \to \Mon(\f C)_{l}$ of monoids, strict and lax monoid morphisms which sits over \f C via a forgetful \f F-functor $U:\fMon(\f C)_{l} \to \f C$.  Likewise we have pseudo and colax monoid morphisms and forgetful \f F-functors $U:\fMon(\f C)_{w} \to \f C$ for each $w \in \{l,p,c\}$.\newline{}
For a simple statement we assume in the following result that  \f C admits \emph{pie limits} \cite{Power1991A-characterization}: these are a good class of limits containing $w$-limits of arrows for each $w$.
\begin{Theorem}\label{thm:monoids}
Let \f C be a monoidal 2-category admitting pie limits and suppose that $U_{s}:\Mon(\f C)_{s} \to \f C$ has a left 2-adjoint.  Let $T$ be the induced 2-monad on $\f C$.  Then for each $w \in \{l,p,c\}$ we have isomorphisms of \f F-categories $\fMon(\f C)_{w} \to \fTAlg_{w}$ over \f C and these are natural in $w$.
\end{Theorem}
\begin{proof}
Let us begin by showing that it suffices to suppose \f C to be strict monoidal.  A straightforward extension of the usual argument for monoidal categories shows that \f C is equivalent to a strict monoidal 2-category \f D via a strong monoidal 2-equivalence $E:\f C \to \f D$.  
$$\xy
(0,0)*+{\fMon(\f C)_{w}}="00"; (30,0)*+{\fMon(\f D)_{w}}="10"; 
(0,-10)*+{\f C}="01"; (30,-10)*+{\f D}="11";
{\ar^{E_{*}} "00";"10"};
{\ar ^{U_{D}}"10";"11"};
{\ar _{U_{C}}"00";"01"};
{\ar_{E} "01";"11"};
\endxy$$
Such an equivalence naturally lifts to a 2-equivalence $E_{*}:\Mon(\f C)_{w} \to \Mon(\f D)_{w}$ for each $w$ and so induces a commuting square of \f F-categories and \f F-functors as above with both horizontal legs equivalences of \f F-categories.  Now to apply Theorem~\ref{thm:monadicity} we must show that $U_{C}:\fMon(\f C)_{w} \to \f C$ is $w$-doctrinal and that $\fMon(\f C)_{w}$ has $\overline{w}$-limits of loose morphisms.  So suppose that $U_{D}$ and $\fMon(\f D)_{w}$ have these properties and let us deduce from these the corresponding properties for \f C.  Certainly if $U_{D}$ were $w$-doctrinal then $U_{C}$ would be too; for both horizontal legs, being equivalences, are $w$-doctrinal and such \f F-functors, being defined by lifting properties (as in Section 3.3), are closed under 2 out of 3.  Likewise any limits existing in $\fMon(\f D)_{w}$ exist in the \f F-equivalent $\fMon(\f C)_{w}$.  Therefore it suffices to suppose that \f C is strict monoidal.
\\
Now certainly $U:\fMon(\f C)_{l} \to \f C$ is locally faithful and reflects identity 2-cells.  Moreover given a strict monoid map $(f,\overline{f},f_{0}):X \to Y$ and adjunction $(\epsilon, f \dashv g,\eta) \in \f C$ taking mates gives 2-cells
$$\xy
(00,0)*+{Y^{2}}="11";(30,0)*+{X^{2}}="31"; (00,-25)*+{Y}="12";(30,-25)*+{X}="32";
(00,-15)*+{Y^{2}}="a";(30,-10)*+{X}="b";
{\ar^{m_{X}} "31"; "b"}; 
{\ar_{g} "12"; "32"}; 
{\ar_{m_{Y}} "a"; "12"}; 
{\ar^{g^{2}} "11"; "31"}; 
{\ar_{1} "11"; "a"}; 
{\ar^{f^{2}} "31"; "a"}; 
{\ar_{f} "b"; "12"}; 
{\ar^{1} "b"; "32"}; 
{\ar@{=>}_{\epsilon^{2}}(14,-3)*+{};(8,-6)*+{}};
{\ar@{=>}_{\eta}(24,-19)*+{};(18,-22)*+{}};
\endxy
\hspace{1cm}
\xy
(00,-7.5)*+{I}="a";(30,-7.5)*+{X}="b";
 (00,-22.5)*+{Y}="12";(30,-22.5)*+{X}="32";
{\ar^{i_{X}} "a"; "b"}; 
{\ar_{i_{Y}} "a"; "12"}; 
{\ar_{g} "12"; "32"}; 
{\ar_{f} "b"; "12"}; 
{\ar^{1} "b"; "32"}; 
{\ar@{=>}_{\eta}(24,-16)*+{};(18,-20)*+{}};
\endxy
$$
It is straightforward to see, by cancelling mates, that these give $g$ the structure of a lax monoid map $(g,\overline{g},g_{0})$, with respect to which $(\epsilon, f \dashv (g,\overline{g},g_{0}),\eta)$ is an adjunction in $\fMon(\f C)_{l}$.  Uniqueness of the lifted adjunction follows from Proposition~\ref{prop:ids}.1.  Therefore $U$ satisfies $l$-doctrinal adjunction.
\\
We will show that $U:\fMon(\f C)_{l} \to \f C$ creates colax limits of loose morphisms.  Consider a lax monoid map $(f,\overline{f},f_{0}):X \to Y$ and the colax limit $C$ of $f$ in \f C with colax cone as below
$$\xy
(0,0)*+{C}="00"; (-10,-12)*+{X}="10";  (10,-12)*+{Y}="11";
 {\ar_{p} "00";"10"};
{\ar^{q} "00";"11"};
{\ar_{f}"10";"11"};
{\ar@{=>}_{\lambda}(0,-4)*+{};(0,-9)*+{}};
\endxy$$
By the universal property of $C$ the composite colax cone left below induces a unique map $m_{C}:CC \to C$ such that $p\thing m_{C}=m_{X}\thing p^{2}$, $q\thing m_{C}=m_{Y}\thing q^{2}$ and such that the left equation below holds.  Likewise we obtain a unique $i_{C}:I \to C$ such that $p\thing i_{C}=i_{X}$, $q\thing i_{C}=i_{Y}$ and satisfying $\lambda \thing  i_{C}=f_{0}$ as on the right below.
$$\xy
(10,10)*+{C^{2}}="-10";
(0,0)*+{X^{2}}="00"; (20,0)*+{Y^{2}}="10"; (0,-12)*+{X}="01"; (20,-12)*+{Y}="11";
 {\ar_{p^{2}} "-10";"00"};
  {\ar^{q^{2}} "-10";"10"};
 {\ar_{f^{2}} "00";"10"};
 {\ar@{=>}_{\lambda^{2}}(10,7)*+{};(10,2)*+{}};
{\ar_{f} "01";"11"};
{\ar_{m_{X}}"00";"01"};
{\ar^{m_{Y}}"10";"11"};
{\ar@{=>}_{\overline{f}}(10,-5)*+{};(10,-10)*+{}};
\endxy
\xy
(0,0)*+{=};
\endxy
\xy
(10,10)*+{C^{2}}="-10";
(0,0)*+{X^{2}}="00"; (20,0)*+{Y^{2}}="10"; 
(10,-2)*+{C}="-1A";
(0,-12)*+{X}="01"; (20,-12)*+{Y}="11";
{\ar_{p^{2}} "-10";"00"};
{\ar^{q^{2}} "-10";"10"};
{\ar_{p} "-1A";"01"};
{\ar^{q} "-1A";"11"};
{\ar|{m_{C}} "-10";"-1A"};
 {\ar@{=>}_{\lambda}(10,-5)*+{};(10,-10)*+{}};
{\ar_{f} "01";"11"};
{\ar_{m_{X}}"00";"01"};
{\ar^{m_{Y}}"10";"11"};
\endxy
\hspace{0.5cm}
\xy
(10,10)*+{I}="-10";
(0,-12)*+{X}="01"; (20,-12)*+{Y}="11";
  {\ar@/_0.5pc/_{i_{X}} "-10";"01"};
  {\ar@/^0.5pc/^{i_{Y}} "-10";"11"};
 {\ar@{=>}_{f_{0}}(10,-1)*+{};(10,-7)*+{}};
{\ar_{f} "01";"11"};
\endxy
\xy
(0,0)*+{=};
\endxy
\xy
(10,10)*+{I}="-10";
(10,-2)*+{C}="-1A";
(0,-12)*+{X}="01"; (20,-12)*+{Y}="11";
  {\ar@/_1pc/_{i_{X}} "-10";"01"};
  {\ar@/^1pc/^{i_{Y}} "-10";"11"};
{\ar_{p} "-1A";"01"};
{\ar^{q} "-1A";"11"};
{\ar|{i_{C}} "-10";"-1A"};
 {\ar@{=>}_{\lambda}(10,-5)*+{};(10,-10)*+{}};
{\ar_{f} "01";"11"};
\endxy
$$
If we can show $(C,m_{C},i_{C})$ to be a monoid then these combined equations will assert exactly that $p$ and $q$ are strict monoid maps and $\lambda:q \Rightarrow (f,\overline{f},f_{0}).p$ a monoid transformation.  To show that $m_{C}\thing (m_{C}1)  = m_{C}\thing (1m_{C}):C^{3} \to C$ amounts to showing that both paths coincide upon postcomposition with the components $p$, $q$ and $\lambda$ of the universal colax cone.  We have that $p\thing m_{C}\thing (m_{C}1)=m_{X}\thing p^{2}\thing (m_{C}1)=m_{X}\thing (m_{X}1)\thing p^{3}=m_{X}\thing (1m_{X})\thing p^{3}=m_{X}\thing p^{2}\thing (1m_{C})=p\thing m_{C}\thing (1m_{C})$ and similarly for $q$ so that associativity of $m_{C}$ will follow if we can show that both paths coincide upon postcomposition with $\lambda$.  Consider the following series of equalities
$$
\xy
(10,30)*+{C^{3}}="00";
(10,15)*+{C^{2}}="11";
(10,0)*+{C}="12";
(0,-12)*+{X}="03"; (20,-12)*+{Y}="13";
 {\ar|{1m_{C}} "00";"11"};
  {\ar|{m_{C}} "11";"12"};
  {\ar_{p} "12";"03"};
    {\ar^{q} "12";"13"};
    {\ar_{f} "03";"13"};
     {\ar@{=>}_{\lambda}(10,-5)*+{};(10,-10)*+{}};
(25,22)*+{=};
\endxy
\xy
(12,30)*+{C^{3}}="00";
(0,15)*+{X^{3}}="01";(12,15)*+{C^{2}}="11";(24,15)*+{Y^{3}}="21";
(0,0)*+{X^{2}}="02"; (24,0)*+{Y^{2}}="12";
(0,-12)*+{X}="03"; (24,-12)*+{Y}="13";
 {\ar_{ppp} "00";"01"};
  {\ar^{qqq} "00";"21"};
 {\ar|{1m_{C}} "00";"11"};
  {\ar_{1m_{X}} "01";"02"};
 {\ar^{1m_{Y}} "21";"12"};
 {\ar_{pp} "11";"02"};
  {\ar^{qq} "11";"12"};
 {\ar_{ff} "02";"12"};
 {\ar@{=>}_{\lambda\lambda}(12,9)*+{};(12,4)*+{}};
{\ar_{f} "03";"13"};
{\ar_{m_{X}}"02";"03"};
{\ar^{m_{Y}}"12";"13"};
{\ar@{=>}_{\overline{f}}(12,-5)*+{};(12,-10)*+{}};
(34,22)*+{=};
\endxy
\xy
(26,30)*+{C^{3}}="00"; 
(13,22)*+{XC^{2}}="01";(39,22)*+{CY^{2}}="11"; 
(0,12)*+{X^{3}}="02"; (26,18)*+{C^{2}}="12"; (52,12)*+{Y^{3}}="22"; 
(13,10)*+{XC}="03"; (39,10)*+{CY}="13"; 
(0,0)*+{X^{2}}="04"; (26,0)*+{XY}="14"; (52,0)*+{Y^{2}}="24"; 
(0,-12)*+{X}="05"; (52,-12)*+{Y}="15";
 {\ar_{p11} "00";"01"};
 {\ar_{1pp} "01";"02"};
  {\ar|{p1} "12";"03"};
{\ar^{1qq} "00";"11"};
{\ar|{1q}"12";"13"};
{\ar^{q11}"11";"22"};
 {\ar|{1m_{C}} "01";"03"};
 {\ar_{1p} "03";"04"};
  {\ar^{1q} "03";"14"};
{\ar|{1m_{Y}}"11";"13"};
{\ar_{p1}"13";"14"};
{\ar^{q1}"13";"24"};
{\ar@{=>}_{1\lambda}(13,7)*+{};(13,2)*+{}};
{\ar@{=>}_{\lambda1}(39,7)*+{};(39,2)*+{}};
{\ar_{1m_{X}} "02";"04"};
{\ar|{1m_{C}}"00";"12"};
{\ar^{1m_{Y}}"22";"24"};
 {\ar_{1f} "04";"14"};
  {\ar_{f1} "14";"24"};
{\ar_{f} "05";"15"};
{\ar_{m_{X}}"04";"05"};
{\ar^{m_{Y}}"24";"15"};
{\ar@{=>}_{\overline{f}}(26,-4)*+{};(26,-9)*+{}};
\endxy
$$
$$
\xy
(0,14)*+{=};
(26,18)*+{C^{3}}="00"; 
(13,10)*+{XC^{2}}="01";(39,10)*+{CY^{2}}="11"; 
(0,0)*+{X^{3}}="02"; (26,0)*+{XY^{2}}="12"; (52,0)*+{Y^{3}}="22"; 
(13,-2)*+{XC}="03"; (39,-2)*+{CY}="13"; 
(0,-12)*+{X^{2}}="04"; (26,-12)*+{XY}="14"; (52,-12)*+{Y^{2}}="24"; 
(0,-24)*+{X}="05"; (52,-24)*+{Y}="15";
 {\ar_{p11} "00";"01"};
 {\ar_{1pp} "01";"02"};
  {\ar^{1qq} "01";"12"};
{\ar^{1qq} "00";"11"};
{\ar_{p11}"11";"12"};
{\ar^{q11}"11";"22"};
 {\ar|{1m_{C}} "01";"03"};
 {\ar_{1p} "03";"04"};
  {\ar|{1q} "03";"14"};
{\ar|{1m_{Y}}"11";"13"};
{\ar|{p1}"13";"14"};
{\ar^{q1}"13";"24"};
{\ar@{=>}_{1\lambda}(13,-5)*+{};(13,-10)*+{}};
{\ar@{=>}_{\lambda1}(39,-5)*+{};(39,-10)*+{}};
{\ar_{1m_{X}} "02";"04"};
{\ar|{1m_{Y}}"12";"14"};
{\ar^{1m_{Y}}"22";"24"};
 {\ar_{1f} "04";"14"};
  {\ar_{f1} "14";"24"};
{\ar_{f} "05";"15"};
{\ar_{m_{X}}"04";"05"};
{\ar^{m_{Y}}"24";"15"};
{\ar@{=>}_{\overline{f}}(26,-16)*+{};(26,-21)*+{}};
(60,14)*+{=};
\endxy
\xy
(32,18)*+{C^{3}}="00"; 
(20,10)*+{XC^{2}}="01";(44,10)*+{CY^{2}}="11"; 
(0,0)*+{X^{3}}="02";(16,0)*+{X^{2}Y}="12"; (32,0)*+{XY^{2}}="22"; (55,0)*+{Y^{3}}="32"; 
(0,-12)*+{X^{2}}="03"; (32,-12)*+{XY}="13"; (55,-12)*+{Y^{2}}="23"; 
(0,-24)*+{X}="04"; (55,-24)*+{Y}="14";
 {\ar_{p11} "00";"01"};
 {\ar_{1pp} "01";"02"};
  {\ar^{1qq} "01";"22"};
{\ar^{1qq} "00";"11"};
{\ar_{p11}"11";"22"};
{\ar^{q11}"11";"32"};
{\ar@{=>}_{1\lambda\lambda}(20,7)*+{};(20,2)*+{}};
{\ar@{=>}_{\lambda11}(45,7)*+{};(45,2)*+{}};
 {\ar_{11f} "02";"12"};
  {\ar_{1f1} "12";"22"};
  {\ar_{f11} "22";"32"};
{\ar_{1m_{X}} "02";"03"};
{\ar^{1m_{Y}}"22";"13"};
{\ar^{1m_{Y}}"32";"23"};
{\ar@{=>}_{1\overline{f}}(20,-4)*+{};(20,-9)*+{}};
 {\ar_{1f} "03";"13"};
  {\ar_{f1} "13";"23"};
{\ar_{f} "04";"14"};
{\ar_{m_{X}}"03";"04"};
{\ar^{m_{Y}}"23";"14"};
{\ar@{=>}_{\overline{f}}(32,-16)*+{};(32,-21)*+{}};
\endxy
$$
The first holds by definition of $m_{C}$, the second merely rewrites the tensor product $\lambda\lambda$ and the third rewrites the commuting central diamond.  The final equality of pasting composites does two things at once: on its left side it again applies the definition of $m_{C}$ postcomposed with $\lambda$; we rewrite the right hand side using $(\lambda1).(1m_{Y}) = (1m_{Y}).(\lambda11)$.  By an entirely similar diagram chase we obtain the equality
$$
\xy
(10,30)*+{C^{3}}="00";
(10,15)*+{C^{2}}="11";
(10,0)*+{C}="12";
(0,-12)*+{X}="03"; (20,-12)*+{Y}="13";
 {\ar|{m_{C}1} "00";"11"};
  {\ar|{m_{C}} "11";"12"};
  {\ar_{p} "12";"03"};
    {\ar^{q} "12";"13"};
    {\ar_{f} "03";"13"};
     {\ar@{=>}_{\lambda}(10,-5)*+{};(10,-10)*+{}};
(40,20)*+{=};
\endxy
\hspace{1cm}
\xy
(30,30)*+{C^{3}}="00"; 
(17,22)*+{X^{2}C}="01";(43,22)*+{C^{2}Y}="11"; 
(0,12)*+{X^{3}}="02";(26,12)*+{X^{2}Y}="12"; (43,12)*+{XY^{2}}="22"; (60,12)*+{Y^{3}}="32"; 
(0,0)*+{X^{2}}="03"; (26,0)*+{XY}="13"; (60,0)*+{Y^{2}}="23"; 
(0,-12)*+{X}="04"; (60,-12)*+{Y}="14";
 {\ar_{pp1} "00";"01"};
 {\ar_{11p} "01";"02"};
  {\ar^{11q} "01";"12"};
{\ar^{11q} "00";"11"};
{\ar_{pp1}"11";"12"};
{\ar^{qq1}"11";"32"};
{\ar@{=>}_{11\lambda}(17,19)*+{};(17,14)*+{}};
{\ar@{=>}_{\lambda\lambda1}(45,19)*+{};(45,14)*+{}};
 {\ar_{11f} "02";"12"};
  {\ar_{1f1} "12";"22"};
  {\ar_{f11} "22";"32"};
{\ar_{m_{X}1} "02";"03"};
{\ar_{m_{X}1}"12";"13"};
{\ar^{m_{Y}1}"32";"23"};
{\ar@{=>}_{\overline{f}1}(42,7)*+{};(42,2)*+{}};
 {\ar_{1f} "03";"13"};
  {\ar_{f1} "13";"23"};
{\ar_{f} "04";"14"};
{\ar_{m_{X}}"03";"04"};
{\ar^{m_{Y}}"23";"14"};
{\ar@{=>}_{\overline{f}}(30,-4)*+{};(30,-9)*+{}};
\endxy
$$
This final composite and the one above it each constitute $\lambda^{3}$ sat atop either side of the first equation for a lax monoid morphism: as such they agree and $m_{C}$ is associative.  Much smaller, though similar, diagram chases show that $m_{C}\thing (i_{C}1)=1$ and that $m_{C}\thing (1i_{C})=1$; thus $C$ is a monoid and the colax cone $(p,\lambda:q \Rightarrow pf,q)$ lifts (uniquely) to a colax cone in $\fMon(\f C)_{l}$.
\\
That the lifted colax cone satisfies the universal property of the colax limit is relatively straightforward and left to the reader.  That $p$ and $q$ detect strict monoid morphisms is a consequence of the fact that they jointly detect identity 2-cells in \f C.
\\
Now if $U_{s}:\Mon(\f C)_{s} \to \f C$ has a left adjoint it is automatically strictly monadic by the enriched version of Beck's monadicity theorem \cite{Dubuc1970Kan-extensions}.  Therefore, using the above, Theorem~\ref{thm:monadicity} asserts that the isomorphism of 2-categories $E:\Mon(\f C)_{s} \to \TAlgs$ over \f C extends uniquely to an isomorphism of \f F-categories $E_{l}:\fMon(\f C)_{l} \to \fTAlgl$ over \f C.  In a similar way one verifies the conditions of Theorem~\ref{thm:monadicity} when $w \in \{p,c\}$ to obtain isomorphisms of \f F-categories $E_{w}:\fMon(\f C)_{w} \to \fTAlg_{w}$ over \f C for each $w$; by Theorem~\ref{thm:naturality} these isomorphisms are natural in $w$.
\end{proof}
\subsection{A non-example}
All of the examples we have seen are of the strictly monadic variety and indeed this is the case whenever one studies structured objects over some same base 2-category.  Now Theorem~\ref{thm:monadicity} is general enough to cover ordinary monadicity -- up to equivalence of \f F-categories -- but in fact there exist situations of a weaker kind.  Here is one such case.  Let $\Cat_{f} \subset \Cat$ be a full sub 2-category of \Cat whose objects form a skeleton of the finitely presentable categories (the finitely presentable objects in \Cat) and let $[\Cat,\Cat]_{f} \subset [\Cat,\Cat]$ be the full sub 2-category consisting of those endo 2-functors preserving filtered colimits: this is the tight part of the \f F-category $\fps(\Cat,\Cat)_{f}:[\Cat,\Cat]_{f} \to \ps(\Cat,\Cat)_{f}$ whose loose morphisms are pseudonatural transformations.  Likewise we have an \f F-category $\fps(\Cat_{f},\Cat):[\Cat_{f},\Cat] \to \ps(\Cat_{f},\Cat)$ and now restriction along the inclusion $\Cat_{f} \to \Cat$ induces a forgetful \f F-functor $R:\fps(\Cat,\Cat)_{f} \to \fps(\Cat_{f},\Cat)$.  Further restricting along the inclusion $ob\Cat_{f} \to \Cat_{f}$ gives a commuting triangle
$$\xy
(0,0)*+{\fps(\Cat,\Cat)_{f}}="00";
(35,0)*+{\fps(\Cat_{f},\Cat)}="20";
(17.5,-20)*+{[ob \Cat_{f},\Cat]}="1-2";
{\ar^{R} "00"; "20"}; 
{\ar_{SR} "00"; "1-2"}; 
{\ar^{S} "20"; "1-2"}; 
\endxy$$
The composite $S_{\tau}R_{\tau}:[\Cat,\Cat]_{f} \to [ob\Cat_{f},\Cat]$ is monadic though not strictly so: for the induced 2-monad $T$ we have $\fTAlg_{p}$ isomorphic to $\fps(\Cat_{f},\Cat)$ with $R_{\tau}:[\Cat,\Cat]_{f} \to [\Cat_{f},\Cat]$ the Eilenberg-Moore comparison 2-functor -- that this is a 2-equivalence follows from \Cat's being locally finitely presentable as a 2-category.  Whilst $R_{\tau}$ is a 2-equivalence the 2-functor $R_{\lambda}:\ps(\Cat,\Cat)_{f} \to \ps(\Cat_{f},\Cat)$ is not: indeed $\ps(\Cat_{f},\Cat)$ is locally small whereas $\ps(\Cat,\Cat)_{f}$ is not.  Yet $R_{\lambda}$ turns out to be a biequivalence and $R$ the uniquely induced \f F-functor to the \f F-category of algebras.
\bibliographystyle{acm}
\bibliography{bibdata}
\end{document}